\documentclass[a4paper,12pt]{article}
\usepackage[utf8]{inputenc}
\usepackage{xcolor}

\usepackage[T1]{fontenc}
\usepackage{amsfonts}
\usepackage{amsmath}
\usepackage{amssymb}
\usepackage{amsthm}
\usepackage{hyperref}
\usepackage{color}
\usepackage{flowchart}
\usetikzlibrary{arrows}

\newcommand{\R}{{\mathbb R}}

\newcommand{\N}{{\mathbb N}}

\newcommand{\E}{{\mathbb E}}
\newcommand{\PP}{{\mathbb P}}

\newcommand{\grad}{{\mathrm  {grad}\ }}

\renewcommand{\[}{\begin{eqnarray*}}
\renewcommand{\]}{\end{eqnarray*}}

\newtheorem{definition}{Definition}[section]
\newtheorem{proposition}[definition]{Proposition}
\newtheorem{lemma}[definition]{Lemma}

\theoremstyle{plain}
\newtheorem{theorem}{Theorem}[section]
\newtheorem{corollary}[definition]{Corollary}

\begin{document}

\title{ 
Local Limit Theorems and 
Strong Approximations for Robbins-Monro Procedures}
\author{Valentin Konakov   \footnote{National Research University Higher School of Economics (HSE), Russian Federation, vkonakov@hse.ru, research of Valentin Konakov for the article was prepared within the framework of the Basic Research Program at HSE University.} \and Enno Mammen  \footnote{Institute for  Mathematics, Heidelberg University, Germany, mammen@math.uni-heidelberg.de} \and  Lorick~Huang\footnote{INSA - Toulouse, Toulouse, France, {lhuang@insa-toulouse.fr} }}
\date{\today}

\maketitle

\begin{abstract}
The Robbins-Monro algorithm is a recursive, simulation-based stochastic procedure to approximate the zeros of a function that can be written as an expectation.
It is known that under some technical assumptions, Gaussian limit  theorems approximate the stochastic performance of  the algorithm.
Here, we are interested in  strong approximations for Robbins-Monro procedures. The main tool for getting them are local limit theorems, that is, studying the convergence of the density of the algorithm.
The analysis relies on a version of parametrix techniques for Markov chains converging to diffusions. The main difficulty that arises here is the fact that the drift is unbounded.
\end{abstract}

\section{Introduction}

This paper is devoted to strong approximations for Robbins-Monro procedures. The approximations are based on the study of a local limit theorem for  Robbins-Monro procedures.
These algorithms have first been introduced in \cite{robb:monr:51} to approximate the solution of an equation $h(\theta)=0$, where randomly disturbed values of $h(\theta)$ are  observed at updated points $\theta$. Since then, extensive literature have been published on the subject, but to the best of our knowledge, a local limit theorem has never been obtained. 
 We refer to the monographs   \cite{neve:khas:73}, \cite{benv:meti:prio:90}, \cite{duf:97}, and  \cite{kush:yin:03} for a general mathematical discussion of these algorithms and a review of the literature. An important class of Robbins-Monro procedures are optimisation methods based on stochastic gradient decent. There is an increasing literature of their applications in the implementation of artificial neural networks and in reinforcement learning. We refer to
\cite{der:jen:kas:23, jen:rie:22,  mey:22, mou:bac:11} and the references therein for some recent developments and overviews. For applications in statistics see also \cite{bor:che:21, ChenLee, Dieule, LiLou} 

The main idea of this paper was to study what can be obtained for the theory of Robbins-Monro procedures by using the parametrix method. The parametrix method is an approach for getting series expansions for the differences of transition densities of SDE's with variable and with constant coefficients, for a more detailed discussion see Section \ref{subsec:parametrix}. We will apply parametrix expansions to compare transition densities of Robbins-Monro procedures and their diffusion limits. These bounds can be used to get total variation bounds on the multivariate densities of the Robbins-Monro procedures and the limiting diffusion processes, evaluated on an increasing grid of points. In particular, this allows strong approximations of Robbins-Monro procedures by the limiting diffusion processes. For a further discussion of applications see Section \ref{subsec:main}.

We fix a probability space $(\Omega ,\mathcal{F} ,\PP)$ on which all random variables we consider below are defined.
Let $(\gamma_k)_{k\ge0}$ be a decreasing step sequence that will be specified later, and let $(\eta_k)_{k\ge 0}$ be a collection of independent and identically distributed random variables.
We define the following recursive procedure:
\begin{equation}\label{robbins:monro:lineaire}
\theta _{n+1}=\theta _{n}-\gamma _{n+1}H(\theta _{n},\eta_{n+1}),
\ \theta_0 \in \R^d,
\end{equation} 
where $H$ is a function from $\R^d\times \mathcal X$ to $\R^d$ with $\mathcal X$ equal to the support of $\eta_i$. Without loss of generality we can assume that $\mathcal X$ is a subset of $\R$.
Generally, the Robbins-Monro procedure is used to approximate the zeros of the function:
$h(\theta) = \E[H(\theta,\eta)]$, where $\eta$ has the same distribution as  $\eta_k$.

Even though the general theory extends to the case of multiple zeros, in this paper, we assume that $h$ has only one zero, $\theta^*$ (i.e. $h(\theta)=0$ iff $\theta =\theta^*$ ). We assume that the sequence $(\gamma_k)_{k\ge 1}$ is chosen as \begin{equation} \label{eq:gammdef}
\gamma_k=  \frac A {k^{ \beta}+ B}\end{equation} with  constants $A>0$, $B \geq 0$ and $1/2< \beta \leq 1$. For this choice  we get that
\begin{equation} \label{eq:gammconv}
\sum_{k \ge 1} \gamma_k = +\infty, \ \ \sum_{k\ge 1} \gamma_k^2 < +\infty,
\end{equation}
which is usually assumed for the step sequence $(\gamma_k)_{k\ge 1}$. Our theory can be generalized to other monotonically decreasing choices of the step sequence $\gamma_k$ as long as we have that \eqref{eq:gammconv} holds,  and that 
\begin{equation} \label{eq:sqgammconv}\frac {\sqrt {\gamma_k} -\sqrt {\gamma_{k+1}}} {(\gamma_k)^{3/2}}\to   \bar \alpha
\end{equation}
 for some constant $ \bar \alpha$. Note that for the choice \eqref{eq:gammdef} we have that $\bar \alpha =0$, if  $1/2< \beta < 1$ and $\bar \alpha =(2 A (B+1))^{-1}$, if  $ \beta = 1$.

Under appropriate assumptions, it can be shown that the convergence:
\begin{equation} \label{eq:thetaconv}
\theta_n \underset{n \rightarrow +\infty}{\longrightarrow} \theta^*, 
\end{equation}
holds almost surely, see \cite{benv:meti:prio:90}. Furthermore, Gaussian limit theorems have been proved. For a formulation of such a result we remark first that after a renormalisation the procedure \eqref{robbins:monro:lineaire} stabilizes around the solution of the following Ordinary Differential Equation (ODE):
\begin{equation}\label{EDO_h}
{\frac{d}{dt} \bar{\theta}^N_t = - h(\bar{\theta}^N_t) \text{ with initial value }  \bar{\theta}^N_0.}
\end{equation}
{Note that  $\bar \theta^N _{t}$ depends on $N$ because we allow that its initial value $\bar \theta^N _{0}$ depends on $N$.} 
Fluctuations of the Robbins-Monro algorithm should be considered with respect to the solution $(\bar{\theta}^N_t)_{t \ge0}$ of the ODE \eqref{EDO_h}.
For the defintion of the renormalisation, we consider a \textit{shift} in the indexation of the procedure that will allow us to consider $(\theta_n)_{n \ge 0}$ in a region that is close to stationarity.
Let $N\in \N$, and consider a sequence $\left( \theta _{n}^{N}\right) _{n\geq 0}=\left( \theta _{N+n}\right) _{n\geq 0}$,  of shifted Robbins-Monro algorithms.
These algorithms satisfy the following recurrence equation:
\begin{equation} \label{14}
\theta _{n+1}^{N}=\theta _{n}^{N}-\gamma _{n+1}^{N}H(\theta_{n}^{N},\eta_{n+1}^{N}) \\ 
\end{equation}
with starting value $\theta _{0}^{N} \in \R^d$, where $\eta_{n+1}^{N}=\eta_{N+n+1}$, and $\gamma _{n+1}^{N} = \gamma _{N+n+1}$. 
Set now:
\[
t_{0}^{N}=0, \ t_{1}^{N}=\gamma _{1}^{N}, \ t_{2}^{N}=\gamma
_{1}^{N}+\gamma _{2}^{N},\ \dots, \  t_k^N = \gamma_1^N + \cdots + \gamma_k^N
\]
and define for an arbitrary fixed terminal time $T>0$:
$$
M(N) = \inf\{ k \in \N \ ; \ t_k^N\ge T \}.
$$
Closeness of $\theta _{n}^{N}$ and $\bar \theta^N _{t}$ on the grid $t_{0}^{N},\ \dots, \  t_{M(N)}^N$ becomes intuitively clear because $\bar \theta^N _{t}$ is for $t=t_{n+1}^N$ close to its Euler approximation 
$${ \check \theta^N _{n+1} =\check \theta^N _{n} - \gamma_{n+1} ^N h(\check \theta^N _{n} ) \text{ with } \check \theta^N _{0} = \bar \theta^N _{0}.}$$
Note that the Robbins -Monro 
procedure can be rewritten as perturbed Euler scheme
$$ \theta^N _{n+1} =  \theta^N _{n} - \gamma_{n+1} ^N h( \theta^N _{n} )+ \varepsilon_n^N,$$
where $$\varepsilon_n^N = \theta^N _{n+1} -  \theta^N _{n}+ \gamma_{n+1} ^N h(  \theta^N _{n} ) = 
- \gamma_{n+1} ^N (H( \theta^N _{n}, \eta^N _{n+1}) - h( \theta^N _{n})).$$
The centered innovations $\varepsilon_n^N $ may be considered as "small fluctuations".

On the interval $[0,T]$ we consider  the renormalized process $U_t^N$ that is equal to :
\begin{equation}\label{proc:renor}
U_t^N= \frac{\theta_k^N - \bar{\theta}^N_{t_k^N}}{\sqrt{\gamma_k^N}}
\end{equation}
as long as $ t \in [ t_k^N , t_{k+1}^N )$.
Under our assumptions, stated in Subsection \ref{subsec:assumptions},  it can be shown that the convergence \eqref{eq:thetaconv} holds and that  the sequence of processes $\left \{ (U_t^N)_{0 \leq t \leq T}, N \geq 1\right \}$ converges weakly  to  the   solution $(X_t)_{0 \leq t \leq T}$ of the $d$-dimensional SDE:
\begin{equation}\label{EDS_limit}
\mathrm dX^i_t = \bar \alpha X^i_t \mathrm dt- \sum_{j=1} ^d \frac {\partial h_i} {\partial x_j} (\bar \theta^N_t) \cdot X^j_t \mathrm dt+  \sum_{j=1} ^d R_{ij} ^{1/2} (\bar \theta^N_t) \mathrm dW^j_t, \ i=1,...,d,
\end{equation} with $X_0 = U_0^N$
or in matrix notation 
\begin{equation}\label{EDS_limit2}
\mathrm dX_t = \left(\bar \alpha I -\mathcal Dh (\bar{\theta}^N_t)  \right)X_t \mathrm dt + R^{1/2}(\bar \theta^N_t)\mathrm dW_t,
\end{equation}
where $W$ is a $d$-dimensional Brownian motion, where for $\theta \in \R^d$ the matrix $R( \theta)$ is the covariance of $H(\theta, \eta)$, and where  
 we write  $\mathcal D h (x)= ({\grad} h_1,...,\grad h_d)^\intercal (x)= \left (\frac {\partial h_i} {\partial x_j} (x) \right) _{1 \leq i,j \leq d}$ for the $d\times d$ valued derivative of $h$ at the point $x \in \R^d$. { To simplify notation, dependence of $X$ on $N$ is not denoted in notation.} This convergence can be shown by application of results discussed in \cite{benv:meti:prio:90} or in \cite{kush:yin:03}. For a motivation of drift and diffusion factor we remark that \eqref{eq:sqgammconv} holds and $\gamma_{k}^N \to 0$ and $\gamma_{k+1}^N / \gamma_k^N \to 1$ for $N\to \infty$. 
 
 Few results on the existence of densities are available for diffusions with unbounded drift as in \eqref{EDS_limit2}. To obtain upper bounds by application of the parametrix method we need to control terms in the parametrix series, and in case of unbounded drifts this becomes a delicate problem. For drifts with sublinear growth the generalization of the parametrix method was obtained in \cite{DeckKruse}, but the method developed there fails for drifts with linear growth as in \eqref{EDS_limit2}. It seems quite plausible that a linearly growing drift is exactly the borderline case, starting from which it is necessary to introduce a forward flow corresponding to the transport of the initial condition or, equivalently, a backward flow corresponding to the transport of the terminal condition \cite{dela:meno:10, kona:meno:molc:10, menpeszha21}. The presence of forward or backward flows complicates the proof of convergence of the parametrix series. Convergence of the complete parametrix series was established after the method of majorizing diffusions appeared, which was first applied in \cite{dela:meno:10}, and then developed in \cite{bitt:kona:21, Pesce}. Diffusions with dynamics  \eqref{EDS_limit2} and unbounded drift appeared in many applied problems and, as follows from the results of this work, Robbins-Monro (RM) stochastic approximation models are one of such applications. The weak convergence  of suitable Markov chains associated with the RM procedure to Gaussian diffusions is well known. As for the convergence of transition densities and their rate of convergence, as far as we know, this work contains the first results of this kind.
 
 Our main results are stated in the next section. 
Throughout the paper $C$ denotes a positive constant that is chosen large enough. The value of the constant $C$ may vary from line to line.

\section{Main Results}

\subsection{Assumptions and outline of the paper} \label{subsec:assumptions} Throughout the paper we will make the following assumptions. 

\begin{itemize}
\item[(A1)] The innovations $\eta_1, \eta_2, ... $ are i.i.d. with some distribution $\mu$ and the step $\gamma_k$  is given by   \eqref{eq:gammdef}.

\item[(A2)] For any compact subset $Q$ of $\R^d$ there exist constants $C_Q$ and $q_Q$, possibly depending on $Q$ such that for all $\theta \in Q$
$$|H(\theta, x)| \leq C_Q ( 1 + |x|^{q_Q} ).$$

\item[(A3)]  { The function $h(\theta)=\int H(\theta, x ) \mu( \mathrm d x)$ has a unique stationary point $\theta^*$, $h(\theta^*)=0$.  The solutions $\bar \theta^N_t$ of \eqref{EDO_h} with initial values $\bar \theta^N_0$ converge to $\theta^*$, uniformly for $t \in [0,T]$. The function $h(\theta)$ is bounded and has two bounded derivatives and the function $H(\theta, x)$ is Lipschitz w.r.t{.}\ its first argument  with a constant not depending on  $x$ for $\theta$ in a neighborhood of $\theta^*$. }
\end{itemize}
{  A natural choice of  $\bar \theta^N_0$ is $\bar \theta_{NT}$, where $\bar \theta_{t}$ is a solution of 
$\frac{d}{dt} \bar{\theta}_t = - h(\bar{\theta}_t)$ with  some initial value $\bar{\theta}_0 \in \R$. In such setting, convergence and stability of a solution of the differential equation is a classical topic of mathematics. For two more recent contributions see e.g.  
\cite {Cara2001} and \cite{damak2021}.
The first paper contains results on stability based on the construction of Lyapunov functions. The latter  paper relates stability of nonlinear equations to linearized versions. In particular, one gets that $\bar{\theta}^N_t \to  \theta^*$ for $t \to \infty$ if all eigen values of $\mathcal D h(\theta^*)$ have strictly negative real parts,  see also Section 4.5.3 in Part II of \cite{benv:meti:prio:90}.}

For the next assumption we need some notation that will be used again when we define the truncated processes. We define { with $a_N = \ln( 1/\gamma_1^N)$
\begin{eqnarray*}
\chi_{N}(x) &=& \left\{\begin{array}{cc}x & \text{for } \|x\| \leq a_N, \\ a_N \frac x {\|x\|} \phi_{N}(\|x\|) & \text{for } \|x\| > a_N. \end{array}\right.\\
\phi_{N}(\|x\|) &=& \left\{\begin{array}{cl}k_N \int_{a_N}^{3 a_N- \|x\|} \exp\left( -\frac 1 {(t-a_N)(2a_N-t)}\right ) \mathrm d t& \text{for } a_N < \|x\| \leq 2a_N, \\
1 & \text{for } \|x\| \leq a_N, \\ 0  & \text{for } \|x\| >2a_N , \end{array}\right. \\
\alpha_t^N &=&\alpha_{t_k^N} ^N =  \frac{\sqrt{\gamma_k^N}-\sqrt{\gamma^N_{k+1}}}{(\gamma^N_{k+1})^{3/2}} \text { for } t_k^N \leq t < t_{k+1}^N,
\end{eqnarray*}
 where the value $k_N$ depends on $N$ and it is equal to $(\int_{a_N}^{2a_N } \exp (-(t-a_N)^{-1} (2a_N-t)^{-1} ) \mathrm d t)^{-1}= (\int_{0}^{a_N } \exp (-v^{-1} (a_N-v)^{-1} ) \mathrm d v)^{-1}$.} Thus $\phi_{N}(u)$ is continuous at $u = a_N$. We will also make use of the fact that $\phi_{N}$  is infinitely differentiable.
\begin{itemize}
\item[(A4)] { For $\xi(\theta,v) = H(\theta,v)- h(\theta)$  with $\theta \in \R^d, v \in \R$ and for $y $ in a neighborhood of 0, the centered random variables 
$\xi  ( \theta^* +  y , \eta_{k}  )$
 have densities   $f_{y}(z)$ that are  five times continuously differentiable with derivatives that are at least of polynomial decay of order $M> 2d +6$, i.e.
 for $x \in \mathbb R^d$, for $t \in [0,T]$ and  
for   multi-indices $\nu$ , $|\nu| \leq 5$ it holds with a constant $C>0$ that
\begin{eqnarray} \label{polybound}| D_z^{\nu} f_{x}(z) | \leq C Q_M(z),\end{eqnarray}
where, for  all $r > d$ the function $Q_r:\mathbb R ^d \to \mathbb R$ is defined by 
$$Q_r(z) = c_r (1 + \|z\|)^{-r},$$ with $c_r$  chosen such that 
$\int_{\mathbb R ^d} Q_r(z) \mathrm d z =1$.  Furthermore, it holds for $x,y,z \in \mathbb R^d$, $1 \leq i,j \leq d$  that 
\begin{eqnarray}
\nonumber&&\int _{\mathbb R^d} f_{y}(z) z_i \mathrm d z=0 , \\ \nonumber&&\int _{\mathbb R^d} f_{y}(z) z_i z_j\mathrm d z=R_{ij} (\theta^* + y ) , \\ 
\label{bounddeltadensity}&&| f_{x} (z)- f_{y} (z)| \leq C  \|x-y\| Q_M(z)
\end{eqnarray}
for some constant $C>0$. Here, $R_{ij} (\theta )$ are the elements of   the covariance matrix $R(\theta)$ of $H(\theta, \eta) $ which is assumed to exist.} \end{itemize}
{ We will use the notation
$$f^N_{t,x}(z) = f_x(\bar \theta_t^N - \theta^*+z)$$
for the density of 
$\xi  ( \bar \theta^N_t +  x , \eta_{k}  )$.}

\begin{itemize}
\item[(A5)]  The covariance matrix $R(\theta)$ of $H(\theta, \eta) $ has a smallest eigenvalue bounded away from 0, uniformly for $\theta$. The elements $R_{ij}$ of the covariance matrix are absolutely bounded and are Lipschitz continuous with a uniformly valid  Lipschitz constant in a tubular neighborhood of $\bar \theta^N_t$ for all $t \in \Gamma_N$ for $N$ large enough.
\end{itemize}

We shortly discuss our assumptions. Assumptions (A1)--(A3) and (A5) are similar to the assumptions used e.g. in 
 \cite{benv:meti:prio:90} for the proof of functional central limit theorems. Here, the assumptions are slightly weaker and simpler because we consider the case of i.i.d. innovations $(\eta_n)_{n\geq 1}$ whereas in  \cite{benv:meti:prio:90} the innovations are Markov chains. In the additional condition (A4) we assume that the innovations have densities which allow for derivatives  up to order 5. This is   a technical assumption that is needed for Taylor expansions of the density which are used in the parametrix method. The polynomial bound in \eqref{polybound} is used when we study convolutions of the innovation densities, see also \cite{kon:koz:men:17}. In particular, the convolutions of the upper bounds can be easily handled, see e.g. \eqref{eq:lem3cent5}, which helps in the proofs.

 In Subsection \ref{subsec:truncRobMon} we introduce a truncated version $V_t^N$ of the Robbins-Monro procedure   and show that it approximates the untruncated version $U_t^N$. More precisely, we will show that the supremum of the absolute difference of the two processes is of order $O_P(\sqrt{\gamma_1^N}) = O_P(N^{-\beta/2})$. In the following subsection \ref{subsec:main} we will state our main result. We will show that the truncated  Robbins-Monro procedure can be approximated by the diffusion  $X_t$ in the following sense. We will give bounds on the total variation distance and Hellinger distance between the transition densities of the two processes. Furthermore, we will consider the joint distributions on an increasing grid of time points for the two processes. We will state bounds on the total variation distance and Hellinger distance for the distribution of the processes on the grid. This result allows a strong approximation of the  truncated  Robbins-Monro procedure by the diffusion  $X_t$. We will give bounds on the total variation distance and Hellinger distance between the values of the truncated  Robbins-Monro procedure and the diffusion on an  increasing grid of time points. For the proof of these results 
we define a truncated modification of the diffusion  in Subsection \ref{subsec:compdiff}. For the comparison of the truncated Robbins-Monro process to the truncated diffusion we will make use of the parametrix method. How this approach can be adapted to our setting will be explained in Subsection \ref{subsec:parametrix}. Subsection \ref{subsec:compRobMon} states our result on the comparison of the truncated processes and outlines its proof.  All proofs of our results will be given in Section \ref{sec:proofs}.

In all stated lemmas, propositions and theorems of the paper we make the assumptions (A1) -- (A5). { Constants that depend on the parameters introduced in the assumptions are denoted by $C$. The value of $C$ may be different on different locations. }

\subsection{Approximation of Robbins-Monro algorithm by a truncated modification} \label{subsec:truncRobMon}
In this section we  introduce a truncated modification of $(U_t^N)_{t \ge 0}$ for which we will show uniform convergence to the untruncated version. For a discussion of practical and theoretical aspects of truncated Robbins-Monro procedures we refer to \cite{kush:yin:03}. For a motivation how we truncate $U^N_{t_{k}^N}$ we rewrite the process $U^N_{t_{k}^N}$ as specified in the following lemma.
\begin{lemma}\label{DYN:MKV:CHAIN} With 
\begin{eqnarray*}
\beta_{k+1} ^N &=&  \sqrt {\gamma^N_{k+1}}\left ( -h (\bar{\theta}^N_{t^N_k} ) - \frac {\bar{\theta}^N_{t^N_{k+1}} -\bar{\theta}^N_{t^N_k} } { \gamma^N_{k+1}} \right ) \end{eqnarray*}  the Markov chain $(U_{t_k^N}^N)$ has the following representation:
\begin{eqnarray}\label{DYN_GEN}
U^N_{t^N_{k+1}} &=& U_{t^N_k}^N + G_N(t_k^N,U_{t^N_k}^N) \gamma^N_{k+1} U_{t^N_k}^N - \sqrt{\gamma^N_{k+1}} \xi \left ( \bar{\theta}^N_{t_k^N} +  \sqrt{\gamma^N_{k}} U_{t^N_k}^N , \eta_{k+1}^N  \right ) + \beta_{k+1} ^N,
\end{eqnarray}
where
\begin{eqnarray*}
G_N(t_k^N,x)&=& \alpha_{t^N_k}^N I - \sqrt { \frac {\gamma^N_{k}}{\gamma^N_{k+1}}} \int_0^1\mathcal D h \left
(\bar{\theta}^N_{t^N_k} + \delta x \sqrt{\gamma^N_{k}} \right )   \mathrm d\delta . \end{eqnarray*}
\end{lemma}

The proof of the lemma will be given in Subsection \ref {sec:DYN:MKV:CHAIN}.

The representation \eqref{DYN_GEN} motivates the following truncated process $V^N_{t_{k}}$:
\begin{eqnarray*}V^N_{t^N_{k+1}} &=& V_{t^N_k}^N + F_N(t_k^N,V_{t^N_k}^N) \gamma^N_{k+1} \chi_{N} (V_{t^N_k}^N) - \sqrt{\gamma^N_{k+1}} \xi \left ( \bar{\theta}^N_{t_k^N} +  \sqrt{\gamma^N_{k}} \chi_{N} (V_{t^N_k}^N) , \eta_{k+1}^N  \right ) ,
\end{eqnarray*}
where {\begin{eqnarray} \label{eq:defFN}
F_N(t,x)&=&  \overline{\alpha}\ I - \ \int_{0}^{1}{\mathcal{D}h({\overline{\theta}}_{t}\  + \delta}\chi_{N}(x)\sqrt{\gamma_{1}^{N}}\ ) 
   \mathrm d\delta , \end{eqnarray}}
 and where $I$ is the $d \times d$ identity matrix. The term $a_N$ and the functions $\chi_{N}$ and  $\phi_{N}$ were defined after the statement of Assumption (A3).
Note that in the definition of the truncated process the term $\beta_{k+1} ^N$ is omitted. This is done because this term is asymptotically negligible, as stated in the following lemma.
\begin{lemma}\label{BETA}
It holds that:
$$
\sup_{k \geq 1} \|\beta_{k+1}^N \| \underset{N \rightarrow +\infty}{\longrightarrow} 0 .
$$
\end{lemma}
For a proof of the lemma see Subsection \ref{sec:BETA}. 
Before we will come in the next subsection to the statement of our main result we will compare  the truncated and the untruncated Robbins-Monro algorithm in the following proposition. {
\begin{proposition}\label{TRUNC}
For $C>0$ large enough it holds for $\frac{1}{2} < \beta < \frac{2}{3}$ and for $\beta=1$ that:
$$
\mathbb P \left (\sup_{1 \leq k \leq M(N)} \left \| U_{t_k^N}^N - V_{t_k^N}^N \right \| > C N^{- \beta/2}\right ) \underset{N \rightarrow +\infty}{\longrightarrow} 0 .
$$
Furthermore, for $\frac{2}{3} \leq \beta < 1$ it holds for $C>0$ large enough
$$
\mathbb P \left (\sup_{1 \leq k \leq M(N)} \left \| U_{t_k^N}^N - V_{t_k^N}^N \right \| > C N^{- (1 - \beta)}\ln N\right ) \underset{N \rightarrow +\infty}{\longrightarrow} 0 .
$$
\end{proposition}}
For a proof of the proposition see Subsection \ref{sec:TRUNC}.
\subsection{Main result} \label{subsec:main}

The following theorem compares the truncated Robbins-Monro procedure and the diffusion process. It states a bound for the difference in total variation norm and Hellinger norm for the  transition densities $p_N$ of the truncated Robbins-Monro procedure $V^N_t$  and the transition density $q$ of the diffusion process $X_t$. For $x,z \in \R^d$ and  $0 \leq s < t $ we denote the conditional density of $X_t$ at $z$ given $X_s=x$ by $q(s,t,x,z)$ and for $x,z \in \R^d$ and $s,t \in \{t_{0}^{N},\ \dots, \  t_{M(N)}^N\}$ with $s< t$ we write $p_N(s,t,x,z)$ for the conditional density of $V_{t}^N$ at $z$ given $V_{s}^N=x$.

\begin{theorem}\label{theo:main} { There exists $N_0 >0$ such that for $N \geq N_0$}, for $s,t \in \{t_{0}^{N},\ \dots, \  t_{M(N)}^N\}$ with $s< t$ and $x \in \R^d$ with $|x| \leq a_N/2$  and $a_N = \ln( 1/\gamma_1^N)$ it holds that \begin{eqnarray*}
&& \int_{\R^d} |p_N - q|(s,t,x,z) \mathrm d z    \\
&& \qquad \leq C   
\left\{\begin{array}{cc}(\ln N)^2 N^{-\beta/2} & \text{if } 1/2 < \beta \leq 2/3, \\  N^{-(1-\beta)} &\text{if } 2/3 < \beta < 1,\\ (\ln N)^2 N^{-1/2} & \text{if }  \beta =1.\end{array}\right.
 \end{eqnarray*}
\end{theorem} 

The theorem follows directly by application of Proposition \ref{prop:compdiff} in Subsection \ref{subsec:compdiff}  and Theorem \ref{prop:comptrunc} in Subsection \ref{subsec:compRobMon}. Note that because of $\int_{\R^d}( \sqrt{p_N} -\sqrt {q}) ^2 (s,t,x,z) \mathrm d z  \leq    \int_{\R^d} |p_N - q|(s,t,x,z) \mathrm d z$ we get also a bound for the Hellinger norm.

{We shortly discuss the rates in the upper bound of Theorem \ref{theo:main}. We do this for the choice $s=0$ and $t=T$. For our definition \eqref{eq:gammdef} of $\gamma_k$ one can show that $M(N)$ is of order $N^\beta$. Thus we have that $V_T^N$ is a statistic of i.i.d. variables $\eta_1,...,\eta_m$ with $m$ of order $O(N^\beta)$. This motivates that the rate of convergence is not faster than of order $O(N^{-\beta/2})$, a rate, which we would expect for statistics in the domain of central limit theorems. Up to a logarithmic factor we get this rate for $1/2 < \beta \leq 2/3$ and for $\beta=1$. Thus one would expect that, up to the logarithmic factor, the rate of the theorem is optimal for this values of $\beta$. If we will compare the transition densities of the truncated and untruncated diffusion processes in Theorem \ref{prop:comptrunc}  we will have an additional term of order  $O(N^{-(1-\beta)})$ for $1/2 < \beta < 1$. Now, for $ 2/3 < \beta < 1$ this term dominates the $O(N^{-\beta/2})$ term. We do not know if the rate $O(N^{-(1-\beta)})$ is sharp for the error caused by our choice of truncation or if one could get a better result than the bound of Theorem \ref{prop:comptrunc} by choosing another truncation. At least one could expect this because of the jump in the rate at $\beta=1$. But anyway, at this stage we cannot say anything about the optimality of the order of the bound of Theorem \ref{theo:main} for $ 2/3 < \beta < 1$.}

These  results can also be used  for getting a result on the distributions of the truncated Robbins-Monro procedure and the diffusion process on an increasing grid of time points $\tau^N_0=0 < \tau^N_1 < ...< \tau^N_{m_N}$ with $ \tau^N_1, ..., \tau^N_{m_N} \in \{ t_{0}^{N},\ \dots, \  t_{M(N)-1}^N\}$ and $\tau^N_{m_N} =t_{M(N)-1}^N$, { where $m_N$ is some sequence of natural numbers.} For $m_N \geq 1$, $z_1,...,z_{m_N},x \in \R^d$ put $z=(z_1,...,z_{m_N})$ and denote the conditional distribution of $(X_{\tau^N_j}:1 \leq j \leq m_N)$ given $X_0 = x$ and of $(V^N_{\tau^N_j}:1 \leq j \leq m_N)$ given $V^N_{0} = x$ by $Q_{x}^{m_N}$ or $P_{N,x}^{m_N}$, respectively.  We get the following corollary of Proposition \ref{prop:compdiffcorr} and Proposition \ref{prop:comptrunccorr} for the L$_1$-distance between these measures. 
\begin{corollary}\label{cor:main} { Suppose that for some sequence $m_N \to \infty$ and grid points $\tau^N_0=0 < \tau^N_1 < ...< \tau^N_{m_N}$ we have that 
\begin{eqnarray} \label{eq:tau}  C^{-1}  m_N^{-1} \leq  {|\tau^N_{j} - \tau^N_{j-1}|} \leq C m_N^{-1} \text{ for }1 \leq j \leq m_N, \quad  N^{-\beta}  m_N   (\ln N)^{4} \to 0.\end{eqnarray}}
With a measure  $\nu$ that dominates $Q_{x}^{m_N}$ and $P_{N,x}^{m_N}$ it holds for $x \in \R^d$ with $|x| \leq a_N/2$ that 
\begin{eqnarray*}
&&\int \left | \frac { \mathrm d Q_{x}^{m_N}}{\mathrm d \nu } 
-  \frac {  \mathrm d P_{N,x}^{m_N}}{ \mathrm d \nu }
\right | \mathrm d \nu \leq C \delta_N , \end{eqnarray*}
where 
\begin{eqnarray*}
&& \delta _N =  m_N^{1/4} \left (I_{\{\frac 1 2 <\beta<1\}} N^{-(1-\beta)/2} +  \sqrt{\ln(N)}N^{-\beta/4}\right )\\
&& \qquad \qquad + m_N (\ln N)^2 N^{-\beta/2}. 
\end{eqnarray*}\end{corollary}
In particular, we have that the upper bound $ \delta _N$ in the corollary converges to 0 if $m_N$ is of the form $m_N = C N^\mu$ with $\mu <  \beta/2$ for $\frac 1 2 <\beta\leq \frac 4 5 $, $\mu < 2(1- \beta)$ for $\frac 4 5 \leq \beta<1$ and $\mu < 1/2$ for $\beta=1$. 

The corollary can be used to get a strong approximation result. It is well known that for two probability densities $f$ and $g$ with respect to a measure $\nu$ one can construct random variables $X$ and $Y$ with densities $f$ or $g$, respectively, such that $\mathbb P\left ( X=Y\right ) = \int \text{min}(f,g) \mathrm{d}\nu$. Thus
Corollary \ref{cor:main} implies that on a large enough probability space we can construct versions of 
$V^N_t$ and $X_t$ \begin{equation} \label{strongappgrid} \mathbb P\left ( V^N_{\tau^N_j} = X_{\tau^N_j}: 0 \leq j \leq m_N\right ) \geq 1 - C \delta_N,\end{equation} 
{ see also formula (6.11) in \cite{Vill2006}.}
{ For illustration let us consider the case where $\gamma_k = k^{-1}$, i.e. we choose $A=1, B=0$ and $\beta=1$ in \eqref{eq:gammdef}. 
Choose $m_N$ and $N$ such that $m_N ( \ln N)^2 (N) ^{-1/2} \to 0$. {We now show that Corollary \ref{cor:main} can be used to get strong approximations for $(\theta_{k_0^N}= \theta_N, \theta_{k_1^N},...,\theta_{k_{m_N-1}^N},$ $\theta_{k_{m_N}^N} = \theta_{2 N})$, where $k_0^N, k_1^N, ..., k_{m_N-1}^N,k_{m_N}^N $ is an approximately equidistant grid of $[N,2N]$ with $k_0^N =N$ and $k_{m_N}^N=2N$, more precisely we choose $k_j^N$ such that $|k_j^N- N - (j/m_N) N| \leq 1/2$ for $j=1,...,m_N$. We now apply our theory with  $\tau_j^N $ equal to partial sums of harmonic series $ \frac 1 {N+1} + ... + \frac 1 {k_j^N}$ which is approximately equal to $ \ln( \frac  {k_j^N} N)$. One can show that 
the partition $\tau_j^N: j=0,...,m_N$ of the interval $[0,\ln 2]$ fulfills the condition \eqref{eq:tau}. With the help of \eqref{strongappgrid} we get that for each $N$ there exists a diffusion $X_t: t \in [0,\ln 2]$  such that 
$$\mathbb P \left ( V_{\tau_j^N}^N= X_{\tau_j^N} \text{ for } j \in \{N, k_1^N, ..., k_{m_N-1}^N, 2N\} \right ) \to 1.$$ Now Proposition \ref{TRUNC} gives with some constant $C$
$$\mathbb P \left ( \|U_{\tau_j^N}^N- X_{\tau_j^N}\|\leq C N^{-1/2}  \text{ for } j \in \{N, k_1^N, ..., k_{m_N-1}^N, 2N\} \right ) \to 1.$$ 

Using the definition of $U_t^N$ this implies that
$$\mathbb P \left ( \left \| \theta _{t_j^N} - \left [ \bar \theta^N _{\tau_j^N} + (N+j)^{-1/2} X_{\tau_j^N}\right] \right \|\leq C N^{-1}  \text{ for } j \in \{N, k_1^N, ..., k_{m_N-1}^N, 2N\} \right ) \to 1.$$}
Thus we get a strong approximation of the Robbins-Monro sequence on a grid with an increasing number of elements. 
This strong approximation result has some applications when termination or uncertainty quantifications of a Robbins-Monro procedure is chosen depending on the development of the past values of the procedure.  }
Other possible applications of Corollary \ref{cor:main} are proofs of  local invariance principles, that is, of the convergence in total variation for a wide class of stochastic functionals of the Robbins-Monro procedure to the functionals of a limiting diffusion process. This application will be discussed in another publication. The proof of such a result would   be based on the results of this paper and the stratification method developed in the papers of Y.\ Davydov, see \cite{dav:80a, dav:83, dav:80b, dav:81}. Furthermore, our results allow to get results that hold uniformly for applications to unbounded test functions as e.g. moments. Finally,  one could use the results the other way around: how well can diffusions be approximated by Robbins-Monro series.

\subsection{Comparison of truncated and untruncated diffusion}\label{subsec:compdiff}  
In this section we will compare $X_t$ with a smoothly truncated modification $X_t^N$ defined by
$$\mathrm d X_t^N = F_N(t, X_t^N) \chi_{N}(X_t^N) \mathrm d t + R^{1/2} ( \bar \theta^N_t) \mathrm d B_t$$
with $F_N$ defined in \eqref{eq:defFN}. The  conditional density of $X^N_t$ at $z$ given $X^N_0=x$ is denoted by $q_N(0,t,x,z)$.

The following proposition states a bound for the difference of the two densities in total variation norm.

\begin{proposition}\label{prop:compdiff} { There exists $N_0 >0$ such that for $N \geq N_0$}, for $0 \leq t \leq T$ and $x \in \R^d$ with $|x| \leq a_N/2$ it holds that \begin{eqnarray*}
&&\int_{\R^d}( \sqrt{q_N} -\sqrt {q}) ^2 (s,t,x,z) \mathrm d z  \leq    \int_{\R^d} |q_N - q|(s,t,x,z) \mathrm d z    \\
&& \qquad \leq C  (t-s)^{1/2} \left( I_{\{\frac 1 2 <\beta<1\}} (\gamma_1^N)^{\beta^{-1}-1} +  \ln(1/\gamma_1^N) (\gamma_1^N)^{1/2}\right ) \\
&& \qquad \leq C  (t-s)^{1/2} \left( I_{\{\frac 1 2 <\beta<1\}} N^{-(1-\beta)} +  (\ln N) N^{-\beta/2}\right ). \end{eqnarray*}
\end{proposition} 

{ Here in the last inequality we used that  $C^{-1} N ^{-\beta} \leq \gamma_1^N \leq C N ^{-\beta}$ which directly follows from \eqref{eq:gammdef}. The proof of the other inequalities }of the proposition will be given in Section \ref{sec:compdiff}.  The proposition can be used for getting a result on the distributions of the diffusions on an increasing grid of time points. With $m_N \geq 1$, $z_1,...,z_{m_N},x \in \R^d$, $z$, $\tau^N_j$ and $Q_{x}^{m_N}$ defined as in the last subsection, denote the conditional distribution of $(X_{N,t_j}:1 \leq j \leq m_N)$ given $X_{N,0} = x$ by $Q_{N,x}^{m_N}$.  We get the following corollary of Proposition \ref{prop:compdiff} for the Hellinger distance and L$_1$-distance between the measures $Q_{x}^{m_N}$ and $Q_{N,x}^{m_N}$. The proof of this result can also be found in Section \ref{sec:compdiff}. 

\begin{proposition}\label{prop:compdiffcorr} Suppose that \eqref{eq:tau} holds. With a measure  $\nu$ that dominates $Q_{x}^{m_N}$ and $Q_{N,x}^{m_N}$ it holds { for $N \geq 1$}, for $x \in \R^d$ with $|x| \leq a_N/2$ that 
\begin{eqnarray*}
&&\frac 1 2 \int \left | \frac { \mathrm d Q_{x}^{m_N}}{\mathrm d \nu } 
-  \frac {  \mathrm d Q_{N,x}^{m_N}}{ \mathrm d \nu }
\right | \mathrm d \nu \leq  H (Q_{x}^{m_N},Q_{N,x}^{m_N})   \\
&& \qquad \leq C  m_N^{1/4} \left (I_{\{\frac 1 2 <\beta<1\}} (\gamma_1^N)^{(\beta^{-1}-1)/2} +  \sqrt{\ln(1/\gamma_1^N)} (\gamma_1^N)^{1/4}\right ) \\
&& \qquad \leq C  m_N^{1/4} \left (I_{\{\frac 1 2 <\beta<1\}} N^{-(1-\beta)/2} +  \sqrt{\ln(N)}N^{-\beta/4}\right ). \end{eqnarray*}\end{proposition}
In particular, we have that the upper bound in the proposition converges to 0 if $m_N$ is of the form $m_N = C N^\mu$ with $\mu <  \beta$ for $\frac 1 2 <\beta\leq \frac 2 3$, $\mu < 2(1- \beta)$ for $\frac 2 3 \leq \beta<1$ and $\mu < 1$ for $\beta=1$.

\subsection{The parametrix method} \label{subsec:parametrix}  

The main tool of our proofs is the parametrix method. This method allows to represent transition densities of certain processes by so-called parametrix series. A parametrix for a differential operator is often easier to construct than a fundamental solution and for many purposes it is almost as powerful. Sometimes it is also possible to construct a fundamental solution from a parametrix by iteratively improving it. The idea of parametrix representations is old and goes back to \cite{lev:07}, who considered non-degenerate second-order operators in non-divergent form, see \cite{Fri64}. It is based on perturbative theory methods. In real time, the density of an SDE with variable coefficients is a priori close to the density of an SDE with constant coefficients, for which we have good density controls. The idea of the method is to use Kolmogorov equations satisfied by the two densities for precise estimates of their difference. In addition to Levi's approach, other versions of the parametrix method have been developed. \cite{MS67} proposed an approach
to obtain asymptotic expansions of the Laplacian spectrum on a manifold as a function of its curvature. This approach allows to study errors of discrete approximation schemes. In the framework of inhomogeneous non-degenerate diffusion processes it has been used in \cite{KM00} to obtain local limit propositions for approximating Markov processes and in \cite{KM02} to get error bounds for Euler schemes. In recent years there were some progresses of the parametrix method in the literature. Without claiming to be complete, we only mention extensions to processes with jumps, to McKean--Vlasov type equations, see \cite{FKM21}, to  degenerate Kolmogorov type diffusions, see \cite{dela:meno:10,kona:meno:molc:10} and to parabolic SPDEs, see \cite {PP21}. We now explain the core of the method for the example of a classical Brownian diffusion without going into technical details. We just explain the "Master formula" as it appeared in the celebrated  article \cite {MS67}. In this example we are interested in studying Brownian SDEs of the form
\begin{equation} \label{parintro1} Z_t = z + \int_0^t b(s,Z_s) \mathrm d s + \int_0^t \sigma(s,Z_s) \mathrm d W_s, \end{equation} 
where $(W_s)_{s \geq 0}$ is an $\R^k$-valued Brownian motion on some filtered probability space $(\Omega. \mathcal F, ( \mathcal F_t)_{t \geq 0}, \mathbb P)$ and where the process $Z_t$ 
is $\R^k$-valued. The coefficients $b$ and $\sigma$ are $\R^k$-valued or $\R^k\times \R^k$-valued, respectively and under certain assumptions on $b$ and $\sigma$ a unique weak solution of \eqref{parintro1} exists and admits a transition density/fundamental solution $p(s,t,x,y)$. Along with equation  \eqref{parintro1}, consider the equation with coefficients "frozen" at the point $y$   and put $\tilde p(s,t,x,y) = p^y(s,t,x,y)$ where $p^z(s,t,x,y)$ is the Gaussian  transition density of
\begin{equation} \label{parintro2} \tilde Z_v = \tilde Z_0 + \int_0^v b(u,z) \mathrm d u + \int_0^v \sigma(u,z) \mathrm d W_u. \end{equation} 
Below we will make use of the backward and forward Kolmogorov equations:
\begin{equation} \label{parintro3}\frac {\partial \tilde p}{\partial s} + \tilde L \tilde p = 0, \quad \frac {\partial  p}{\partial s} +  L  p = 0, \quad - \frac {\partial \tilde p}{\partial t} + \tilde L^* \tilde p = 0 , \quad - \frac {\partial  p}{\partial t} +  L ^*  p = 0.\end{equation} 
together with the initial conditions
\begin{equation} \label{parintro4} \tilde p (t,t,x,y) = \delta(x-y)  \text{  and   }    p (t,t,x,y) = \delta(x-y) . \end{equation} 
With the help of  \eqref{parintro3} and  \eqref{parintro4} we can write the basic equality for the parametrix method:
\begin{eqnarray*}&&p (s,t,x,y) - \tilde p (s,t,x,y)\\
&& \qquad = \int_s^t \mathrm d u \frac {\partial}{\partial u} \left [ \int_{\R^k}   p (s,u,x,z)  \tilde p (u,t,z,y) \mathrm d z\right ]\\
&& \qquad = \int_s^t \mathrm d u \int_{\R^k}   \left [  \tilde p (u,t,z,y) L^* p (s,u,x,z)-   p (s,u,x,z) \tilde L \tilde p (u,t,z,y)\right ]  \mathrm d z\\
&& \qquad = \int_s^t \mathrm d u \int_{\R^k}   \left [   p (s,u,x,z) (L- \tilde L) \tilde  p (u,t,z,y)\right ]  \mathrm d z.
\end{eqnarray*}
This equation can be written as 
\begin{equation} \label{parintro5} p = \tilde p (t,t,x,y) + p \otimes H,\end{equation} 
where $H= [L - \tilde L] \tilde p$ 
and
the convolution type binary operation $\otimes$ is defined by
\begin{equation} \label{parintro5a} (f \otimes g) (s,t,x,y) = \int_s^t \mathrm d u \int _{\mathbb R^d} f(s,u,x,z) g(u,t,z,y) \mathrm d z.\end{equation} 
Iterative application of \eqref{parintro5}  gives an infinite series
\begin{equation} \label{parintro6} p =  \sum_{r=0} ^\infty \tilde p \otimes H^{(r)},\end{equation} 
where $\tilde p \otimes H^{(0)}= \tilde p$ and $\tilde p \otimes H^{(r+1)} = (\tilde p \otimes H^{(r)}) \otimes H$ 
for $r=0,1,2,...$ . An important property of representation \eqref{parintro6} is that it allows us to express the non-Gaussian density $p$ in terms of Gaussian densities $\tilde p$. Equation \eqref{parintro6} is the "Master formula" in our proof. We will apply it twice, to the truncated diffusion $X_t^N$ and to the truncated Robbins-Monro process $U_t^N$. In the latter application the parametrix method allows us to express transition densities of the Robbins-Monro process as an expression depending on the densities of sums of independent random variables. The main idea of the proof is based on the comparison of the densities of sums with the Gaussian densities showing up in the parametrix expansion of the truncated diffusion $X_t^N$.

We will make use of infinite series expansions of  $q_N$
\begin{eqnarray} \label{eq:paraqN}
q_N(t,s,x,y) &=& \sum_{r=0}^\infty \tilde q_N\otimes H_N^{(r)}(t,s,x,y),
\end{eqnarray}
where the notation on the right hand side of this equation will be explained now. The  equation is based on looking at a frozen diffusion process $\tilde X_t^{s,x,y}$ 
$$\tilde X_t^{s,x,y} = x + \int_s^t F_N(u, \theta_{u,s}^N(y)) \chi_{N} (\theta_{u,s}^N(y)\mathrm d u + \int_s^t R^{1/2}(\bar \theta^N_u) \mathrm d B_u,$$
where 
the functions $\theta_{t,s}^N$ ($0 \leq t \leq s$) are the solutions of the following ordinary differential equations
\begin{eqnarray} \label{eq:defthetaN}
\frac {\mathrm d}{\mathrm d t} \theta_{t,s}^N(y) = F_N(t,\theta_{t,s}^N(y)) \chi_{N} (\theta_{t,s}^N(y))
\end{eqnarray}
with terminal condition
$\theta_{s,s}^N(y)=y$.
This is a Gaussian process with transition density
$$\tilde q_N(t,s,x,y) =\tilde q_N^{t,x,y}(t,s,x,z)|_{z=y} = g_{\bar \sigma}(t,s, \theta_{t,s}^N(y) -x),$$
where \begin{eqnarray*} 
g_{\bar \sigma}(t,s,z) &=& \frac 1 {(2 \pi)^{d/2} \sqrt{\det \bar \sigma(t,s)} }\exp \left ( -\frac 1 2 z^\intercal \bar \sigma(t,s)^{-1} z \right ),\\
\bar \sigma(t,s)&=& \int_t^s R(\bar \theta^N_u) \mathrm d u.
\end{eqnarray*}
We define
$$H_N(t,s,x,y) = ( L_t^N -\tilde L_t^N)\tilde q_N(t,s,x,y), $$where $ L_t^N$ and $\tilde L_t^N$ are  the following generators:
\begin{eqnarray*} 
L_t^N&=& \frac 1 2 \sum_{i,j=1}^d R_{ij}(\bar \theta^N_t) \frac {\partial^2}{ \partial x_i \partial x_j} + \sum_{i=1}^d \left ( \sum_{j=1}^d\left [F_N(t, x )\right]_{i,j} \left [\chi_{N}( x )\right]_{j}\right ) \frac {\partial}{ \partial x_i },\\
\tilde L_t^N&=& \frac 1 2 \sum_{i,j=1}^d R_{ij}(\bar \theta^N_t) \frac {\partial^2}{ \partial x_i \partial x_j} + \sum_{i=1}^d \left ( \sum_{j=1}^d\left [F_N(t, \theta_{t,s}^N(y) )\right]_{i,j} \left [\chi_{N}( \theta_{t,s}^N(y) )\right]_{j}\right ) \frac {\partial}{ \partial x_i },
\end{eqnarray*}
where the flow $\theta_{t,s}^N(y)$ is defined in \eqref{eq:defthetaN}.
The convolution type operation $\otimes$ is defined as in \eqref{parintro5a} and for $r=1,2,...$ the $r$-fold convolution is given by $g\otimes H_N^{(r)}= (g\otimes H_N^{(r-1)})\otimes H_N$ with $g\otimes H_N^{(0)}=g$. The validity of formula  \eqref{parintro6} for our choices of $p, \tilde p$, and  $H$ and the correctness  of \eqref{eq:paraqN}    for $q_N,\tilde q_N$, and $H_N$ follows from Lemmas \ref{em2:boundstheta} and \ref{lem4:boundsH} stated in Subsection \ref{sec:somebounds}.

For the proof in the next subsection we will make use of  the  following series expansion of $p_N$ for $l < k$
\begin{eqnarray} \label{eq:parapN}
p_N(t_l^N, t_k^N,x,y) &=& \sum_{r=0}^N \tilde p_N\otimes_N \mathcal K_N^{(r)}(t_l^N, t_k^N,x,y),
\end{eqnarray}
where the notation on the right hand side of this equation will be explained now. For this purpose we consider the following frozen Markov chain 
\begin{eqnarray} \label{eq:frozenMC}
\tilde V_{t_{i+1}^ N}^{N,y}  &=&\tilde V_{t_{i}^ N}^{N,y} + \left ( \theta_{t_{i+1}^ N, t_{j}^ N}^{N} (y) -  \theta_{t_{i}^ N, t_{j}^ N}^{N} (y)\right )\\
&& \quad - \sqrt{\gamma_{i+1}^N} \xi\left ( \bar \theta^N_{t_i^N}+ \chi_{N} \left ( \theta_{t_{i}^ N, t_{j}^ N}^{N} (y) \right )\sqrt{\gamma_{i}^N}, \eta_{i+1}\right) \nonumber
\\ 
&=&\tilde V_{t_{i}^ N}^{N,y} + \int_{t_{i}^ N}^{t_{i+1}^ N} F_N\left (u,
 \theta_{u, t_{j}^ N}^{N} (y)\right ) \chi_{N} \left (  \theta_{u, t_j^N}^N (y)\right )\mathrm d u\nonumber \\
&& \quad - \sqrt{\gamma_{i+1}^N} \xi\left ( \bar \theta^N_{t_i^N}+ \chi_{N} \left ( \theta_{t_{i}^ N, t_{j}^ N}^{N} (y) \right )\sqrt{\gamma_{i}^N}, \eta_{i+1}\right). \nonumber\end{eqnarray}
By iterative application of \eqref{eq:frozenMC} we get that for $k < j$
$$\tilde V_{t_{j}^ N}^{N,y}  =\tilde V_{t_k^ N}^{N,y} + y - \theta_{t_{k}^ N, t_{j}^ N}^{N} (y) - \sum_{i=k} ^{j-1} \sqrt {\gamma_{i+1}^N} \xi_i,$$
where $$\xi_i =  \xi\left ( \bar \theta^N_{t_i^N}+ \chi_{N} \left ( \theta_{t_{i}^ N, t_{j}^ N}^{N} (y) \right )\sqrt{\gamma_{i}^N}, \eta_{i+1}\right).$$
For the transition density $  \tilde p_N^y( t_{k}^ N, t_{j}^ N, x,z) $ of the frozen Markov chain we have 
\begin{eqnarray*}   \tilde p_N^y( t_{k}^ N, t_{j}^ N, x,z)&=& \frac {\mathrm d}{\mathrm d z} P\left ( \tilde V_{t_j^ N}^{N,y} \in \mathrm d z \big | \tilde V_{t_k^ N}^{N,y}=x\right )\\
&=& p_{S_N}\left (z-x-y +  \theta_{t_{k}^ N, t_{j}^ N}^{N} (y)\right ),
\end{eqnarray*} 
where $p_{S_N}$ denotes the density of $-\sum_{i=k} ^{j-1} \sqrt {\gamma_{i+1}^N} \xi_i$. Note that 
$$ \tilde p_N^y( t_{k}^ N, t_{j}^ N, x,y)= p_{S_N}\left (  \theta_{t_{k}^ N, t_{j}^ N}^{N} (y)- x\right ).$$
For a test function $\phi$ we define now the one step generators:
\begin{eqnarray*}  \mathcal L_N \phi( t_{i}^ N, t_{j}^ N, x,y)&=& \frac1 {\gamma_{i+1}^N} \int_{\mathbb R^d} \left (\phi( t_{i+1}^ N, t_{j}^ N, z,y) - \phi( t_{i+1}^ N, t_{j}^ N, x,y)\right )p_N( t_{i}^ N, t_{i+1}^ N, x,z) \mathrm d z,\\
\tilde {\mathcal L}_N \phi( t_{i}^ N, t_{j}^ N, x,y)&=& \frac1 {\gamma_{i+1}^N} \int_{\mathbb R^d} \left (\phi( t_{i+1}^ N, t_{j}^ N, z,y) - \phi( t_{i+1}^ N, t_{j}^ N, x,y)\right ) \tilde p^y_N( t_{i}^ N, t_{i+1}^ N, x,z) \mathrm d z.
\end{eqnarray*} 
We put $$\mathcal K_N( t_{i}^ N, t_{j}^ N, x,y) = \left ( {\mathcal L}_N - \tilde {\mathcal L}_N\right ) \tilde p^y_N( t_{i}^ N, t_{j}^ N, x,y)$$ and with the discretized time convolution $$(f \otimes_N g) ( t_{i}^ N, t_{j}^ N, x,y) = \sum_{k=i} ^{j-1} \gamma_{k+1} ^N \int_{\mathbb R^d} f ( t_{i}^ N, t_{k}^ N, x,z) g( t_{k}^ N, t_{j}^ N, z,y) \mathrm d z$$
we define for $r=1,2,...$ the $r$-fold convolution as $g\otimes_N \mathcal K_N^{(r)}= (g\otimes_N \mathcal K_N^{(r-1)})\otimes_N \mathcal K_N$ with $g\otimes_N \mathcal K_N^{(0)}=g$.
With this notation one can show that \eqref{eq:parapN} holds. For the proof one makes repeated use of the Markov property, see also Lemma 3.6 in Konakov and Mammen (2000). 

\subsection{Comparison of the truncated version of the Robbins-Monro algorithm with the truncated  diffusion}\label{subsec:compRobMon}
In this subsection we want to prove the following bound for the difference $|q_N -p_N|$ of the transition density $p_N$ of the truncated Robbins-Monro procedure  $V_{t_i^N}^N$ and of the transition density $q_N$ of the truncated diffusion $X_t^N$. The main result of this section is the following result
\begin{theorem}\label{prop:comptrunc} { There exists  a constant $N_0 >0$ such that for $N \geq N_0$, for $i < j$ and  for all $x,y \in \mathbb R^n$ }it holds that
\begin{eqnarray*}&&|q_N(t_{i}^ N, t_{j}^ N,x,y) - p_N(t_{i}^ N, t_{j}^ N,x,y)| \\ && \qquad \leq C \sqrt{\gamma_1^N} \ln^2(1 / \gamma_1^N) \mathcal Q_{M-d-6}(t_{j}^ N-  t_{i}^ N,y-\theta_{t_{j}^ N,  t_{i}^ N}^N(x)).\end{eqnarray*}
\end{theorem} 
In the statement of the theorem for a natural number  $m$ and positive real numbers $t$ we define 
$$ \mathcal Q_{m}(t,x) = t^{-d/2} Q_m( t^{-1/2} x)$$ 
with $Q_m$ defined in Assumption (A4).

We now come to the proof of Theorem \ref{prop:comptrunc}. Note that by \eqref{eq:paraqN} and \eqref{eq:parapN} we have for  $i < j$ that 
$$q_N(t_{i}^ N,t_{j}^ N,x,y) - p_N(t_{i}^ N,t_{j}^ N,x,y) =  \sum_{r=0}^\infty \tilde q_N\otimes H_N^{(r)}(t_{i}^ N,t_{j}^ N,x,y) -  \sum_{r=0}^N \tilde p_N\otimes_N \mathcal K_N^{(r)}(t_{i}^ N,t_{j}^ N,x,y). $$ With this expansion Theorem \ref{prop:comptrunc} follows immediately from the following lemmas. In the first lemma we replace  the convolution operation $\otimes$ in the parametrix expansion of $q_N$ by the discrete convolution  operator $\otimes_N$. In Lemma \ref{lem2:comptrunc} we show that it suffices to consider only the first $N$ terms in the expansion. Lemma \ref{lem3:comptrunc} now is the heart of our argument. We replace  in the parametrix expansions the Gaussian densities $\tilde q_N$ by the densities $\tilde p_N$ of normed sums of independent random variables. 
We use Edgeworth expansion arguments and local limit propositions that offer powerful tools to bound the errors of this replacement. Note that $\tilde q_N$ is replaced by  $\tilde p_N$ at two places: in the summation and in the definition of the kernels $H_N$ and $K_N$. At this point we apply Lemma \ref{lem:boundsec31 2}. The kernel $K_N$ is defined as $H_N$, but with $\tilde q_N$ replaced by  $\tilde p_N$. Finally, Lemma \ref{lem4:comptrunc} bounds errors that show up by replacing the kernel $K_N$ by the kernel $\mathcal K_N$, which is used in the parametrix expansion of the Robbins-Monro algorithm. 

\begin{lemma} \label{lem1:comptrunc}
For $i < j$ it holds with some constant $C > 0$ that
 \begin{eqnarray*} &&\left |\sum_{r=0}^\infty \tilde q_N\otimes H_N^{(r)}(t_{i}^ N, t_{j}^ N,x,y) - \sum_{r=0}^\infty \tilde q_N\otimes_N H_N^{(r)}(t_{i}^ N, t_{j}^ N,x,y) \right |\\
 && \qquad  \leq C \ln^2(1 / \gamma_1^N) \sqrt {\gamma_1^N} \sqrt { t_{j}^ N- t_{i}^ N} \bar q_N(t_{i}^ N, t_{j}^ N,x,y),\end{eqnarray*} 
 where $\bar q_N$ is the transition density of the diffusion $\bar X_t^N$ defined in \eqref{eq:majdiff2}.
\end{lemma}
Note that we have the bound \eqref{eq:boundmajdiff2} for $\bar q_N$.
\begin{lemma} \label{lem2:comptrunc}
{ There exists  a constant $N_0 >0$ such that for $N \geq N_0$,} for $i < j$ it holds  that
 \begin{eqnarray*} &&\left |\sum_{r=N+1}^\infty \tilde q_N\otimes_N H_N^{(r)}(t_{i}^ N, t_{j}^ N,x,y) \right |\\
 && \qquad  \leq C \exp(-CN) \bar q_N(t_{i}^ N, t_{j}^ N,x,y).\end{eqnarray*} 
\end{lemma}
\begin{lemma} \label{lem3:comptrunc}
{ There exists  a constant $N_0 >0$ such that for $N \geq N_0$,} for $i < j$ it holds  that
 \begin{eqnarray*} &&\left |\sum_{r=0}^N \tilde q_N\otimes_N H_N^{(r)}(t_{i}^ N, t_{j}^ N,x,y)- \sum_{r=0}^N \tilde p_N\otimes_N K_N^{(r)}(t_{i}^ N, t_{j}^ N,x,y) \right |\\
 && \qquad  \leq C \ln(1 / \gamma_1^N) \sqrt {\gamma_1^N}( t_{j}^ N- t_{i}^ N) ^{1/2}\mathcal Q_{M-d-2}(t_{j}^ N-t_{i}^ N, y-\theta_{t_{j}^ N, t_{i}^ N}^N(x)),\end{eqnarray*} 
 where $$K_N( t_{i}^ N, t_{j}^ N,x,y)= (L_{ t_{i}^ N}^N- \tilde L_{ t_{i}^ N}^N) \tilde p_N(t_{i+1}^ N, t_{j}^ N,x,y).$$
 The convolutions $K_N^{(r)}$ are calculated using the convolution $\otimes_N$, in contrast to $H_N^{(r)}$ where as above the convolution operation $\otimes$ is used.
\end{lemma}
\begin{lemma} \label{lem4:comptrunc}
{ There exists  a constant $N_0 >0$ such that for $N \geq N_0$,} for $i < j$ it holds  that
 \begin{eqnarray*} &&\left |\sum_{r=0}^N  \tilde p_N\otimes_N  K_N^{(r)}(t_{i}^ N, t_{j}^ N,x,y) - \sum_{r=0}^N \tilde p_N\otimes_N \mathcal K_N^{(r)}(t_{i}^ N, t_{j}^ N,x,y) \right |\\
 && \qquad  \leq C \sqrt {\gamma_1^N} \ln {(1/\gamma_1^N)}  \mathcal Q_{M-d-6}(t_{j}^ N-t_{i}^ N, y-\theta_{t_{j}^ N, t_{i}^ N}^N(x)).\end{eqnarray*} 
\end{lemma}
{ Lemma \ref{lem3:comptrunc}  will be proved in Subsection \ref{sec:lem27}. The proofs of the other lemmas of this subsection can be found in the online supplement of the paper.}

Theorem \ref{prop:comptrunc}  can be used for getting a result on the distributions of the truncated diffusion and truncated Robbins-Monro procedure on an increasing grid of time points. With $m_N \geq 1$, $z_1,...,z_{m_N},x \in \R^d$, $z$, $\tau^N_j$, $Q_{N,x}^{m_N}$ and $P_{N,x}^{m_N}$ defined as in the subsections \ref{subsec:main} and
\ref{subsec:compdiff} we get the following corollary of Proposition \ref{prop:compdiff} for the Hellinger distance and L$_1$-distance between the measures $Q_{x}^{m_N}$ and $Q_{N,x}^{m_N}$. The proof of this result can  be found in Section \ref{sec:comptrunc}. 

\begin{proposition}\label{prop:comptrunccorr} Suppose that \eqref{eq:tau} holds. With a measure  $\nu$ that dominates $P_{N,x}^{m_N}$ and $Q_{N,x}^{m_N}$ it holds for $x \in \R^d$ with $|x| \leq a_N/2$ { and $N \geq 1$,} that 
\begin{eqnarray*}
&& \int \left | \frac { \mathrm d Q_{N,x}^{m_N}}{\mathrm d \nu } 
-  \frac {  \mathrm d P_{N,x}^{m_N}}{ \mathrm d \nu }
\right | \mathrm d \nu  \leq C   \sqrt{\gamma_1^N} m_N \ln^2(1 / \gamma_1^N) \leq C  m_N N^{-\beta/2} \ln^2(N)  \end{eqnarray*}\end{proposition}
In particular, we have that the upper bound in the proposition converges to 0 if $m_N$ is of the form $m_N = C N^\mu$ with { $\mu <  \beta/2$.}

\section{Proofs} \label{sec:proofs}
 { From now on inequalities with lower and upper bounds depending on $N$ shall be understood as being valid  for $N \geq N_0$ where $N_0$ is chosen large enough. If the upper bound depends only on $N$ this implies that the inequality holds for all $N\geq 1$.}
\subsection{Some bounds} \label{sec:somebounds}
In this subsection we will state some bounds that will be used in the proofs in the following
subsections. The proofs of the lemmas of this subsection can be found in Subsection \ref{sec:proofsbounds}.
The first lemma states   that $F_N(t_k^N,x) \chi_{N}(x)$ is uniformly Lipschitz in $x \in \R^d$ for $N$ large enough:

\begin{lemma} \label{lem:boundsec31 1}   { With some constant $L>0$  that only depends on  the upper bounds on the first and second derivatives of $h$ introduced  in Assumption (A3) but in particular not  on $x$, $y$ and $N$, it holds  for $x,y \in \R^d$ and $N$ large enough }that
\begin{eqnarray}\label{eqapp5}&&\|F_N(t_k^N,x) \chi_{N}(x)-F_N(t_k^N,y) \chi_{N}(y) \| \leq L  \|x-y\|.
\end{eqnarray} 
\end{lemma}

Furthermore, in the following subsections we will make use of the following inequalities stated in the next lemma.
\begin{lemma} \label{lem:boundsec31 2}  For $0 \leq |\nu| \leq 4$ and $z,y \in \R^d$ it holds that
 \begin{eqnarray}
\label{eq:lem3cent2}&& |D_z^{\nu} (\tilde p_N - \tilde q_N)(t_i^N, t_j^N,z,y)| 
\leq C {\sqrt {\gamma_1 ^N}} \left (   {( t_j^N-t_i^N)^{-(|\nu|+1)/2}}+ (t_j^N-t_i^N)^{1- |\nu|/2}a_N \right ) 
 \\ \nonumber && \quad  \quad  \quad \times 
\mathcal Q_{M-d-1}(t_j^N-t_i^N,\theta^N_{t_i^N,t_j^N}(y)-z)
,\\
\label{eq:lem3cent2add} && |D_z^{\nu} \varphi(t_i^N,  t_j^N,z,y)|
\leq C ( t_j^N-t_i^N)^{-|\nu|/2} \mathcal Q_{M-d-1}(t_j^N-t_i^N,\theta^N_{t_i^N,t_j^N}(y)-z)\end{eqnarray}
for $\varphi = \tilde q_N$ and $\varphi = \tilde p_N$. 
\end{lemma}

 We now compare the flow $\theta_{t,s}^N$  defined in \eqref{eq:defthetaN}  with the flow  $\theta_{t,s}$ ($0 \leq t \leq s$) that is defined as the solutions of the following ordinary differential equations
\begin{eqnarray} \label{eq:deftheta}
\frac {\mathrm d}{\mathrm d t} \theta_{t,s}(y) = \left (\bar \alpha I - \mathcal D h(\bar \theta^N_t) \right )\theta_{t,s}(y)
\end{eqnarray}
with terminal condition
$\theta_{s,s}(y)=y$. The following lemma collects bounds for and between $ \theta_{t,T}(y)$, $ \theta^N_{t,T}(y)$, and $y$.
\begin{lemma} \label{em2:boundstheta} For all $t,t_k^N \in [0,T]$ and for all $x,y \in \R^d$ the following bounds hold
with a constant $C> 1$, depending only on $T$.
\begin{eqnarray} \label{eq:lem3cent5a}
 && C^{-1} \|\theta^N_{s,t}(y) - x\| \leq   \|\theta^N_{t,s}(x) - y\|  \leq  C \|\theta^N_{s,t}(y) - x\|, \\
  \label{eq:lem3cent5b}
 && C^{-1} \|\theta^N_{s,t}(y) \| \leq   \| y\|  \leq  C \|\theta^N_{s,t}(y)\|,\\
 \label{eq:boundstheta1} &&C^{-1} \|y\| \leq \|\theta_{t,T}(y) \| \leq C \|y\|. \end{eqnarray} 
 \end{lemma} 

\begin{lemma} \label{lem4:boundsH} For all $t,v \in [0,T]$ and for all $x,y \in \R^d$ the following bounds hold with some constant $C>0$
\begin{eqnarray} \label{eq:boundsH1} &&\left | H^{(r)} (t,v,x,y)\right| \leq C^r \frac {\Gamma^r(1/2)}{\Gamma(r/2)} (v-t) ^{(r-d-2)/2} \exp \left ( - \frac {(x - \theta_{t,v}(y))^2}{C|v-t|}\right ), \\
\label{eq:boundsH2} &&\left | H_N^{(r)} (t,v,x,y)\right| \leq C^r \frac {\Gamma^r(1/2)}{\Gamma(r/2)} (v-t) ^{(r-d-2)/2}\exp \left ( - \frac {(x - \theta_{t,v}^N(y))^2}{C|v-t|}\right ), \\
\label{eq:boundsH3} &&\left | \left( H^{(r)}- H_N^{(r)}\right ) (t,v,x,y)\right| \\
\nonumber &&\qquad \leq (r+1) C^{r+1} \frac {\Gamma^r(1/2)}{\Gamma(r/2)} \ln^2(1/\gamma_1^N)\sqrt{\gamma_1^N} (v-t) ^{{(r-d-2)/2} }(1+ |y|) \exp \left ( - \frac {(x - \theta_{t,v}^N(y))^2}{C|v-t|}\right ), 
\end{eqnarray}
where  $\Gamma(z) = \int_0^\infty t^{z-1} e^{-t} \mathrm d t$ is the Gamma function.
\end{lemma} 
The following bound follows from Theorem 1.2 in \cite{menpeszha21}.
\begin{lemma} \label{lem:boundqNadd} 
For $s < t$ and $x,y \in \R^d$  it holds that
 \begin{eqnarray*}
q_N (s,t,x,y) \leq C (t-s)^{-d/2} \exp \left ( - C \frac {(y- \theta^N_{t,s}(x))^2} {t-s}\right ). \end{eqnarray*}
\end{lemma} 

{ We conclude this subsection by stating the following simple lemma which follows from
the following inequalities
\begin{eqnarray*}  
&&(1 + |u+v|) ^{-M} \leq (1 + |u|/2) ^{-M} \leq 2^M (1 + |u|) ^{-M} \qquad \text{ for } |u| > 2 |v|, \\
&&(1 + |u+v|) ^{-M} \leq 1 \leq  (1 + 2 |v|) ^{M} (1 + |u|) ^{-M} \qquad \text{ for } |u| \leq  2 |v| .\end{eqnarray*}}

\begin{lemma} \label{lem:boundsec31 1add} 
For $r \geq 1$ and $t >0$, $z,\delta \in \R^d$ it holds that
 \begin{eqnarray}
\label{eq:Qr1}\mathcal Q_{r}(t,z + \delta )&\leq& C {}\mathcal Q_{r}(t,z  )( 1 + \| t^{-1/2} \delta\|) ^r,\\ 
\label{eq:Qr2}\| t^{-1/2} z\| \mathcal Q_{r}(t,z  )&\leq& C {}\mathcal Q_{r-1}(t,z  ).\end{eqnarray}
\end{lemma}

 \subsection{Proof of Proposition \ref{TRUNC}} \label {sec:TRUNC}

We introduce the exit time $$\tau^N_{a_N}=\inf\{k\in [1, M(N)]:\|V_{t_k^N}^N\| \geq a_N\}$$
and consider the processes $U_{t_k^N}^N$ and $V_{t_k^N}^N$ for $k \leq \tau^N_{a_N}$. { We get that { 
\begin{eqnarray*}\left\| {U}_{t_{k + 1}^{N}}^{N} - {V}_{t_{k + 1}^{N}}^{N} \right\| &\leq& \left\| {U}_{t_{k}^{N}}^{N} - {V}_{t_{k}^{N}}^{N} \right\|   + \left\| G_{N}\left( t_{k}^{N},U_{t_{k}^{N}}^{N} \right)  U_{t_{k}^{N}}^{N}{ \gamma}_{k + 1}^{N} - G_{N}\left( t_{k}^{N},U_{t_{k}^{N}}^{N} \right)  {\chi_{N}(V}_{t_{k}^{N}}^{N}){ \gamma}_{k + 1}^{N} \right\|\\
&& \qquad + \left\| G_{N}\left( t_{k}^{N},U_{t_{k}^{N}}^{N} \right)  {\chi_{N}(V}_{t_{k}^{N}}^{N}){ \gamma}_{k + 1}^{N} - F_{N}\left( t_{k}^{N},\ V_{t_{k}^{N}}^{N} \right)  {\chi_{N}(V}_{t_{k}^{N}}^{N}){ \gamma}_{k + 1}^{N} \right\| \\
&& \qquad + C\sqrt{{ \gamma}_{k}^{N}{ \gamma}_{k + 1}^{N}}\ \left\| U_{t_{k}^{N}}^{N}\  - \chi_{N}\left( V_{t_{k}^{N}}^{N} \right) \right\| + \left\| \beta_{k + 1}^{N} \right\| \\
&\leq& (1 + C\gamma_{k + 1}^{N})\left\| U_{t_{k}^{N}}^{N\ }\  - V_{t_{k}^{N}}^{N\ } \right\| + a_{N\ }\gamma_{k + 1}^{N}\left\| G_{N}(t_{k}^{N},U_{t_{k}^{N}}^{N\ })- G_{N}(t_{k}^{N},V_{t_{k}^{N}}^{N\ }) \right\|
\\ && \qquad + a_{N\ }\gamma_{k + 1}^{N}\left\| G_{N}(t_{k}^{N},V_{t_{k}^{N}}^{N\ })-F_{N}(t_{k}^{N},V_{t_{k}^{N}}^{N\ })\right)\| + \left\| \beta_{k + 1}^{N} \right\| 
\\&\leq &(1 + C\gamma_{k + 1}^{N})\left\| U_{t_{k}^{N}}^{N\ }\  - V_{t_{k}^{N}}^{N\ } \right\|{+ a}_{N\ }\gamma_{k + 1}^{N}\left\| G_{N}(t_{k}^{N},V_{t_{k}^{N}}^{N\ })\  - \ F_{N}(t_{k}^{N},V_{t_{k}^{N}}^{N\ }) \right\| \\ && \qquad+ \left\| \beta_{k + 1}^{N} \right\|.\end{eqnarray*}

For the second summand in the r.h.s. we have the upper bound (note that
\(\chi_{N}(V_{t_{k}^{N}}^{N}) = V_{t_{k}^{N}}^{N}\)
for \(k \leq \tau_{a_{N}}^{N}\) ).
\[a_{N}{ \gamma}_{k + 1}^{N}\left\| G_{N}\left( t_{k}^{N},V_{t_{k}^{N}}^{N} \right) - F_{N}\left( t_{k}^{N},V_{t_{k}^{N}}^{N} \right) \right\| \leq a_{N}{ \gamma}_{k + 1}^{N}\left| \alpha_{t_{k}^{N}}^{N}\  - \overline{\alpha} \right| + C{{{a}_{N}  (\gamma}_{k + 1}^{N})}^{2}  \alpha_{t_{k}^{N}}^{N}.\]
We now  argue that
 \begin{eqnarray*} && \sqrt {\gamma_{k}^N / \gamma_{k+1}^N} -1 \leq C  (\gamma_1^N)^{\beta^{-1}}, \label{inrev1}\\
 && | \bar \alpha - \alpha_{t_{k+1}^N}^N| \leq C I_{\{\beta=1\}} \gamma_1^N + C I_{\{\frac 1 2 <\beta<1\}}  (\gamma_1^N)^{\beta^{-1}-1},\label{inrev2}   \end{eqnarray*} 
 which can be easily shown.
 This implies that $$a_{N}{ \gamma}_{k + 1}^{N}\left| \alpha_{t_{k}^{N}}^{N}\  - \overline{\alpha} \right| \leq Ca_{N}\mathbb{I}_{\{\beta = 1\}}{{(\gamma}_{k + 1}^{N})}^{2}\  + C{a_{N}\mathbb{I}}_{\{\frac{1}{2}\  < \beta < 1\}}{(\gamma_{k + 1}^{N}\ )}^{\beta^{- 1}}. $$ We now use this together with the fact that \(\left\| \beta_{k + 1}^{N} \right\|\) is bounded by a constant times
\({{(\gamma}_{k + 1}^{N})}^{3/2}\) and we get that 
\begin{eqnarray*}\left\| {U}_{t_{k + 1}^{N}}^{N} - {V}_{t_{k + 1}^{N}}^{N} \right\| &\leq& (1 + C{ \gamma}_{k + 1}^{N})\left\| U_{t_{k}^{N}}^{N} - {V}_{t_{k}^{N}}^{N} \right\| + C{{(\gamma}_{k + 1}^{N})}^{3/2}\  + C{a_{N}\mathbb{I}}_{\{\frac{1}{2}\  < \beta < 1\}}{(\gamma_{k + 1}^{N}\ )}^{\beta^{- 1}}. \end{eqnarray*}

  Now, by definition of $M(N)$ we have that $ \gamma_{1}^N + ...+ \gamma_{M(N)-1}^N < T$. Because of $k< M(N)$ for $k <\tau^N_{a_N}$ we have for some constants $C'$, $C''$ that
\begin{eqnarray*} \left \|U_{t_{k+1}^N}^N-V_{t_{k+1}^N}^N\right \| &\leq& (1 + C' \gamma_{M(N)}^N) \cdot  ... \cdot  (1 + C' \gamma_{1}^N) \sum_{l=1} ^{M(N)} C' ( (\gamma_{l}^{N})^{3/2}  + a_{N}\mathbb{I}_{\{\frac{1}{2}  < \beta < 1\}} (\gamma_{l}^{N})^{\beta^{- 1}})\\
&\leq&C'' ( (\gamma_{1}^{N})^{1/2}  + a_{N}\mathbb{I}_{\{\frac{1}{2}  < \beta < 1\}} (\gamma_{1}^{N})^{\beta^{- 1}-1}),\end{eqnarray*}
where $M(N) \leq C (\gamma^{N}_1) ^{-1}$ has been used. 
We  conclude that}
\begin{eqnarray*}&&\mathbb P \left (\sup_{1 \leq k \leq M(N)} \left \| U_{t_k^N}^N - V_{t_k^N}^N \right \| > C'' ( (\gamma_{1}^{N})^{1/2}  + a_{N}\mathbb{I}_{\{\frac{1}{2}  < \beta < 1\}} (\gamma_{1}^{N})^{\beta^{- 1}-1})\right )\\
&& \qquad  \leq  \mathbb P \left (\tau^N_N < M(N)\right)
\\
&& \qquad = \mathbb P \left (\sup_{1 \leq k \leq M(N)-1} \left \| V_{t_k^N}^N  \right \| > a_N\right )\\
&& \qquad \leq \mathbb P \left (\sup_{1 \leq k \leq M(N)} \left \| V_{t_k^N}^N  \right \| > a_N\right ).\end{eqnarray*}}
As mentioned after the statement of our assumptions the process $U_{t}^N$ converges in distribution to the diffusion $(X_t: 0 \leq t \leq T)$ defined in \eqref{EDS_limit}. The same holds for the process $V_{t}^N$, see  Lemma 11.2.1, Theorem 10.2.2 and Theorem 11.2.3  in \cite{stro:vara:79}. Now for any $\epsilon > 0$ there exists a level $K_\epsilon$ with 
$$  \mathbb  P(\sup_{0 \leq t \leq T} \|X_t\| \geq K_\epsilon) \leq \epsilon.$$
This shows that the upper bound $\mathbb P \left (\sup_{1 \leq k \leq M(N)} \left \| V_{t_k^N}^N  \right \| > a_N\right )$ converges to 0 because of $a_N \to \infty$ for $N\to \infty$ which concludes the proof of the proposition.  \hfill $\square$
 
  \subsection {Proof of Lemma \ref{lem3:comptrunc}} \label{sec:lem27} 
 We will show\begin{eqnarray}
\label{eq:lem3cent1a} |(H_N -K_N)(t_i^N,  t_j^N,z,y)| &\leq& C |z - \theta_{t_i^N, t_j^N}^N(y)|\  |\nabla_z (\tilde p_N - \tilde q_N)(t_i^N,  t_j^N,z,y)|,\\
\label{eq:lem3cent3}
|H_N(t_i^N, t_j^N,z,y)| &\leq&  C  \mathcal Q_{M-d-1}(t_j^N-t_i^N,z-\theta^N_{t_i^N,t_j^N}(y)),\\
\label{eq:lem3cent4} |(\tilde p_N \otimes_N K^{(r)}_N)(t_i^N, t_j^N, z,y)| &\leq& \frac {(C(t_j^N- t_i^N))^{r}}{r!} \mathcal Q_{M-d-1}(t_j^N- t_i^N,z- \theta_{t_i^N, t_j^N}^N(y)),
\end{eqnarray}
and with some constant $\bar c$ and with $m= M-d-5-\gamma$
\begin{eqnarray}
\label{eq:lem3cent5} && \int_{\mathbb R ^d} \mathcal Q_{m}(t_k^N-t_i^N,z-\theta^N_{t_k^N,t_i^N}(x))\mathcal Q_{m}(t_j^N-t_k^N,y-\theta^N_{t_j^N,t_k^N}(z)) \mathrm d z \\
\nonumber && \qquad \leq  \bar c \mathcal Q_{m}(t_j^N-t_i^N,y-\theta^N_{t_j^N,t_i^N}(x))
\end{eqnarray}
for all $1 \leq i < k \leq j, x,y \in \mathbb R ^d$. In the proof of the lemma we will make repeated use of Lemma \ref{lem:boundsec31 2}. We will use it to bound the right hand side of 
\eqref{eq:lem3cent1a} and in the proof of \eqref{eq:lem3cent4}. At both places we replace the Gaussian densities $\tilde q_N$ by the densities $\tilde p_N$ of normed sums of independent random variables. By application of the lemma we get with the help of \eqref{eq:Qr2} from \eqref{eq:lem3cent1a} that
\begin{eqnarray}
\label{eq:lem3cent6}
|(H_N -K_N)(t_i^N,  t_j^N,z,y)| &\leq&  C \left ( \frac {\sqrt {\gamma_1 ^N}} {\sqrt{ t_j^N-t_i^N}}+ (t_j^N-t_i^N)a_N \sqrt{\gamma_1^N}\right ) \\
\nonumber && \quad \times \mathcal Q_{M-d-2}(t_j^N-t_i^N,\theta^N_{t_i^N,t_j^N}(y)-z).
\end{eqnarray}
We now show that \eqref{eq:lem3cent3}--\eqref{eq:lem3cent6} imply the statement of the lemma. 
For a proof of this claim we write 
\begin{eqnarray}
\label{eq:lem3cent7}
\tilde q_N\otimes_N H_N^{(r+1)}( t_i^N, t_j^N,x,y)-  \tilde p_N\otimes_N K_N^{(r+1)}( t_i^N, t_j^N,x,y) =I + II,
\end{eqnarray}
where 
\begin{eqnarray*}
I&=& (\tilde q_N\otimes_N H_N^{(r)}
-  \tilde p_N\otimes_N K_N^{(r)})\otimes_N H_N( t_i^N, t_j^N,x,y),\\
II&=& \tilde p_N\otimes_N K_N^{(r)} \otimes_N (H_N
- K_N)( t_i^N, t_j^N,x,y).
\end{eqnarray*}

For a discusssion of the second term II note that we get directly from \eqref{eq:lem3cent4}--\eqref{eq:lem3cent6} that
\begin{eqnarray*}\big | II \big |&\leq& \sum_{k=i}^{j-1} \gamma_{k+1}^N \int_{\mathbb R ^d} |\tilde p_N\otimes_N K_N^{(r)} | (t_i^N,t_k^N,x,z) |H_N
- K_N|(t_k^N,t_j^N,z,y) \mathrm d z \\
&\leq& a_N \sqrt {\gamma_1^N}  \frac {C^{r+1}}{r!}  \sum_{k=i}^{j-1} \frac {(t_k^N-t_i^N)^{r}}{\sqrt{t_j^N- t_k^N}}\gamma_{k+1}^N \int_{\mathbb R ^d}\mathcal Q_{M-d-1}(t_k^N- t_i^N,x-\theta^N_{t_i^N,t_k^N}(z))\\
&& \qquad \times  \mathcal Q_{M-d-1}(t_j^N-t_k^N,\theta^N_{t_k^N,t_j^N}(y)-z)  \mathrm d z\\
&\leq& a_N \sqrt {\gamma_1^N}  \frac {C^{r+1}}{r!}  \sum_{k=i}^{j-1} \frac {(t_k^N-t_i^N)^{r}}{\sqrt{t_j^N- t_k^N}}\gamma_{k+1}^N \mathcal Q_{M-d-1}(t_j^N-t_i^N,y-\theta^N_{t_j^N,t_i^N}(x))\\
&\leq& a_N \sqrt {\gamma_1^N}  \frac {C^{r+1}}{r!}  (t_j^N-t_i^N)^{r+ \frac 1 2 } B\left ( \frac 1 2, r+1\right )
\mathcal Q_{M-d-1}(t_j^N-t_i^N,y-\theta^N_{t_j^N,t_i^N}(x))
\\
&\leq& a_N \sqrt {\gamma_1^N}  \frac {C^{r+2}}{\Gamma \left ( r+\frac 3 2 \right )} (t_j^N-t_i^N)^{r+ \frac 1 2 } \mathcal Q_{M-d-1}(t_j^N-t_i^N,y-\theta^N_{t_j^N,t_i^N}(x)),
\end{eqnarray*}
where $B(z_1,z_2) = \int_0^1 t^{z_1-1} (1-t)^{z_2-1} \mathrm d t$ is the Beta function and $\Gamma$ is the Gamma function, see the statement of Lemma \ref{lem4:boundsH}.
We can apply this inequality to show that
\begin{eqnarray}\label{eq:lem3tool1}&&\sum_{r=0}^\infty  \left |\tilde p_N\otimes_N K_N^{(r)} \otimes_N (H_N
- K_N)(t_i^N, t_j^N,x,y)\right |\\ \nonumber && \qquad \leq a_N (t_j^N-t_i^N)^{1/2}\sqrt {\gamma_1^N} \mathcal Q_{M-d-1}(t_j^N-t_i^N,y-\theta^N_{t_j^N,t_i^N}(x)).
\end{eqnarray}
We now write with  $H_N^{(k),\otimes_N}$ similarly defined as $H_N^{(k)}$ as a $k$-times convolution of $H_N$ but now with using the convolution operator $\otimes_N$ instead of $\otimes$
\begin{eqnarray*}
&&\left|\sum_{r=0} ^N \tilde q_N\otimes_N H_N^{(r)}(t_i^N,t_j^N,x,y)-  \tilde p_N\otimes_N K_N^{(r)}(t_i^N, t_j^N,x,y) \right|\\
&& \qquad = \left|\sum_{r=0} ^N \tilde q_N \otimes_N (H_N^{(r)} -H_N^{(r), \otimes_N})(t_i^N, t_j^N,x,y)\right.\\
&& \qquad \qquad + \left. \sum_{r=0} ^N (\tilde q_N -  \tilde p_N)\otimes_N H_N^{(r), \otimes_N}(t_i^N, t_j^N,x,y)\right.\\
&& \qquad \qquad - \left .\sum_{r=1} ^N \sum_{k=0} ^{r-1} \tilde p_N\otimes_N K_N^{(k)} \otimes_N (H_N
- K_N) \otimes_N H_N^{(r-1-k), \otimes_N} (t_i^N, t_j^N,x,y)
\right|
\\
&& \qquad \leq  \left|\tilde q_N \right | \otimes_N \sum_{r=0} ^N \left |H_N^{(r)}-H_N^{(r), \otimes_N}\right|(t_i^N,t_j^N,x,y)\\
&& \qquad \qquad +  \left|\tilde q_N -  \tilde p_N\right | \otimes_N \sum_{r=0} ^N \left |H_N^{(r), \otimes_N}\right| (t_i^N,t_j^N,x,y)\\
&& \qquad \qquad + \sum_{k\geq 0}  \left |\tilde p_N\otimes_N K_N^{(k)} \otimes_N (H_N
- K_N) \right | \otimes_N \sum_{r=0} ^N \left |H_N^{(r), \otimes_N} 
\right| (t_i^N, t_j^N,x,y),
\end{eqnarray*}
where 
$$ \tilde p_N\otimes_N (  H_N^{(r), \otimes_N}-K_N^{(r)})= \sum_{k=0} ^{r-1} \tilde p_N\otimes_N K_N^{(k)} \otimes_N (H_N
- K_N) \otimes_N H_N^{(r-1-k), \otimes_N}$$
for $r\geq 1$ and 
$$ \tilde p_N\otimes_N (  H_N^{(r), \otimes_N}-K_N^{(r)})=0$$
for $r=0$ has been used. 

The first term of this upper bound can be bounded as follows.
 \begin{eqnarray} \label{eq:claim1added} \left|\tilde q_N \right | \otimes_N \sum_{r=0} ^N \left |H_N^{(r)}-H_N^{(r), \otimes_N}\right|(t_i^N,t_j^N,x,y) \leq C \ln^2(1 / \gamma_1^N) \sqrt {\gamma_1^N} \sqrt { t_{j}^ N- t_{i}^ N} \bar q_N(t_{i}^ N, t_{j}^ N,x,y). \end{eqnarray} 
 This inequality can be shown by similar arguments as the proof of \eqref{eq:claim1} in the proof of Lemma \ref{lem1:comptrunc}.
 
{The statement of the lemma now follows by application of 
\eqref{eq:lem3cent2}, \eqref{eq:boundsH1},
\eqref{eq:lem3tool1}, and \eqref{eq:claim1added}.  Furthermore, we apply  \eqref{eq:claim1chelp} that was proved in the proof of Lemma \ref{lem1:comptrunc}.} It remains to show \eqref{eq:lem3cent1a}--
\eqref{eq:lem3cent5}. 

For a proof of \eqref{eq:lem3cent1a} note that 
\begin{eqnarray*} 
&&(H_N-K_N)(t_i^N, t_j^N, z,y) = (L_{t_i^N}^N - \tilde L_{t_i^N}^N) (\tilde q_N - \tilde p_N)(t_i^N, t_j^N,z,y)\\
&& \qquad =\left (\sum_{i=1}^d \left ( \sum_{j=1} ^d [F_N(t_i^N,z)]_{i,j}[\chi_N(z)]_j\right )    
-\sum_{i=1}^d \left ( \sum_{j=1} ^d [F_N(t_i^N,\theta_{t_i^N,t_j^N} ^N(y))]_{i,j}[\chi_N(\theta_{t_i^N,t_j^N} ^N(y))]_j\right )           
\right )\\
&& \qquad \qquad \times \frac {\partial} {\partial z_i} (\tilde q_N - \tilde p_N)(t_i^N, t_j^N,z,y) .
\end{eqnarray*}
With the help of \eqref{eqapp5} this shows  \eqref{eq:lem3cent1a}.

Claim \eqref{eq:lem3cent3} follows directly by application of  
\eqref{eq:boundsH2}.

We now come to the proof of \eqref{eq:lem3cent5}. For this purpose we 
  show for $0 \leq t < s < u \leq T$, $x,y \in \mathbb R^d$ that with some constant $C> 0$
\begin{eqnarray*}
I (s,t,u,x,y) \leq   C \mathcal Q_{m}(u-t,y-\theta^N_{u,t}(x)),
\end{eqnarray*} 
where $$I(s,t,u,x,y) = \int_{\mathbb R ^d} \mathcal Q_{m}(s-t,z-\theta^N_{s,t}(x))\mathcal Q_{m}(u-s,y-\theta^N_{u,s}(z)) \mathrm d z.$$
This is equivalent to \eqref{eq:lem3cent5}.
For the proof of this claim note first that we get from \eqref{eq:lem3cent5b} that
$$I (s,t,u,x,y) \leq C \int_{\mathbb R ^d} \mathcal Q_{m}(s-t,z-\theta^N_{s,t}(x)) \mathcal Q_{m}(u-s,z-\theta^N_{s,u}(y))\mathrm d z.$$
We now consider two cases: I. $\| y-\theta^N_{u,t}(x)\| \leq \sqrt{u-t}$ and II.  $\| y-\theta^N_{u,t}(x)\| > \sqrt{u-t}$. We start by considering case I. We make the additional assumption that 
$s-t \geq \frac 1 2 (u-t)$. The case $s-t \leq \frac 1 2 (u-t)$ can be treated with the same type of arguments and for this reason its discussion is omitted. Now we get in the latter case:
\begin{eqnarray*} 
&&\mathcal Q_{m}(s-t,z-\theta^N_{s,t}(x)) = \frac 1 {(s-t)^{d/2}  } Q_m\left ( \frac {z-\theta^N_{s,t}(x)} {\sqrt{s-t}}\right )\\
&& \qquad \leq \frac {2 ^{d/2} } {(u-t) ^{d/2}} c_m  \leq \frac {2 ^{d/2} } {(u-t) ^{d/2}} c_m \frac {2 ^m} {\left ( 1 + \frac {\| y - \theta^N_{u,t}(x)\|}{\sqrt {u-t}}\right)^m}\\
&& \qquad = 2 ^{m+d/2}\mathcal Q_{m}(u-t,y-\theta^N_{u,t}(x))
\end{eqnarray*}
This gives in case I the following bound for $I (s,t,u,x,y)$
$$I (s,t,u,x,y) \leq C  \mathcal Q_{m}(u-t,y-\theta^N_{u,t}(x)) \int_{\mathbb R ^d}  \mathcal Q_{m}(u-s,z-\theta^N_{s,u}(y))\mathrm d z  \leq C  \mathcal Q_{m}(u-t,y-\theta^N_{u,t}(x)),$$
which shows \eqref{eq:lem3cent5} for case I. We now consider case II: $\| y-\theta^N_{u,t}(x)\| > \sqrt{u-t}$. We define the following two sets:
\begin{eqnarray*} 
 A_1 &=& \{ z \in \mathbb R^d: \|z - \theta_{s,t}^N(x)\| \geq \frac 1 2 \| \theta_{s,u} ^N(y)- \theta_{s,t}^N(x) \| \}, \\
 A_2 &=& \{ z \in \mathbb R^d: \|z - \theta_{s,u}^N(y)\| \geq \frac 1 2 \| \theta_{s,u} ^N(y)- \theta_{s,t}^N(x) \| \}.
\end{eqnarray*}
It holds $A_1 \cup A_2 =  \mathbb R^d$. We only consider values of $z$ in $A_2$. For such $z$ we get by application of \eqref{eq:lem3cent5b} 
\begin{eqnarray*} 
 && \|z - \theta_{s,u}^N(y)\| \geq \frac 1 2 \| \theta_{s,u} ^N(y)- \theta_{s,t}^N(x) \| \\
 && \qquad =  \frac 1 2 \| \theta_{s,u} ^N(y)- \theta_{s,u} ^N(\theta_{u,t}^N(x))\| \geq C \| y- \theta_{u,t}^N(x)\|.
 \end{eqnarray*}
 This gives in case II
 \begin{eqnarray*} 
 && \int_{A_2} \mathcal Q_{m}(s-t,z-\theta^N_{s,t}(x)) \mathcal Q_{m}(u-s,z-\theta^N_{s,u}(y))\mathrm d z  \\
 && \qquad \leq  \int_{A_2} \mathcal Q_{m}(s-t,z-\theta^N_{s,t}(x))\frac {c_m (u-s) ^{(m-d)/2} }{ \|\theta_{s,u}^N(y) -z\| ^m}\mathrm d z\\
 && \qquad \leq  \int_{A_2} \mathcal Q_{m}(s-t,z-\theta^N_{s,t}(x))\mathrm d z \frac {c_m (u-s) ^{(m-d)/2} }{C^m \|y- \theta_{u,t}^N(x) \| ^m}\\
 && \qquad \leq   c_m (u-t) ^{-d/2} C^{-m} 2^m \left (1+ \|y- \theta_{u,t}^N(x) \|\right ) ^{-m}\\
 && \qquad \leq  \mathcal Q_{m}(u-t,y-\theta^N_{u,t}(x)).
 \end{eqnarray*}
The same bound can be shown for integrals over the set $A_1$. This completes the proof of \eqref{eq:lem3cent5} for case II.

It remains to show \eqref{eq:lem3cent4}.
For a proof of this claim note that 
\begin{eqnarray} \label{eq:boundKN}
&&K_N(t_i^N, t_j^N, z,y) = (L_{t_i^N}^N - \tilde L_{t_i^N}^N)  \tilde p_N(t_i^N, t_j^N,z,y)\\ \nonumber
&& \qquad =\left (\sum_{i=1}^d \left ( \sum_{j=1} ^d [F_N(t_i^N,z)]_{i,j}[\chi_N(z)]_j\right )    
-\sum_{i=1}^d \left ( \sum_{j=1} ^d [F_N(t_i^N,\theta_{t_i^N,t_j^N} ^N(y))]_{i,j}[\chi_N(\theta_{t_i^N,t_j^N} ^N(y))]_j\right )           
\right )\\ \nonumber
&& \qquad \qquad \times \left (\frac {\partial} {\partial z_i} (\tilde p_N- \tilde q_N )+ \frac {\partial} {\partial z_i}  \tilde q_N \right )(t_i^N, t_j^N,z,y) \\ \nonumber
&& \qquad \leq C \left (\frac { \sqrt {\gamma_1^N}} {\sqrt{t_j ^N - t_i ^N}} + (t_j ^N - t_i ^N) a_N \sqrt {\gamma_1^N} \right ) \mathcal Q_{M-d-2} (t_j ^N - t_i ^N, \theta^N_{t_i ^N, t_j ^N}(y)-z)\\ \nonumber
&& \qquad \qquad + C \bar q_N(t_i^N, t_j^N,z,y)\\ \nonumber
&& \qquad \leq C  \mathcal Q_{M-d-2} (t_j ^N - t_i ^N, \theta^N_{t_i ^N, t_j ^N}(y)-z),
\end{eqnarray}
where again Lemma \ref{lem:boundsec31 2} has been used.
With the help of \eqref{eq:lem3cent5}  this gives that
\begin{eqnarray*} 
&&\left | ( \tilde p_N \otimes_N K_N)(t_i^N,t_j^N, x, z) \right | \leq \sum_{k=i}^{j-1} \gamma_{k+1}^N \int_{\R^d} | \tilde p_N (t_i^N,t_k^N, x,v)|  | K_N(t_k^N,t_j^N, v, z)| \mathrm d v\\
&& \qquad \leq C \sum_{k=i}^{j-1} \gamma_{k+1}^N \int_{\R^d}  \mathcal Q_{M-d-2} (t_k ^N-t_i ^N, v-\theta^N_{t_k^N,t_i ^N}(x) )\mathcal Q_{M-d-2} (t_j ^N - t_k ^N, \theta^N_{t_k^N, t_j ^N}(z)-v) \mathrm d v\\
&& \qquad \leq C  \mathcal Q_{M-d-2} (t_j ^N-t_i ^N, z-\theta^N_{t_j ^N,t_i ^N}(x) ) \sum_{k=i}^{j-1} \gamma_{k+1}^N \\
&& \qquad \leq C  \mathcal Q_{M-d-2} (t_j ^N-t_i ^N, z-\theta^N_{t_j ^N,t_i^N}(x) )(t_j^N -t_i ^N).
\end{eqnarray*}
Similarly we get that 
\begin{eqnarray*} 
&&\left | ( \tilde p_N \otimes_N K_N^{(2)})(t_i^N,t_j^N, x, z) \right | \leq \sum_{k=i}^{j-1} \gamma_{k+1}^N \int_{\R^d} |( \tilde p_N \otimes_N K_N)(t_i^N,t_k^N, x,v)|  | K_N(t_k^N,t_j^N, v, z)| \mathrm d v\\
&& \qquad \leq C^2 \sum_{k=i}^{j-1} \gamma_{k+1}^N (t_k^N-t_i ^N)\int_{\R^d}  \mathcal Q_{M-d-2} (t_k ^N-t_i ^N, v-\theta^N_{t_k ^N,t_i^N}(x) )\\
&& \qquad \qquad \times \mathcal Q_{M-d-2} (t_j^N - t_k ^N, \theta^N_{t_k ^N, t_j^N}(z)-v) \mathrm d v\\
&& \qquad \leq \frac {(C (t_j ^N-t_i^N) ) ^2} {2!}\mathcal Q_{M-d-2} (t_j ^N-t_i ^N, z-\theta^N_{t_j ^N,t_i ^N}(x) ),
\end{eqnarray*}
where $C$ is the same constant as in the last inequality. By induction we conclude that
\begin{eqnarray*} 
&&\left | ( \tilde p_N \otimes_N K_N^{(r+1)})(t_i^N,t_j^N, x, z) \right | \leq \sum_{k=i}^{j-1} \gamma_{k+1}^N \int_{\R^d} |( \tilde p_N \otimes_N K_N^{(r)})(t_i^N,t_k^N, x,v)|  | K_N(t_k^N,t_j^N, v, z)| \mathrm d v\\
&& \qquad \leq \frac {C^{r+1} } {r!} \sum_{k=i}^{j-1} \gamma_{k+1}^N (t_k^N-t_i^N) ^r\int_{\R^d}  \mathcal Q_{M-d-2} (t_k ^N- t_i ^N, v-\theta^N_{t_k ^N,t_i ^N}(x) )\\
&& \qquad \qquad  \times \mathcal Q_{M-d-2} (t_j^N - t_k ^N, \theta^N_{t_k ^N, t_j ^N}(z)-v) \mathrm d v\\
&& \qquad \leq  \frac {C^{r+1} } {r!}  \int_0^{t_j ^N-t_i ^N} u^r \mathrm d u\mathcal Q_{M-d-2} (t_j ^N-t_i ^N, z-\theta^N_{t_j^N,t_i ^N}(x) )\\
&& \qquad \leq  \frac {(C(t_j ^N-t_i ^N))^{r+1} } {(r+1)!} \mathcal Q_{M-d-2} (t_j ^N-t_i ^N, z-\theta^N_{t_j^N,t_i ^N}(x) ),
\end{eqnarray*}
which shows \eqref{eq:lem3cent4} and concludes the proof of Lemma \ref{lem3:comptrunc}.
 \hfill $\square$
 
  \subsection {Proof of Proposition \ref{prop:comptrunccorr}} \label{sec:comptrunc}
  Using a telescopic sum, we have with putting $x_0=x$, $x_{m_N}=y$, $\prod_{k=1} ^0 ... = 1$, and $\prod_{k=l+1} ^{l} ... = 1$
  \begin{eqnarray} \label{eq:helpgridbound}
&& \int \left | \frac { \mathrm d Q_{N,x}^{m_N}}{\mathrm d \nu } 
-  \frac {  \mathrm d P_{N,x}^{m_N}}{ \mathrm d \nu }
\right | \mathrm d \nu  \\ \nonumber
&& = \int_{\mathbb R^{dm_N}} \left| \prod_{i=1} ^{m_N} q_N(\tau^N_{i-1}, \tau^N_i, x_{i-1}, x_{i})- \prod_{i=1} ^{m_N} p_N(\tau^N_{i-1}, \tau^N_i, x_{i-1}, x_{i}) \right | \mathrm d x_1 ... \mathrm d x_{m_N }\\ \nonumber
&& \leq \sum_{i=1} ^{m_N} \int_{\mathbb R^{dm_N}} \left|  q_N(\tau^N_{i-1}, \tau^N_i, x_{i-1}, x_{i})- p_N(\tau^N_{i-1}, \tau^N_i, x_{i-1}, x_{i}) \right | \\ \nonumber
&&\qquad \times 
\prod_{k=1} ^{i-1} q_N(\tau^N_{k-1}, \tau^N_k, x_{k-1}, x_{k}) \prod_{l=i+1} ^{m_N} p_N(\tau^N_{l-1}, \tau^N_l, x_{l-1}, x_{l})\mathrm d x_1 ... \mathrm d x_{m_N}
\\ \nonumber
&& \leq \sum_{i=1} ^{m_N} \int_{\mathbb R^{3d}} \left|  q_N(\tau^N_{i-1}, \tau^N_i, x_{i-1}, x_{i})- p_N(\tau^N_{i-1}, \tau^N_i, x_{i-1}, x_{i}) \right | \\ \nonumber
&&\qquad \times 
q_N(0, \tau^N_{i-1}, x, x_{i-1}) p_N(\tau^N_{i}, T, x_{i}, y)\mathrm d x_{i-1} \  \mathrm d x_{i} \  \mathrm d y\\
&& = \sum_{i=1} ^{m_N} \int_{\mathbb R^{2d}} \left|  q_N(\tau^N_{i-1}, \tau^N_i, x_{i-1}, x_{i})- p_N(\tau^N_{i-1}, \tau^N_i, x_{i-1}, x_{i}) \right | \\ \nonumber
&&\qquad \times 
q_N(0, \tau^N_{i-1}, x, x_{i-1}) \mathrm d x_{i-1} \  \mathrm d x_{i} 
.\end{eqnarray}

Now by Lemma \ref{lem:boundqNadd}  we have that  
 \begin{eqnarray*}
  q_N(0, \tau^N_{i-1}, x, x_{i-1}) &\leq& C (\tau^N_{i-1})^{-d/2} \exp ( - C (x_{i-1} -\theta_{\tau^N_{i-1},0}^N(x)) ^2 / \tau^N_{i-1})\\
  &\leq& C 
  \mathcal Q_{M-d-6}(\tau^N_{i-1},x_{i-1}-\theta_{\tau^N_{i-1},  0}^N(x)).
  \end{eqnarray*}
  Furthermore, we have by Theorem \ref{prop:comptrunc} and \eqref{eq:tau} 
 \begin{eqnarray*}&&\left|  q_N(\tau^N_{i-1}, \tau^N_i, x_{i-1}, x_{i})- p_N(\tau^N_{i-1}, \tau^N_i, x_{i-1}, x_{i}) \right |\\
 && \qquad \leq C  \sqrt{\gamma_1^N} \ln^2(1 / \gamma_1^N) \mathcal Q_{M-d-6}(\tau^N_{i}-  \tau^N_{i-1},x_{i}-\theta_{\tau^N_{i}, \tau^N_{i-1}}^N(x_{i-1}))\\
 && \qquad \leq C \sqrt{\gamma_1^N} \ln^2(1 / \gamma_1^N)  \mathcal Q_{M-d-6}(\tau^N_{i}-  \tau^N_{i-1},x_{i}-\theta_{\tau^N_{i}, \tau^N_{i-1}}^N(x_{i-1})).\end{eqnarray*}

Now the statement of the proposition follows by  application of \eqref{eq:lem3cent5}. \hfill $\square$

	\subsection{Proof of Lemma \ref{DYN:MKV:CHAIN}} \label {sec:DYN:MKV:CHAIN}
Note first that
\begin{eqnarray*}
U^N_{t_{k+1}^N} &=& \frac{\theta_{k+1}^N - \bar{\theta}^N_{t_{k+1}^N}}{\sqrt{\gamma^N_{k+1}}}\\
&=&\frac{\theta^N_{k} - \gamma_{k+1}^NH(\theta_{k}^N,\eta_{k+1}^N) - \bar{\theta}^N_{t_{k+1}^N}}{\sqrt{\gamma^N_{k+1}}}\\
&=& \frac{\theta^N_{k} - \bar{\theta}^N_{t_{k}^N}}{\sqrt{\gamma^N_{k}}} \frac{\sqrt{\gamma^N_k}}{\sqrt{\gamma^N_{k+1}}} - \sqrt{\gamma^N_{k+1}} H(\theta_k^N,\eta_{k+1}^N) - \frac{\bar{\theta}^N_{t_{k+1}^N} -\bar{\theta}^N_{t_{k}^N}}{\sqrt{\gamma^N_{k+1}}}\\
&=&U_{t_k}^N  \frac{\sqrt{\gamma^N_k}}{\sqrt{\gamma^N_{k+1}}} - \sqrt{\gamma^N_{k+1}} H(\bar{\theta}^N_{t_{k}^N}+ U_{t_k}^N \sqrt{\gamma^N_{k}},\eta_{k+1}^N) - \frac{\bar{\theta}^N_{t_{k+1}^N} -\bar{\theta}^N_{t_{k}^N}}{\sqrt{\gamma^N_{k+1}}}.
\end{eqnarray*}

Now, we can write:
\begin{eqnarray*}
U_{t_k}^N  \frac{\sqrt{\gamma^N_k}}{\sqrt{\gamma^N_{k+1}}} &=& U_{t_k}^N + U_{t_k}^N  \frac{\sqrt{\gamma^N_k}-\sqrt{\gamma^N_{k+1}}}{\sqrt{\gamma^N_{k+1}}} \\
&=&U_{t_k}^N +  \alpha_{t_k^N}^N U_{t_k}^N   \gamma^N_{k+1}.
\end{eqnarray*}
The dynamics of $U_{t_k}^N $ becomes:
\begin{eqnarray*}
  U^N_{t_{k+1}^N}&=& U_{t_k^N}^N +  \alpha_{t_k}^N U_{t_k^N}^N   \gamma^N_{k+1}- \sqrt{\gamma^N_{k+1}} \Big( H(\bar{\theta}^N_{t_{k}^N}+ U_{t_k^N} ^N\sqrt{\gamma^N_{k}},\eta_{k+1}^N) - h(\bar{\theta}^N_{t_{k}^N}+U_{t_k^N}^N\sqrt{\gamma^N_{k}}
  ) \Big) \\
  &&- \sqrt{\gamma^N_{k+1}} \Big( h(\bar{\theta}^N_{t_{k}^N}+U_{t_k^N}^N\sqrt{\gamma^N_{k}}
  ) - h(\bar{\theta}^N_{t_{k}^N}) \Big) +\sqrt{\gamma^N_{k+1}} \left(- h(\bar{\theta}^N_{t_{k}^N})- \frac{\bar{\theta}^N_{t_{k+1}^N} -\bar{\theta}^N_{t_{k}^N}}{\gamma^N_{k+1}} \right)\\
  &=&U_{t_k^N}^N  + G_N(t_k^N,U_{t_k^N}^N) U_{t_k^N}^N \gamma^N_{k+1} - \sqrt{\gamma^N_{k+1}}  \xi(\bar{\theta}^N_{t_{k}^N}+ U_{t_k^N} ^N\sqrt{\gamma^N_{k}},\eta_{k+1}^N) + \beta_{k+1}^N,
\end{eqnarray*}
which shows the statement of the lemma.  \hfill $\square$

	\subsection{Proof of Lemma \ref{BETA}} \label {sec:BETA} {Here we need that $\| \bar \theta^N_t - \theta^*\|$ is small enough such  that }
 $\|h(\bar \theta^N_t)\| \leq C$ and $\|\mathcal D h(\bar \theta^N_t)\| \leq C$ for some $C>0$ for all $t$.
This gives 
\begin{eqnarray*}\|\beta_{k+1}^N\| &\leq & 
\sqrt{\gamma^N_{k+1}} \left \| -h(\bar{\theta}^N_{t^N_k})- \frac{\bar{\theta}^N_{t^N_{k+1}} -\bar{\theta}^N_{t^N_k}}{\gamma^N_{k+1}} \right\|
 \\ 
 &= & \frac {\sqrt{\gamma^N_{k+1}}} {t^N_{k+1} - t^N_{k}}\left \|  \int _{t^N_{k}}^{t^N_{k+1}} \left (-h(\bar \theta^N_{t^N_k})+h(\bar{\theta}^N_t)\right ) \mathrm d t  \right\|
  \\ 
 &= & \frac {\sqrt{\gamma^N_{k+1}}} {t^N_{k+1} - t^N_{k}}\left \| \int _{t^N_{k}}^{t^N_{k+1}} \int _0^1 \mathcal D h(\bar \theta^N_{t^N_k}+\delta (\bar{\theta}^N_t-\bar \theta^N_{t^N_k}))
 \left (\bar \theta^N_{t^N_k}-\bar{\theta}^N_t\right )  \mathrm d \delta \mathrm d t  \right\|
  \\ 
 &\leq & C \frac {\sqrt{\gamma^N_{k+1}}} {t^N_{k+1} - t^N_{k}} \int _{t^N_{k}}^{t^N_{k+1}} \left \|\bar \theta^N_{t^N_k}-\bar{\theta}^N_t\right\| \mathrm d t  
   \\ 
 &\leq & C \frac {\sqrt{\gamma^N_{k+1}}} {t^N_{k+1} - t^N_{k}} \int _{t^N_{k}}^{t^N_{k+1}} \left \| \int _{0}^{1} h(\bar \theta^N_{t^N_{k}+ \delta (t-    t^N_{k} )})  (t-    t^N_{k} ) \mathrm d \delta \right\|  \mathrm d t 
  \\ 
 &\leq & C^2  {\sqrt{\gamma^N_{k+1}}} (t^N_{k+1} - t^N_{k})
 \\ 
 &= & C^2 ( \gamma^N_{k+1}) ^{3/2} .\end{eqnarray*}
Now $ \gamma^N_{k+1}$ converges to 0 under our assumptions. This implies the statement of the lemma.  \hfill $\square$

 \subsection{Proofs of Propositions \ref{prop:compdiff} and \ref{prop:compdiffcorr}} \label {sec:compdiff}
 \begin{proof} [Proof of Proposition \ref{prop:compdiff}]
 We start by showing that 
  \begin{eqnarray} \label{diftruncbound1} && \left ( \int_{\R^d} \left | q_N- q\right |(s,t,x,z) \mathrm d z \right )^2 \\ \nonumber
  && \qquad \leq C \int_s^t \int_{\R^d} \left | R^{-1/2} (\bar \theta^N _u) \left ( [\bar \alpha I - \mathcal D h (\bar \theta^N _u)]y - F_N(u,y) \chi_N(y) \right ) \right |^2 q(s,u,x,y) \mathrm d y \mathrm d u.
 \end{eqnarray} 
This claim follows from Corollary 1.2 in  \cite{BRS16}. We have to check the assumptions of the corollary. First, for
both diffusions $X_t$ and $X_t^N$, the drift coefficients have Borel measurable locally bounded entries, i.e., for every ball $U \subset \R^d$, there is a number $B=B(U)$ such that \begin{eqnarray*} \left | \left [\bar \alpha I - \mathcal D h(\bar \theta^N_t)\right ] x \right | \leq B(U) \text{ for all } x \in U, t\in [0,T], \\
 \left |F_N(t,x) \chi_N(x)  \right | \leq B(U) \text{ for all } x \in U, t\in [0,T],
 \end{eqnarray*} 
 and $R^{1/2}(\bar \theta^N_t) $ is locally strictly positive, i.e, for every ball $U \subset \R^d$, there exist $0 <   c_R(U) < C_R(U)$ with $c_R(U)I  \leq  R^{1/2}(\bar \theta^N_t)  \leq C_R(U)I $ 
 for all $x \in U, t\in [0,T]$. Furthermore, we have to check that $R^{-1/2} (\bar \theta^N _u) ([\bar \alpha I - \mathcal D h (\bar \theta^N _u) ]y - F_N(u,y) \chi_N ( y) )$ is square integrable with respect to the measure $q(s,u,x,$ $y)$ $\mathrm d y \ \mathrm d u$ on $\R^d \times [s,t] $ and,  choosing Assumption (a) in Theorem 1.1 in \cite{BRS16}, that $(1 + |y|) ^{-2} |R^{i,j} (\bar \theta^N _u) |$, $(1 + |y|) ^{-1} |F_N(u,y) \chi_N ( y) |$  and $(1 + |y|) ^{-1} |[\bar \alpha I + \mathcal D h (\bar \theta^N _u) ]y - F_N(u,y) \chi_N ( y) |$ are integrable, again with respect to the measure $q(s,u,x,y) \mathrm d y \ \mathrm d u$ on $\R^d \times [s,t] $. All these assumptions can be easily verified using in particular the fast decay of Gaussian densities. 
 
 We now argue that for $u \in [t_k^N, t_{k+1}^N)$, $|y| < a_N$
 \begin{eqnarray} \label{diftruncbound2} &&  \left |  [\bar \alpha I - \mathcal D h (\bar \theta^N _u)]y - F_N(u,y) \right |\\
 && \qquad  \leq C \left  ( \left | \bar \alpha - \alpha_{t_k^N}^N\right | + \sqrt {\gamma_k^N / \gamma_{k+1}^N} -1 +  \left | \bar \theta^N_u - \bar \theta^N_{t_k^N}\right | + a_N \sqrt {\gamma_k^N } \right ). \nonumber
 \end{eqnarray} 
 For a proof of this bound note that for $u \in [t_k^N, t_{k+1}^N)$, $|y| < a_N$
 the function $F_N(u,y)$ can be bounded as follows:
 \begin{eqnarray*} &&  \left |\bar \alpha I - \mathcal D h (\bar \theta^N _u)]y - F_N(u,y) \right |\\
 && \qquad  \leq | \bar \alpha - \alpha_{t_k^N}^N) | + \left |
 \left (\sqrt {\gamma_k^N / \gamma_{k+1}^N} -1\right )
  \int_0^1 \mathcal D h( \bar \theta^N_{t_k^N} + \delta \sqrt {\gamma_{k}^N } y) \mathrm d \delta\right |, \end{eqnarray*} 
 which can be used to show \eqref{diftruncbound2}. Finally, we argue that
 \begin{eqnarray} && \sqrt {\gamma_{k}^N / \gamma_{k+1}^N} -1 \leq C  (\gamma_1^N)^{\beta^{-1}}, \label{inrev1}\\
 && | \bar \alpha - \alpha_{t_{k+1}^N}^N| \leq C I_{\{\beta=1\}} \gamma_1^N + C I_{\{\frac 1 2 <\beta<1\}}  (\gamma_1^N)^{\beta^{-1}-1},\label{inrev2} \\
 && \left | \bar \theta^N_u - \bar \theta^N_{t_k^N}\right |  \leq C \gamma_1^N,\label{inrev3}
  \end{eqnarray} 
 which can be easily shown and implies with \eqref{diftruncbound2} that for $|y| < a_N$
  \begin{eqnarray} \label{diftruncbound3} &&  \left |  [\bar \alpha I - \mathcal D h (\bar \theta^N _u)]y - F_N(u,y) \right |\\
 && \qquad  \leq C \left (  I_{\{\frac 1 2 <\beta<1\}} (\gamma_1^N)^{\beta^{-1}-1} + a_N \sqrt {\gamma_1^N} \right ). \nonumber
 \end{eqnarray} 
Thus it holds that
  \begin{eqnarray} \label{diftruncbound1a} &&  \int_s^t \int_{|y| < a_N} \left | R^{-1/2} (\bar \theta^N _u) \left ( [\bar \alpha I - \mathcal D h (\bar \theta^N _u)]y - F_N(u,y) \chi_N(y) \right ) \right |^2 q(s,u,x,y) \mathrm d y \mathrm d u\\
 && \qquad  \leq C (t-s) \left (  I_{\{\frac 1 2 <\beta<1\}} (\gamma_1^N)^{2\beta^{-1}-2} +  a_N^2 \gamma_1^N\right). \nonumber
 \end{eqnarray} 
 For $|y| \geq a_N$ we have that 
 $$ \left | [\bar \alpha I - \mathcal D h (\bar \theta^N _u)]y - F_N(u,y) \chi_N(y) \right | \leq C |y|.$$
 This implies for  $|x| \leq a_N/2$ that
  \begin{eqnarray} \label{diftruncbound1b} &&  \int_s^t \int_{|y| \geq a_N} \left | R^{-1/2} (\bar \theta^N _u) \left ( [\bar \alpha I - \mathcal D h (\bar \theta^N _u)]y - F_N(u,y) \chi_N(y) \right ) \right |^2 q(s,u,x,y) \mathrm d y \mathrm d u\\
 && \qquad  \leq C    \int_s^t \int_{|y| \geq a_N}  | y |^2 q(s,u,x,y) \mathrm d y \mathrm d u \nonumber
 \\
 && \qquad  \leq C    \int_s^t \int_{|y| \geq a_N}  | y |^2 u ^{-d/2} \exp \left ( -C \frac {|x-\theta_{0,u}^N(y)|^2 } u  \right )\mathrm d y \mathrm d u \nonumber
  \\
 && \qquad  \leq C    \int_s^t \int_{|y| \geq a_N}  | y |^2 u ^{-d/2} \exp \left ( -C \frac {|x-y|^2 } u \right )\mathrm d y \mathrm d u \nonumber
   \\
 && \qquad  \leq C   (t-s)  \int_{|y| \geq C a_N}  | y |^2  \exp \left ( -C  |y|^2   \right )\mathrm d y  \nonumber \\
 && \qquad  \leq C   (t-s) \exp(- C a_N^2)   \nonumber \\
 && \qquad  \leq C  (t-s) (\gamma_1^N)^\rho   \nonumber
 \end{eqnarray} 
 for all $\rho > 0$.
 The second inequality of Proposition \ref{prop:compdiff}  now follows from \eqref{diftruncbound1}, \eqref{diftruncbound1a} and \eqref{diftruncbound1b}. The first inequality follows by using the bound \begin{equation} \label{Hellingerbound1} H^2 (P,Q) \leq \int \left | \frac { \mathrm d P} {\mathrm d \nu}-  \frac { \mathrm d Q} {\mathrm d \nu}\right | \mathrm d \nu \end{equation}  for the Hellinger distance between two measures $P$ and $Q$ with dominating measure $\nu$. The last inequality follows by definition of $\gamma_1^N$. 
 \end{proof}
  \begin{proof} [Proof of Proposition \ref{prop:compdiffcorr}]
  The first  inequality follows by application of the bound $$ \int \left | \frac { \mathrm d P} {\mathrm d \nu}-  \frac { \mathrm d Q} {\mathrm d \nu}\right | \mathrm d \nu  \leq 2 H (P,Q) $$ for  two measures $P$ and $Q$ with dominating measure $\nu$.  The  third  inequality follows as in the proof of Proposition \ref{prop:compdiff} by definition of $\gamma_1^N$. 
  
  For the proof of the second inequality we apply Proposition 2.1 in 
 \cite{CLEM22}. This proposition is stated for one-dimensional homogeneous Markov chains but it can be easily generalized to multi-dimensional non-homogeneous Markov chains. Stated for the measures  $Q_{x}^{m_N}$ and $Q_{N,x}^{m_N}$ the proposition says that 
 \begin{eqnarray*} 
 H^2(Q_{x}^{m_N},Q_{N,x}^{m_N}) \leq  \frac {1} 2 \sum_{j=1} ^{m_N} \left (\mathbb E[H^2_{X_{t_{j-1}},j}]+ \mathbb E[H^2_{X_{t_{j-1}}^N,j}]\right ),
 \end{eqnarray*}
 where for $x\in \R ^d$
  \begin{eqnarray*} 
H^2_{x,j}  = \int \left ( \sqrt{q(t_{j-1}, t_j, x, z)}- \sqrt{q_N(t_{j-1}, t_j, x, z)}\right ) ^2 \mathrm d z.\end{eqnarray*}
By application of \eqref{Hellingerbound1} we get that 
 \begin{eqnarray*} 
 H^2(Q_{x}^{m_N},Q_{N,x}^{m_N}) &\leq & \frac {1} 2 \sum_{j=1} ^{m_N} \left (\mathbb E\left [ \int \left | {q(t_{j-1}, t_j, X_{t_{j-1}}, z)}- {q_N(t_{j-1}, t_j, X_{t_{j-1}}, z)}\right | \mathrm d z\right ]\right . \\
 && \qquad 
 + \left .\mathbb E\left [ \int \left | {q(t_{j-1}, t_j, X_{t_{j-1}}^N, z)}- {q_N(t_{j-1}, t_j, X_{t_{j-1}}^N, z)}\right | \mathrm d z\right ]\right ). \end{eqnarray*}
 By application of Proposition \ref{prop:compdiff} we get that 
  \begin{eqnarray*} 
&& \mathbb E\left [ \int \left | {q(t_{j-1}, t_j, X_{t_{j-1}}, z)}- {q_N(t_{j-1}, t_j, X_{t_{j-1}}, z)}\right | \mathrm d z\right ]\\
 && \qquad \leq C  m_N^{-1/2} \left( I_{\{\frac 1 2 <\beta<1\}} (\gamma_1^N)^{\beta^{-1}-1} +  \ln(1/\gamma_1^N) (\gamma_1^N)^{1/2}\right )
 + \mathbb P [|X_{t_{j-1}}| > a_N/2].\end{eqnarray*}
 The same inequality applies with $X_{t_{j-1}}$ replaced by $X_{t_{j-1}}^N$. Thus for the proof of the proposition it suffices to show that:
 \begin{eqnarray} \label{eq:mnbound1}
&& \mathbb P [|X_{t_{j-1}}| > a_N/2]+ \mathbb P [|X_{t_{j-1}}^N| > a_N/2]  \leq C  m_N^{-1/2}   \ln(1/\gamma_1^N) (\gamma_1^N)^{1/2}.\end{eqnarray}
Now, by application of Gaussian bounds for the transition densities of $X_t$ and $X_t^N$, stated in \cite{dela:meno:10}, we get that 
 \begin{eqnarray*} 
&& \mathbb P [|X_{t_{j-1}}| > a_N/2]+ \mathbb P [|X_{t_{j-1}}^N| > a_N/2]  \leq C  \exp (- C a_N^2).\end{eqnarray*}
Furthermore, note that we can assume that $m_N \leq C N^\mu$ for some $\mu >0$. Otherwise the second inequality of the proposition holds trivially because the Hellinger distance is bounded and the right hand side of the inequality would converge to infinity. But for the case that $m_N \leq C N^\mu$ for some $\mu >0$ one can easily verify that \eqref{eq:mnbound1} holds. This completes the proof of the proposition.
  \end{proof}

\subsection{Proofs of the bounds stated in the lemmas of Subsection \ref{sec:somebounds}} \label {sec:proofsbounds}

\begin{proof}[Proof of Lemma \ref{lem:boundsec31 1} ] {
To see this we argue that
\begin{eqnarray*}&&\|F_N(t_k^N,x) \chi_{N}(x)-F_N(t_k^N,y) \chi_{N}(y) \| \\ \nonumber  && \qquad\leq C \int _0^1 \|\mathcal Dh(\bar \theta^N_{t_k} + \delta \sqrt{\gamma_k^N} \chi_{N}(x)) \chi_{N}(x) - \mathcal Dh(\bar \theta^N_{t_k} + \delta \sqrt{\gamma_k^N} \chi_{N}(y)) \chi_{N}(y)\| \mathrm d \delta .\end{eqnarray*} 
One gets the Lipschitz bound \eqref{eqapp5} by using that the elements of $\mathcal D h$ and its  partial derivatives  are uniformly bounded in a tubular neighbourhood of $\theta_t$ and that furthermore all partial derivatives of $\chi_{N}(x) = (\chi_{N1}(x),...,\chi_{Nd}(x))$ are bounded because of
\begin{eqnarray*} && \frac {\partial} {\partial x_i} \chi_{Nj}(x) = \frac {\partial} {\partial x_i} \left [ x_j {\text{\textbf{I}}}_{\|x\| \leq a_N}(x) + a_N \frac {x_j}{\|x\|} \textbf I_{\|x\| > a_N}(x) \phi_{N}(\|x\|)\right ]\\
&& \qquad = \left\{\begin{array}{cc}\\ 0 & \mbox{if } \|x\| \geq 2a_N, 
\\ \delta_{i,j} & \mbox{if } \|x\| \leq a_N,
\\a_N \left [\left ( \frac  {\delta_{i,j}} {\|x\|} -\frac  {x_i x_j} {\|x\|^3} \right )  \phi_{N}( \|x\|) -
\frac  {x_i x_j} {\|x\|^2} k_N \exp \left ( - \frac 1 {(2 a_N- \|x\| ) (\|x\| -a_N)} \right ) \right ]& \mbox{if } a_N < \|x\| < 2 a_N.\end{array}\right.
\end{eqnarray*}
It remains to show that  \begin{eqnarray} a_N k_N \exp \left ( - \frac 1 {(2 a_N- \|x\| ) (\|x\| -a_N)} \right )  \label{boundrev} \end{eqnarray} is bounded for $a_N < \|x\| < 2 a_N$.

Note that for \(0 < \delta < 1\)
\[k_{N} = \frac{1}{\int_{0}^{a_{N}}{e^{- \frac{1}{v(a_{N} - v)}}\ \mathrm dv}}\  \leq \frac{1}{\int_{0}^{{\delta a}_{N}}{e^{- \frac{1}{v(a_{N} - v)}}\ \mathrm dv}} \leq \frac{1}{\int_{0}^{{\delta a}_{N}}{e^{- \frac{1}{a_{N}v(1 - \delta)}}\ \mathrm dv}}.\]

We now apply  formula 3.471 (2)  in \cite{GR15}:
\begin{eqnarray*} \int_0^u x^{\nu -1} (u-x) ^{\mu-1} \exp(- \beta /x) \mathrm d x = \beta^{(\nu-1)/2} u ^{(2 \mu + \nu -1)/2} \exp(-\beta/(2u)) \Gamma(\mu) W_{(1- 2 \mu-\nu)/2, \nu/2}(\beta /u) \end{eqnarray*}
for $\mu, \beta, u$ with Re$\mu  >0$, Re$\beta  >0$, $u>0$. Here $W_{\lambda,\mu}$ is a Whittaker function with parameter $\lambda,\mu$.

With the choice
\(u = {\delta a}_{N},\ \nu = \mu = 1,\ \beta = \frac{1}{(1 - \delta)a_{N}}\),
we get that
\[k_{N} \leq \frac{e^{\frac{1}{2\delta(1 - \delta)a_{N}^{2}}}}{{\delta a}_{N}W_{- 1,\frac{1}{2}}\ (\frac{1}{\delta(1 - \delta)a_{N}^{2}})}\ \ .\]

We will use the
following properties of Whittaker functions, see formulas 9.232 (1) and 9.222 (4) in \cite{GR15}:
\begin{eqnarray*} W_{\lambda,\mu}(z) &=& W_{\lambda,-\mu}(z), \\ W_{\lambda,\mu+ \frac 1 2 }(z) &=& z^{\mu+1} \exp(z/2) \int_0^\infty (1+t) ^{2\mu} \exp(-zt) \mathrm d t.
\end{eqnarray*}
This gives that
\[k_{N} \leq \frac{e^{\frac{1}{2\delta(1 - \delta)a_{N}^{2}}}}{{\delta a}_{N}W_{- 1,\frac{1}{2} + ( - 1)}\ (\frac{1}{\delta(1 - \delta)a_{N}^{2}})}\  = \frac{e^{\frac{1}{\delta(1 - \delta)a_{N}^{2}}}}{{\delta a}_{N}\int_{0}^{\infty}{\frac{1}{{(1 + t)}^{2}}\ e^{- \frac{t}{\delta(1 - \delta)a_{N}^{2}}}\ dt}}.\]
It follows from this upper bound that
\[{a_{N}k}_{N} \leq \frac{e^{\frac{1}{\delta(1 - \delta)a_{N}^{2}}}}{\delta\int_{0}^{\infty}{\frac{1}{{(1 + t)}^{2}}\ e^{- \frac{t}{\delta(1 - \delta)a_{N}^{2}}}\ dt}}\  \leq C.\]
This inequality can be used to show that \eqref{boundrev} is bounded.} \end{proof}

\begin{proof}[Proof of Lemma \ref{lem:boundsec31 2} ] For a proof of 
\eqref{eq:lem3cent2} with $ |\nu|=0$ note first that $\tilde q_n$ is a Gaussian density with 
$$\tilde q_n(t_i^N,t_j^N,x,y) = \frac {(2 \pi)^{-d/2} }{\sqrt {\det {\bar \sigma(t_i^N,t_j^N)}}} \exp \left ( - \frac 1 2 \left \langle  \bar \sigma^{-1} (t_i^N,t_j^N) (\theta^N_{t_i^N,t_j^N}(y) -x),\theta^N_{t_i^N,t_j^N}(y) -x \right \rangle \right ),
$$
where $\bar \sigma = \int_s ^t R(\bar \theta^N_u) \mathrm d u$.

Furthermore, $\tilde p_n(t_i^N,t_j^N,x,y)$ is the density of $S_{i,j} = \sum_{k=i}^{j-1} \zeta_k$ at the point $\theta_{t_i^N,t_j^N} ^N(y)-x$.
Here $S_{i,j}$ is the sum of the independent random variables 
$$\zeta_k = \sqrt{\gamma_{k+1} ^N} \xi \left ( \bar \theta^N_{t_k^N}+ \chi_N\left ( \theta_{t_i^N,t_j^N} ^N(y) \right )  \sqrt{\gamma_{k} ^N}, \eta_{k+1}\right),$$
where $\xi(\theta, \eta) = H(\theta, \eta) - \mathbb E_ \eta [ H(\theta, \eta)]$. It holds $\mathbb E [\zeta_k] =0$ and Cov$(\zeta_k) = R\left ( \bar \theta^N_{t_k^N}+ \chi_N\left ( \theta_{t_i^N,t_j^N} ^N(y) \right )\right .$ $ \left .\sqrt{\gamma_{k} ^N}\right)$. The covariance matrix $W_{i,j}$ of $S_{i,j}$ is equal to
$$ W_{i,j} =\sum_{k=i}^{j-1} \gamma_{k+1} ^N R\left ( \bar \theta^N_{t_k^N}+ \chi_N\left ( \theta_{t_i^N,t_j^N} ^N(y) \right )  \sqrt{\gamma_{k} ^N}\right).$$
It can be easily checked that $ W_{i,j}-\bar \sigma(t_i^N,t_j^N)$ converges to 0 for $N \to \infty$. For the difference we get the following bound.
\begin{eqnarray*}
&&\left | \bar \sigma(t_i^N,t_j^N) - W_{i,j}\right | \leq \sum_{k=i}^{j-1}\gamma_{k+1} ^N \left | R\left ( \bar \theta^N_{t_k^N}+ \chi_N\left ( \theta_{t_i^N,t_j^N} ^N(y) \right )  \sqrt{\gamma_{k} ^N}\right) - R\left ( \bar \theta^N_{t_k^N}\right)
\right |\\
&& \qquad \qquad + \sum_{k=i} ^{j-1} \int_{t_k^N} ^{t_{k+1}^N} \left | R\left ( \bar \theta^N_{t_k^N}\right)- R\left ( \bar \theta^N_{u}\right)\right | \mathrm d u|\\
&& \qquad \leq C (t_j^N - t_i^N) (a_N \sqrt{\gamma_1 ^N} + \gamma_1 ^N) \\
&& \qquad \leq C (t_j^N - t_i^N) a_N \sqrt{\gamma_1 ^N} .
\end{eqnarray*}
With $f_{i,j}$ equal to the density of the normalized sum $W_{i,j} ^{-1/2} \sum_{k=i} ^{j-1} \zeta_k$ we can write
\begin{eqnarray*}
\tilde p_n(t_i^N,t_j^N,x,y) &=& \det{W_{i,j} ^{-1/2} } f_{i,j} (W_{i,j} ^{-1/2} [ \theta_{t_i^N,t_j^N} ^N(y) -x]),\\
D_x \tilde p_n(t_i^N,t_j^N,x,y) &=&-  \det{W_{i,j} ^{-1/2} } D_xf_{i,j} (W_{i,j} ^{-1/2} [ \theta_{t_i^N,t_j^N} ^N(y) -x]).
\end{eqnarray*}

Under our assumptions $R( \theta)$ is uniformly elliptic in a tubular neighborhood of $\bar \theta^N_t$. Because $a_N \sqrt{\gamma_1^N} \to 0$ we have for $N$ large enough for unit vectors $\theta$ with $\| \theta \| =1$
\begin{eqnarray*}
&&C^{-1} (t_j^N - t_i^N) ^{-1/2} \leq \theta ^T W_{i,j} ^{-1/2} \theta \leq C (t_j^N - t_i^N) ^{-1/2}, \\
&&C^{-1} (t_j^N - t_i^N) ^{-d/2} \leq \det{W_{i,j} ^{-1/2} } \leq C (t_j^N - t_i^N) ^{-d/2}.
\end{eqnarray*}
This implies that for $f_{i,j}$ a local classical limit proposition applies with the following Gaussian density as leading term
$$\frac 1 {\sqrt {\det{W_{i,j}}}} \phi_{0,I} \left (W_{i,j} ^{-1/2}[ \theta_{t_i^N,t_j^N} ^N(y) -x]\right ),$$
see Theorem 19.3 in \cite{bhat:rao:86}.
For the number of summands of $S_{i,j}$ we have the following bound:
$$\frac {t_j^N - t_i^N} {\gamma_1 ^N} \leq j-i \leq \frac {t_j^N - t_i^N} {\gamma_j ^N}.$$
By application of Theorem 19.3 in \cite{bhat:rao:86} we get that
\begin{eqnarray*}
&&\left | \tilde p_n(t_i^N,t_j^N,x,y) - \frac 1 {\sqrt {\det{W_{i,j}}}} \phi_{0,I} \left (W_{i,j} ^{-1/2}[ \theta_{t_i^N,t_j^N} ^N(y) -x]\right ) \right | \\
&& \qquad \leq C \frac  {\sqrt{\gamma_1 ^N} } {\sqrt {t_j^N - t_i^N} } \mathcal Q_{M-d-1} (t_j^N -t_i^N , \theta_{t_i^N,t_j^N} ^N(y) -x).\end{eqnarray*}
Claim
\eqref{eq:lem3cent2}  with $ |\nu|=0$ now follows by noting that  
\begin{eqnarray*}
&&\left | \frac 1 {\sqrt {\det{W_{i,j}}}} \phi_{0,I} \left (W_{i,j} ^{-1/2}[ \theta_{t_i^N,t_j^N} ^N(y) -x]\right )- \tilde q_n(t_i^N,t_j^N,x,y)\right | \\
&& \qquad = \left | \frac 1 {\sqrt {\det{W_{i,j}}}} \phi_{0,I} \left (W_{i,j} ^{-1/2}[ \theta_{t_i^N,t_j^N} ^N(y) -x]\right )\right . 
 \\
&& \qquad \qquad \left . -\frac 1 {\sqrt {\det{\bar \sigma(t_i^N,t_j^N)}}} \phi_{0,I} \left (\bar \sigma(t_i^N,t_j^N) ^{-1/2}[ \theta_{t_i^N,t_j^N} ^N(y) -x]\right )\right |
\\ &&
\qquad \leq C\sqrt{\gamma_1 ^N} (t_j^N - t_i^N) a_N \mathcal Q_{M-d-1} (t_j^N -t_i^N , \theta_{t_i^N,t_j^N} ^N(y) -x).
\end{eqnarray*}

By a slight extension of arguments one can also show \eqref{eq:lem3cent2}  with $1 \leq  |\nu| \leq 4$ . This can be done by copying the proof of Theorem 19.3 in \cite{bhat:rao:86} and noting that the derivation $\frac {\partial} {\partial x_k}$ of the density corresponds to the multiplication of its characteristic function by $-it_k$. The proof of  \eqref{eq:lem3cent2add} for $\varphi = \tilde p_N$   follows directly by using that $\tilde p_N$ is a Gaussian density. By using \eqref{eq:lem3cent2}  one gets that \eqref{eq:lem3cent2add} with $\varphi = \tilde p_N$ implies 
 \eqref{eq:lem3cent2add}
 for $\varphi = \tilde q_N$.
 \end{proof}
 
 \begin{proof}[Proof of Lemma \ref{em2:boundstheta}] The proof of \eqref{eq:lem3cent5a} follows by \eqref{eqapp5} and  by application of Gronwall's inequaltiy, compare also Lemma 5.3 in \cite{dela:meno:10}. Note that by application of  \eqref{eq:defthetaN} and  \eqref{eqapp5}  we get for $g(u) = 
 \|\theta^N_{u,s}(x) -  \theta^N_{u,t}(y)\| $ that 
 \begin{eqnarray*} 
g(u)&=& \|\int_s^u \left ( \frac {\partial} {\partial v }\theta^N_{v,t}(y) - \frac {\partial} {\partial v } \theta^N_{v,s}(x) \right ) \mathrm d v +  \theta^N_{s,t}(y)- \theta^N_{s,s}(x)\|\\
&\leq& \|\theta^N_{s,t}(y)- \theta^N_{s,s}(x)\|+
 C\int_s^u \|\theta^N_{v,t}(y) -  \theta^N_{v,s}(x)\| \mathrm d v\\
 &=& \|\theta^N_{s,t}(y)-x\|+
 C\int_s^ug(v) \mathrm d v,
 \end{eqnarray*} 
 which gives for some $C>0$ that
 $$g(u) \leq  \|\theta^N_{s,t}(y)-x\| e^{CT}.$$
 Application of the last inequality with $u=t$ and repetition of the argument with the roles of $s,x$ and $t,y$ interchanged gives \eqref{eq:lem3cent5a}. 
 By putting $x=0$ in \eqref{eq:lem3cent5a}  we get \eqref{eq:lem3cent5b}. Claim \eqref{eq:boundstheta1}  follows by similar arguments as in the proof of  \eqref{eq:lem3cent5a}. \end{proof}
\begin{proof}[Proof of Lemma \ref{lem4:boundsH}  ]
With $\lambda$ large enough we define the diffusions 
\begin{eqnarray}  \label{eq:majdiff1}
\mathrm d \bar X_t &=& \left [ \bar \alpha I - \mathcal Dh(\bar \theta^N_t) \right ] \bar X_t \mathrm d t + \lambda I \mathrm d B_t \text{ with } \bar X_0 = x,\\
\label{eq:majdiff2}
\mathrm d \bar X^N_t &=&  F_N(t, \bar X^N_t)\chi_{N} (\bar X^N_t) \mathrm d t + \lambda I \mathrm d B_t \text{ with } \bar X^N_0 = x
\end{eqnarray}
and denote their transition densities by $\bar p(t,s,x,y)$ and $\bar q_N(t,s,x,y)$, respectively. We will show that
\begin{eqnarray} \label{eq:boundsH1A} &&\left | H^{(r)} (t,v,x,y)\right| \leq C^r \frac {\Gamma^r(1/2)}{\Gamma(r/2)} (v-t) ^{\frac r 2 -1} \bar p(t,v,x,y), \\
\label{eq:boundsH2A} &&\left | H_N^{(r)} (t,v,x,y)\right| \leq C^r \frac {\Gamma^r(1/2)}{\Gamma(r/2)} (v-t) ^{\frac r 2 -1}\bar q_N (t,v,x,y).
\end{eqnarray}
These claims imply the lemma because, by Theorem 1.1 in \cite{dela:meno:10} we get with a constant $C>0$:
\begin{eqnarray}  \label{eq:boundmajdiff1}
\bar p(t,s,x,y) &\leq &C (s-t) ^{-d/2} \exp \left ( - \frac {(x - \theta_{t,s}(y))^2}{C|s-t|}\right ),\\
 \label{eq:boundmajdiff2}
\bar q_N(t,s,x,y) &\leq &C (s-t) ^{-d/2} \exp \left ( - \frac {(x - \theta_{t,s}^N(y))^2}{C|s-t|}\right ).\end{eqnarray}
For the proof of \eqref{eq:boundsH1A}--\eqref{eq:boundsH2A} we now apply Proposition 2.3 and Corollary 2.4 in \cite{bitt:kona:21}. From Corollary 2.4 we get that, for $\lambda > 0$ large enough, the transition densities  of $X_t$ and $X_t^N$ are bounded from above up to a constant factor by the transition densities  of $\bar X_t$ or $\bar X_t^N$, respectively. Hence convolutions of the densities of $X_t$ and $X_t^N$ are bounded  by the convolutions of the densities of the majorizing diffusions $\bar X_t$ and $\bar X_t^N$ multiplied by a constant to the power equal to the multiplicity of convolution. Calculation of the $\otimes$ and $\otimes_N$ convolutions gives the additional factor $1/{\Gamma(r/2)}$, compare e.g. for the operator $\otimes_N$
the calculation of the bound for the term $\big | II \big |$ after
\eqref{eq:lem3cent7}. 
 
 \end{proof}

\newpage


\section{Supplement of \\
"Local Limit Theorems and 
Strong Approximations for Robbins-Monro Procedures" \\ by Valentin Konakov, Enno Mammen and  Lorick Huang
} \label {sec:proofscompRobMon}

This supplement contains the proofs of Lemmas \ref{lem1:comptrunc}, \ref{lem2:comptrunc}, and \ref{lem4:comptrunc} from Subsection \ref{subsec:compRobMon}.

\begin{proof}[Proof of Lemma \ref{lem1:comptrunc}]
With $q_N^{discr}(t_{i}^ N, t_{j}^ N,x,y) = \sum_{r=0}^\infty \tilde q_N\otimes_N H_N^{(r)}(t_{i}^ N, t_{j}^ N,x,y)$ we have to show that 
 \begin{eqnarray} \label{eq:claim1}\left |(q_N- q_N^{discr})(t_{i}^ N, t_{j}^ N,x,y) \right | \leq C \ln^2(1 / \gamma_1^N) \sqrt {\gamma_1^N} \sqrt { t_{j}^ N- t_{i}^ N} \bar q_N(t_{i}^ N, t_{j}^ N,x,y). \end{eqnarray} 
 For a proof of this claim we write 
  \begin{eqnarray*}  &&(\tilde q_N \otimes H_N^{(r)}-\tilde q_N \otimes_N H_N^{(r)})(t_{i}^ N, t_{j}^ N,x,y)\\
&& \qquad =
\left [\left(\tilde q_N \otimes H_N^{(r-1)}\right) \otimes H_N -\left(\tilde q_N \otimes H_N^{(r-1)}\right) \otimes_N H_N \right](t_{i}^ N, t_{j}^ N,x,y)\\
  && \qquad \qquad + \left [\left(\tilde q_N \otimes H_N^{(r-1)}\right)  -\left(\tilde q_N \otimes_N H_N^{(r-1)}\right)  \right]\otimes_N H_N(t_{i}^ N, t_{j}^ N,x,y).   \end{eqnarray*} 
  By summing this up from $r=1$ to $r=\infty$ we get by using linearity of the operations $ \otimes $ and $ \otimes _N$
   \begin{eqnarray*}  &&\left ( q_N- q_N^{discr}\right )(t_{i}^ N, t_{j}^ N,x,y)  = \left ( q_N\otimes H_N- q_N\otimes_N H_N\right )(t_{i}^ N, t_{j}^ N,x,y) \\
&& \qquad +
\left ( q_N- q_N^{discr}\right )\otimes _N H_N(t_{i}^ N, t_{j}^ N,x,y) .   \end{eqnarray*} 
Iterative application of this equation gives
\begin{eqnarray} \nonumber &&\left ( q_N- q_N^{discr}\right )(t_{i}^ N, t_{j}^ N,x,y)  = \left ( q_N\otimes H_N- q_N\otimes_N H_N\right )(t_{i}^ N, t_{j}^ N,x,y) \\
&& \label{eq:zeroout}\qquad +
\left ( q_N\otimes H_N- q_N\otimes_N H_N\right )\otimes _N \Phi_N^{discr}(t_{i}^ N, t_{j}^ N,x,y) ,   \end{eqnarray} 
where $\Phi_N^{discr}(t_{i}^ N, t_{j}^ N,x,y)  = \sum_{k\geq 1} H_N^{(k),\otimes_N}(t_{i}^ N, t_{j}^ N,x,y)$. Here $H_N^{(k),\otimes_N}$ is similarly defined as $H_N^{(k)}$ as a $k$-times convolution of $H_N$ but now with using the convolution operator $\otimes_N$ instead of $\otimes$. Furthermore, instead of $H_N$ we use the convolution of $H_N(t_{i}^ N, t_{j}^ N,x,y) \mathbb I_{t_{j}^ N > t_{i}^ N}$. For the latter change note that $\Phi_N^{discr}(t_{i}^ N, t_{j}^ N,x,y)$ with $t_{i}^ N=t_{j}^ N$ is not used in \eqref{eq:zeroout} because $\Phi_N^{discr}(t_{i}^ N, t_{j}^ N,x,y)$ only appears on the right hand side of  the convolution operator $\otimes_N$. Note  that for this reason it holds  for functions $f$ that $f\otimes _N \Phi_N^{discr}(t_{i}^ N, t_{j}^ N,x,y) = ((f \otimes_N H_N) + (f\otimes_N H_N)\otimes_N H_N + ((f\otimes_N H_N) \otimes_N H_N)\otimes_N H_N+ ...)(t_{i}^ N, t_{j}^ N,x,y) $.

For the statement of the lemma it suffices to show that for some constant $C>0$
\begin{eqnarray}  \label{eq:claim1a} &&\left | \left ( q_N\otimes H_N- q_N\otimes_N H_N\right )(t_{i}^ N, t_{j}^ N,x,y) \right|\\
\nonumber && \qquad \leq C \ln^2(1 / \gamma_1^N) \sqrt {\gamma_1^N} \sqrt { t_{j}^ N- t_{i}^ N} (1 + |y|)\bar q_N(t_{i}^ N, t_{j}^ N,x,y), \\
 \label{eq:claim1b} && \left |\left ( q_N\otimes H_N- q_N\otimes_N H_N\right )\otimes _N \Phi_N^{discr}(t_{i}^ N, t_{j}^ N,x,y)\right |  \\
\nonumber && \qquad \leq 
C \ln^2(1 / \gamma_1^N) \sqrt {\gamma_1^N} (t_{j}^ N- t_{i}^ N) (1 + |y|)\bar q_N(t_{i}^ N, t_{j}^ N,x,y) .  
\end{eqnarray} 
Before we come to the proof of these claims we first argue that the following bound for $\Phi_N^{discr}$ applies
\begin{eqnarray}  \label{eq:claim1c}| \Phi_N^{discr}(t_{i}^ N, t_{j}^ N,x,y)|&\leq& C (t_{j}^ N - t_{i}^ N)^{-1/2} \bar q _N(t_{i}^ N, t_{j}^ N,x,y).
\end{eqnarray} 
This bound follows from
\begin{eqnarray}  \label{eq:claim1chelp}| H_N^{(r),\otimes_N}(t_{i}^ N, t_{j}^ N,x,y)|&\leq& C^r \frac {\Gamma^r (1/2)} {\Gamma (r/2)} (t_{j}^ N - t_{i}^ N)^{\frac r 2-1} \bar q _N(t_{i}^ N, t_{j}^ N,x,y).
\end{eqnarray} 

For a proof of \eqref{eq:claim1chelp} we remark that for the function $f(u) = (u-t_i^N)^{\frac r 2 - 1} (t_j^N - u) ^{-1/2}$ for $r \geq 2$ the following estimate holds because $f$ is monotonically increasing on $[t_{i}^ N, t_{j}^ N]$
\begin{eqnarray*}  \sum_{k=i}^{j-1} \gamma_{k+1} ^N f(t_k^N) &\leq& \int _{t_i^N}^{t_j^N} f(u) \mathrm d u.
\end{eqnarray*} 
Using this bound we can carry over the arguments of the proof of \eqref{eq:boundsH2} to get \eqref{eq:claim1chelp}. 

We now show that \eqref{eq:claim1a} and \eqref{eq:claim1c} imply \eqref{eq:claim1b}. By application of the first two inequalities we get
\begin{eqnarray*} && \left |\left ( q_N\otimes H_N- q_N\otimes_N H_N\right )\otimes _N \Phi_N^{discr}(t_{i}^ N, t_{j}^ N,x,y)\right | \\
&& \qquad \leq \sum_{k=i}^{j-1} \gamma_{k+1}^N \int_{\R^d}  \left |\left ( q_N\otimes H_N- q_N\otimes_N H_N\right )(t_{i}^ N, t_{k}^ N,x,z)\right |  \ \left |
\Phi_N^{discr}(t_{k}^ N, t_{j}^ N,z,y)\right | \mathrm d z
\\
&& \qquad \leq C \ln^2(1/ \gamma_{1}^N) \sqrt{\gamma_{1}^N} \sum_{k=i}^{j-1} \gamma_{k+1}^N \frac {\sqrt{t_k^N - t_i^N}} {\sqrt{t_j^N - t_k^N}} \int_{\R^d} (1 + |z|) \bar q_N(t_{i}^ N, t_{k}^ N,x,z)
\bar q_N(t_{k}^ N, t_{j}^ N,z,y) \mathrm d z\\
&& \qquad \leq C \ln^2(1/ \gamma_{1}^N) \sqrt{\gamma_{1}^N} \sum_{k=i}^{j-1} \gamma_{k+1}^N \frac {\sqrt{t_k^N - t_i^N}} {\sqrt{t_j^N - t_k^N}} \int_{\R^d} \left (1 + |y|  \right ) \bar q_N(t_{i}^ N, t_{k}^ N,x,z)
\bar q_N(t_{k}^ N, t_{j}^ N,z,y) \mathrm d z  \\ 
&& \qquad \leq C \ln^2(1/ \gamma_{1}^N) \sqrt{\gamma_{1}^N} ( t_j^N - t_i^N)  \left (1 + |y|  \right ) \bar q_N(t_{i}^ N, t_{j}^ N,x,y) \int_0^1 \frac { \sqrt {u}}{ \sqrt {1-u}}\mathrm d u
\\
&& \qquad \leq C \ln^2(1/ \gamma_{1}^N) \sqrt{\gamma_{1}^N} ( t_j^N - t_i^N)  \left (1 + |y|  \right ) \bar q_N(t_{i}^ N, t_{j}^ N,x,y).
\end{eqnarray*} 
 This shows \eqref{eq:claim1b}. For the proof of the lemma it remains to show \eqref{eq:claim1a}. For a proof of this claim we introduce
 $$\lambda_u(z) = q_N(t_i^N,u,x,z) H_N(u,t_j^N,z,y).$$
Using a Taylor expansion of second order we get
\begin{eqnarray}&& \nonumber (q_N\otimes H_N - q_N \otimes_N H_N)(t_i^N,t_j^N,x,y) = \sum_{k=i}^{j-1} \int_{t_k^N}^{t_{k+1}^N} \mathrm d u \int_{\R^d} \left(\lambda_u(z)- \lambda_{t_k^N}(z)\right) \mathrm d z\\ \label{eqboundHinLemma42}
&& \qquad = \sum_{k=i}^{j-1} \int_{t_k^N}^{t_{k+1}^N} (u-t_k^N) \int_0^1\int_{\R^d} \frac {\partial}{\partial s}\lambda_s(z)\bigg |_{s=s_k} \mathrm d z\ \mathrm d \delta \ \mathrm d u 
\end{eqnarray}
with $s_k = t_k^N + \delta (u-t_k^N)$. Next by application of forward and backward Kolmogorov equations we get that
\begin{eqnarray*}&&\int_{\R^d} \frac {\partial}{\partial s}\lambda_s(z)\bigg |_{s=s_k} \mathrm d z = \int_{\R^d} \frac {\partial}{ \partial s} \left( q_N(t_i^N,s,x,z) H_N(s,t_j^N,z,y)\right)  \bigg |_{s=s_k} \mathrm d z
\\
&& \qquad = \int_{\R^d} \left ((L_s^N)^* q_N(t_i^N,s,x,z) (L_s^N- \tilde L_s^N) \tilde q_N(s,t_j^N,z,y) \right )  \bigg |_{s=s_k} \mathrm d z
\\
&& \qquad \qquad - \int_{\R^d} \left (q_N(t_i^N,s,x,z) (L_s^N- \tilde L_s^N)  \tilde L_s^N \tilde q_N(s,t_j^N,z,y) \right )  \bigg |_{s=s_k} \mathrm d z
\\
&& \qquad \qquad + \int_{\R^d} q_N(t_i^N,s,x,z) \sum_{h=1} ^d \bigg [ \frac {\partial}{\partial s}
\sum_{l=1} ^d  \bigg([F_N(s,z)]_{h,l} [\chi_N(z)]_l\\
&&\qquad \qquad \qquad - [F_N(s,\theta^N_{s,t_j^N}(y))]_{h,l}[ \chi_N(\theta^N_{s,t_j^N}(y)]_l
\bigg )  \bigg |_{s=s_k} \frac {\partial \tilde q _N(s_k,t_j^N,z,y)} {\partial z_h} \bigg ] \mathrm d z \\
&& \qquad = I + II
\end{eqnarray*}
with 
\begin{eqnarray*}
I&=& \int_{\R^d}  q_N(t_i^N,s,x,z) \big((L_s^N)^2-2 L_s^N \tilde L_s^N+ ( \tilde L_s^N)^2 \big )  \tilde q_N(s,t_j^N,z,y) \bigg |_{s=s_k} \mathrm d z, \\
II&=& \int_{\R^d} q_N(t_i^N,s_k,x,z) \sum_{r=1} ^d \bigg [ \frac {\partial}{\partial s}
\sum_{l=1} ^d  \bigg([F_N(s,z)]_{r,l} [\chi_N(z)]_l\\
&&\qquad \qquad \qquad - [F_N(s,\theta^N_{s,t_j^N}(y))]_{r,l}[ \chi_N(\theta^N_{s,t_j^N}(y)]_l
\bigg )  \bigg |_{s=s_k} \frac {\partial \tilde q _N(s_k,t_j^N,z,y)} {\partial z_r} \bigg ] \mathrm d z.
\end{eqnarray*}
We will show below that 
\begin{eqnarray}
\label{eq:claimboindI}
|I|&\leq &C \frac {\ln^2(1 /\gamma_1^k )}{\sqrt {s_k -t_i^N} \sqrt {t_j^N -s_k}}  \bar q_N(t_i^N,t_j^N,x,y), \\
\label{eq:claimboindII}
|II|&\leq &C \frac {\ln(1 /\gamma_1^k )}{\sqrt {t_j^N -s_k}}  \bar q_N(t_i^N,t_j^N,x,y).
\end{eqnarray}
Before we will come to the proof of \eqref{eq:claimboindI} and \eqref{eq:claimboindII} we now show that these inequalities can be used to prove the lemma. Note first that from \eqref{eq:claimboindI}, \eqref{eq:claimboindII} we get that 
$$\left | \int_{\R^d} \frac {\partial}{\partial s}\lambda_s(z)\bigg |_{s=s_k} \mathrm d z\right | \leq C \frac {\ln^2(1 /\gamma_1^k )}{\sqrt {s_k -t_i^N} \sqrt {t_j^N -s_k}}  \bar q_N(t_i^N,t_j^N,x,y).$$
We now use this bound to estimate the summand on the right hand side of 
\eqref{eqboundHinLemma42}. For the statement of the lemma we have to show that for some constant $C> 0$
\begin{eqnarray}
\label{eq:ineqlem42}
 \sum_{k=i}^{j-1} \int_{t_k^N}^{t_{k+1}^N}  \int_0^1  \frac {(u-t_k^N)}{\sqrt {s_k -t_i^N} \sqrt {t_j^N -s_k}}  \mathrm d \delta \ \mathrm d u \leq C \sqrt {\gamma_1^N} \sqrt{t_j^N -t_i^N}.
\end{eqnarray}
Note that the left hand side of \eqref{eq:ineqlem42} is equal to
$$ \sum_{k=i}^{j-1} \int_0^1    \frac { (t_{k+1}^N-t_{k}^N)^2}{\sqrt {t_{k}^N -t_i^N+v(t_{k+1}^N-t_{k}^N)} \sqrt {t_j^N -t_{k}^N-v(t_{k+1}^N-t_{k}^N)}}  \mathrm d v.$$
We now consider the summands of this sum for $i+1 \leq k \leq j-2$, $k=i$ and $k=j-1$.
We have 
\begin{eqnarray*}
&&\sum_{k=i+1}^{j-2} \int_0^1    \frac { (t_{k+1}^N-t_{k}^N)^2}{\sqrt {t_{k}^N -t_i^N+v(t_{k+1}^N-t_{k}^N)} \sqrt {t_j^N -t_{k}^N-v(t_{k+1}^N-t_{k}^N)}}  \mathrm d v\\
&& \qquad \leq \sum_{k=i+1}^{j-2}  \frac { (t_{k+1}^N-t_{k}^N)^2}{\sqrt {t_{k}^N -t_i^N} \sqrt {t_j^N -t_{k+1}^N}} \\
&& \qquad \leq C  \sqrt {\gamma_1^N} \sum_{k=i+1}^{j-2}  \frac { t_{k+1}^N-t_{k}^N}{\sqrt {t_{k}^N -t_i^N}} \\
&& \qquad \leq C  \sqrt {\gamma_1^N} \int_{t_{i}^N}^{t_{j}^N}  \frac {1}{\sqrt {v -t_i^N}} \mathrm d v
\\
&& \qquad \leq C  \sqrt {\gamma_1^N} \sqrt {t_j^N -t_i^N}.
\end{eqnarray*}
For $k=i<j-1$ we have that 
\begin{eqnarray*}
&& \int_0^1    \frac { (t_{k+1}^N-t_{k}^N)^2}{\sqrt {t_{k}^N -t_i^N+v(t_{k+1}^N-t_{k}^N)} \sqrt {t_j^N -t_{k}^N-v(t_{k+1}^N-t_{k}^N)}}  \mathrm d v\\
&& \qquad \leq \int_0^1  \frac {1}{\sqrt {v}}  \mathrm d v  \frac {(t_{i+1}^N-t_{i}^N)^{3/2} } {\sqrt{t_{j}^N-t_{i+1}^N} }
\\
&& \qquad \leq C  \sqrt {\gamma_1^N} \sqrt {t_j^N -t_i^N}.
\end{eqnarray*}
For $k=j-1 > i$ we have that 
\begin{eqnarray*}
&& \int_0^1    \frac { (t_{k+1}^N-t_{k}^N)^2}{\sqrt {t_{k}^N -t_i^N+v(t_{k+1}^N-t_{k}^N)} \sqrt {t_j^N -t_{k}^N-v(t_{k+1}^N-t_{k}^N)}}  \mathrm d v\\
&& \qquad \leq \int_0^1  \frac {1}{\sqrt {v}}  \mathrm d v  \frac {(t_{j}^N-t_{j-1}^N)^{3/2} } {\sqrt{t_{j-1}^N-t_{i}^N} }
\\
&& \qquad \leq C  \sqrt {\gamma_1^N} \sqrt {t_j^N -t_i^N}.
\end{eqnarray*} Finally, we have for $k=j-1 = i$  that 
\begin{eqnarray*}
&& \int_0^1    \frac { (t_{k+1}^N-t_{k}^N)^2}{\sqrt {t_{k}^N -t_i^N+v(t_{k+1}^N-t_{k}^N)} \sqrt {t_j^N -t_{k}^N-v(t_{k+1}^N-t_{k}^N)}}  \mathrm d v\\
&& \qquad \leq \int_0^1  \frac {1}{\sqrt {v(1-v)}}  \mathrm d v  (t_{j}^N-t_{i}^N)
\\
&& \qquad \leq C  \sqrt {\gamma_1^N} \sqrt {t_j^N -t_i^N}.
\end{eqnarray*} 
This concludes the proof of \eqref{eq:ineqlem42}. For the statement of the lemma it remains to show \eqref{eq:claimboindI} and \eqref{eq:claimboindII}. 

For a proof of \eqref{eq:claimboindI} note first that terms with derivatives  of second order in $L_s^N$ and $\tilde L_s^N$  are equivalent. Hence the terms with derivatives of fourth order in $(L_s^N)^2-2 L_s^N \tilde L_s^N+ ( \tilde L_s^N)^2 $  cancels. Thus the operator $(L_s^N)^2-2 L_s^N \tilde L_s^N+ ( \tilde L_s^N)^2 $ contains only derivatives up to second order. Because of the truncation we have that the coefficients of the operator are bounded by $C \ln^2(1/\gamma_1^N)$. The integrand of the integral in $I$ contains terms with first and second order derivatives of $\tilde q_N(s_k,t_i^N,t_j^N,z,y) $ with respect to $z$. Terms with first order derivatives contain only integrable singularities. For terms with second order derivatives we make use of partial integration and get products of first order derivatives of $\tilde q_N(s_k,t_i^N,t_j^N,z,y) $ and of $q_N(t_i^N,s_k,y,z) $ with respect to $z$. We now use that the norm of the derivative of $q_N(t_i^N,s_k,y,z) $ with respect to $z$ can be bounded by $C (s_k-t_i^N)^{-1/2} q_N(t_i^N,s_k,y,z) $, see (3.12) in   \cite{menpeszha21}. Using these arguments we get the bound 
\eqref{eq:claimboindI}.

For the proof of \eqref{eq:claimboindII} we define
$v=(v_1,...,v_d)^t= \bar \theta^N_s + \delta \sqrt {\gamma_k^N} \chi_N(x)$ and we get for fixed $1 \leq i,l \leq d$
\begin{eqnarray} \nonumber
\frac {\partial} {\partial s} [F_N(s,z)] _ {i,l} &= &\sqrt{\frac {\gamma_{k+ \mathbf{1}_{\|x\|\geq a_N}}^N} {\gamma_{k+1}^N}}\int_0^1 \frac {\partial} {\partial s} \frac {\partial h_i} {\partial v_l}(v) \mathrm d \delta\\ \label{eq:boundFNsz1}
&=&\sqrt{\frac {\gamma_{k+ \mathbf{1}_{\|x\|\geq a_N}}^N} {\gamma_{k+1}^N}}\int_0^1 \sum_{r=1}^d \frac {\partial^2 h_i} {\partial v_l\ \partial v_r}(v) h_n(\bar \theta^N_s)\mathrm d \delta.
\end{eqnarray}
Furthermore, with $w=(w_1,...,w_d)^t=  z_s + \delta \sqrt {\gamma_k^N}  \chi_N(z_s)$ and $z_s= \theta_{s,t_j^N}^N(y)$ we get that
\begin{eqnarray} \nonumber
\frac {\partial} {\partial s} [F_N(s,z_s)] _ {i,l}  [\chi_N(z_s) ]_l&= &\sqrt{\frac {\gamma_{k+ \mathbf{1}_{\|x\|\geq a_N}}^N} {\gamma_{k+1}^N}}\int_0^1 \frac {\partial} {\partial s} \frac {\partial h_i} {\partial w_l}(w) \mathrm d \delta [\chi_N(z_s) ]_l \\
\nonumber && \qquad + [F_N(s,z_s)] _ {i,l} \sum_{r,q=1}^d \frac
{ \partial [\chi_N(z_s) ]_l}{\partial w_r} [F_N(s,z_s)] _{rq}  [\chi_N(z_s) ]_q
\\ \label{eq:boundFNsz2}
&=&\sqrt{\frac {\gamma_{k+ \mathbf{1}_{\|x\|\geq a_N}}^N} {\gamma_{k+1}^N}}\int_0^1 \sum_{r=1}^d \frac {\partial^2 h_i} {\partial w_l\ \partial w_r}(w) \frac {\partial w_r} {\partial s}\mathrm d \delta \\
\nonumber && \qquad + [F_N(s,z_s)] _ {i,l} \sum_{r,q=1}^d \frac
{ \partial [\chi_N(z_s) ]_l}{\partial w_r} [F_N(s,z_s)] _{rq}  [\chi_N(z_s) ]_q.
\end{eqnarray}
Now, we have that 
\begin{eqnarray*}\frac {\partial w_r} {\partial s} &=& \frac {\partial (z_{s,r} +  \delta \sqrt {\gamma_k^N} [ \chi_N(z_s)]_r)} {\partial s} \\ &=&  \sum_{q=1} ^d  [F_N(s,z_s)] _{rq}  [\chi_N(z_s) ]_q+  \delta \sqrt {\gamma_k^N} \sum_{p,q=1} ^d \frac { [ \chi_N(z_s)]_r} {\partial z_{s,p}} [F_N(s,z_s)] _{pq}  [\chi_N(z_s) ]_q. \end{eqnarray*} This shows that
$\| \frac {\partial w_r} {\partial s}\| \leq C a_N$. Using this bound we get from \eqref{eq:boundFNsz2} that 
$$\left \|\frac {\partial} {\partial s} [F_N(s,z_s)] _ {i,l}  [\chi_N(z_s) ]_l \right  \| \leq C a_N.$$
Furthermore from \eqref{eq:boundFNsz1} we have that $$\left \|\frac {\partial} {\partial s} [F_N(s,z)] _ {i,l} \right  \| \leq C .$$
From the last two inequalities we conclude that 
\begin{eqnarray*}
&&\bigg |  \frac {\partial}{\partial s}
  \bigg([F_N(s,z)]_{r,l} [\chi_N(z)]_l - [F_N(s,z_s]_{r,l}[ \chi_N(z_s]_l
\bigg )  \bigg |_{s=s_k}\bigg | \ \bigg | \frac {\partial \tilde q _N(s_k,t_j^N,z,y)} {\partial z_r} \bigg | \\
&&\qquad \leq C  \bigg | \frac {\partial \tilde q _N(s_k,t_j^N,z,y)} {\partial z_r} \bigg |  \ \bigg |  \frac {\partial}{\partial s}
  [F_N(s,z)]_{r,l}  \bigg |_{s=s_k} | z-z_s| + a_N\bigg |
  \\
&&\qquad \leq  \frac C {\sqrt {t_j^N -s_k}} (1+a_N)  \frac {\partial \tilde q _N(s_k,t_j^N,z,y)} {\partial z_r}.
\end{eqnarray*}
This bound can be used to show \eqref{eq:claimboindII}.
\end{proof}

\begin{proof}[Proof of Lemma \ref{lem2:comptrunc}] For the proof of the lemma we use the bound \begin{eqnarray*}
&&\left | \tilde q_N \otimes _N H_N^{(r)} (t_i^N,t_j^N,x,y) \right | \leq C^r \frac {\Gamma^r \left (\frac 1 2\right) } {\Gamma \left (\frac r 2\right) } \gamma_{k+1} ^N (t_j^N - t_i^N)^{\frac r 2 -1} \int _{\mathbb R^d} \bar q_N(t_i^N,t_k^N,x,z)\\
&& \qquad \qquad \times \bar q_N(t_k^N,t_j^N,z,y) \mathrm d z \\
&&\qquad \leq   \frac {C^r } {\Gamma \left (1+\frac r 2\right) } \bar q_N(t_i^N,t_j^N,x,y), 
\end{eqnarray*}
where the constant $C$ does not depend on $r$.With another constant $C'$ it holds that
$$\sum_{r=N+1}^\infty \frac {C^r } {\Gamma \left (1+\frac r 2\right) }  \leq C' e ^{-C'N}.$$
This immediately implies the statement of the lemma.
\end{proof}

\begin{proof}[Proof of Lemma \ref{lem4:comptrunc}] The proof of this lemma heavily depends on the bound 
\begin{eqnarray} \label{eq:bound MN}
&&\left | M_N(t_i^N,t_j^N, x, y) \right | \leq C \frac { \sqrt{\gamma_{i+1}^N}}{t_j^N-t_i^N} 
 \mathcal Q_{M-d-6} (t_j ^N-t_i ^N, \theta^N_{t_i ^N,t_j ^N}(y)-x ),
\end{eqnarray}
for $j> i$
where 
 \begin{eqnarray*} &&M_N(t_i^N,t_j^N, x, y)= \mathcal K_N(t_i^N,t_j^N, x, y) - K_N(t_i^N,t_j^N, x, y).\end{eqnarray*} 
 The proof of this bound can be found in Subsection \ref{sec:bound MN}. For a proof of the lemma note first that 
 \begin{eqnarray*} 
&&\left | ( \tilde p_N \otimes_N \mathcal K_N- \tilde p_N \otimes_N  K_N)(t_i^N,t_j^N, x, y) \right | = \left | ( \tilde p_N \otimes_N M_N)(t_i^N,t_j^N, x, y) \right | \\
&& \qquad \leq  \sum_{k=i}^{j-1} \gamma_{k+1}^N \int_{\R^d} | \tilde p_N (t_i^N, t_k ^N, x,z) |  |M_N (t_k ^N, t_j ^N, z,y) |\mathrm d z\\
&& \qquad \leq C \sum_{k=i}^{j-1} \frac {\sqrt{\gamma_{k+1}^N}}{t_j^N - t_k^N} \gamma_{k+1}^N  \int_{\R^d} \mathcal Q_{M-d-6} (t_k ^N-t_i^N, z-\theta^N_{t_k ^N,t_i^N}(x) )
 \mathcal Q_{M-d-6} (t_j ^N-t_k ^N, y-\theta^N_{t_j ^N,t_k ^N}(z) )\mathrm d z\\
&& \qquad \leq  C \sum_{k=i}^{j-1} \frac {\sqrt{\gamma_{k+1}^N}}{t_j^N - t_k^N} \gamma_{k+1}^N \mathcal Q_{M-d-6} (t_j ^N-t_i^N, y-\theta^N_{t_j ^N,t_i^N}(x) )\\
&& \qquad \leq  C \sqrt{\gamma_1^N} \left ( \int_{\gamma_j^N}^{t_j^N -t_i^N} \frac {\mathrm du}{u} +1 \right ) \mathcal Q_{M-d-6} (t_j ^N-t_i^N, y-\theta^N_{t_j ^N,t_i^N}(x) )\\
&& \qquad \leq  C \sqrt{\gamma_1^N} \ln{\left ( \frac {e (t_j^N-t_i^N)}{\gamma_j^N}\right) } \mathcal Q_{M-d-6} (t_j^N-t_i^N, y-\theta^N_{t_j ^N,t_i^N}(x) )\\
&& \qquad \leq  C \sqrt{\gamma_1^N} \ln{\left ( \frac {C (t_j^N-t_i^N)}{\gamma_1^N}\right) } \mathcal Q_{M-d-6} (t_j ^N-t_i^N, y-\theta^N_{t_j ^N,t_i^N}(x) ).
\end{eqnarray*}
For $r>1$ we use the following recursion argument
 \begin{eqnarray}  \label{eq:iterMN}
&&\left | ( \tilde p_N \otimes_N \mathcal K_N^{(r+1)}- \tilde p_N \otimes_N  K_N^{(r+1)})(t_i^N,t_j^N, x, y) \right | \\
&& \leq \left | ( \tilde p_N \otimes_N \mathcal K_N^{(r)}- \tilde p_N \otimes_N  K_N^{(r)})\right | \otimes_N    |\mathcal K_N| 
(t_i^N,t_j^N, x, y) + \left | ( \tilde p_N \otimes_N  K_N^{(r)})\right | \otimes_N    |M_N| 
(t_i^N,t_j^N, x, y) . \nonumber 
\end{eqnarray}
For the second term on the right hand side of \eqref{eq:iterMN} we get with the help of \eqref{eq:lem3cent4}
 \begin{eqnarray*} 
&& \left | ( \tilde p_N \otimes_N  K_N^{(r)})\right | \otimes_N    |M_N| 
(t_i^N,t_j^N, x, y) \\
&& \leq    \sum_{k=i}^{j-1}\frac {(C(t_k^N-t_i^N))^r}{r!}  \frac {\sqrt{\gamma_{k+1}^N}}{t_j^N - t_k^N} \gamma_{k+1}^N\int_{\R^d} \mathcal Q_{M-d-6} (t_k ^N-t_i^N, z-\theta^N_{t_k ^N,t_i^N}(x) )\\
&&\qquad \times
 \mathcal Q_{M-d-6} (t_j ^N-t_k ^N, y-\theta^N_{t_i ^N,t_k ^N}(z) )\mathrm d z\\
 && \leq    \sum_{k=i}^{j-1}\frac {(C(t_k^N-t_i^N))^r}{r!}  \frac {\sqrt{\gamma_{k+1}^N}}{t_j^N - t_k^N} \gamma_{k+1}^N\mathcal Q_{M-d-6} (t_j ^N-t_i^N, y-\theta^N_{t_j ^N,t_i^N}(x) ) .
\end{eqnarray*}
This gives 
\begin{eqnarray*} 
&& \sum_{r=0}^\infty\left | ( \tilde p_N \otimes_N  K_N^{(r)})\right | \otimes_N    |M_N| 
(t_i^N,t_j^N, x, y) \\
&& \leq  C  \sum_{k=i}^{j-1} \frac {\sqrt{\gamma_{k+1}^N}}{t_j^N - t_k^N} \gamma_{k+1}^N\mathcal Q_{M-d-6} (t_j ^N-t_i^N, y-\theta^N_{t_j ^N,t_i^N}(x) ) \\
&& \qquad \leq  C \sqrt{\gamma_1^N} \ln{\left ( \frac {C (t_j^N-t_i^N)}{\gamma_1^N}\right) } \mathcal Q_{M-d-6} (t_j ^N-t_i^N, y-\theta^N_{t_j ^N,t_i^N}(x) ).
\end{eqnarray*}
 We now treat the first  term on the right hand side of \eqref{eq:iterMN}. First note that for $r=1$
  \begin{eqnarray*} 
&& \left |  \tilde p_N \otimes_N  (K_N +M_N)^{(1)}- \tilde p_N \otimes_N  K_N ^{(1)})\right | \otimes_N    |\mathcal K_N| 
(t_i^N,t_j^N, x, y) \\
&& \leq C \sqrt{\gamma_1 ^N}  \sum_{k=i}^{j-1}\gamma_{k+1}^N \ln{\frac {C(t_k^N-t^N_i)}{\gamma_{1}^N}} \int_{\R^d} \mathcal Q_{M-d-6} (t_k ^N-t_i^N, z-\theta^N_{t_k ^N,t_i^N}(x) )\\
&& \qquad \qquad \times
 \mathcal Q_{M-d-6} (t_j ^N-t_k ^N, \theta^N_{t_k ^N,t_j ^N}(y)-z )\mathrm d z\\
 && \leq C \sqrt{\gamma_1 ^N} \int_{0}^{t_j ^N-t^N_i} \ln{\left(\frac {Cu}{\gamma_{1}^N}\right )} \mathrm d u  \mathcal Q_{M-d-6} (t_j ^N-t_i^N, y-\theta^N_{t_i ^N,0}(x) )\\
  && = C \sqrt{\gamma_1 ^N} {(t_j ^N-t_i^N)} \left (\ln{\left (\frac {C(t_j^N-t^N_i)}{\gamma_{1}^N}\right )}-1\right )  \mathcal Q_{M-d-6} (t_j ^N-t_i^N, y-\theta^N_{t_j^N,t_i ^N}(x) )\\
 && \leq C \sqrt{\gamma_1 ^N} {(t^N_i-t_i ^N)} \ln{\left (\frac {C(t^N_j-t_i ^N)}{\gamma_{1}^N}\right )}  \mathcal Q_{M-d-6} (t_j ^N-t_i^N, y-\theta^N_{t_j ^N,t^N_i}(x) ).
\end{eqnarray*}
Similarly, we get that 
  \begin{eqnarray*} 
&& \left |  \tilde p_N \otimes_N  (K_N +M_N)^{(1)})- \tilde p_N \otimes_N  K_N ^{(1)})\right | \otimes_N    | M_N| 
(t_i^N,t^N_j, x, y) \\
&& \leq C \sqrt{\gamma_1 ^N}  \sum_{k=0}^{i-1}\gamma_{k+1}^N \frac {\sqrt{\gamma_{k+1}^N}}{t_j^N - t_k^N}\ln{\frac {C(t_k^N-t^N_i)}{\gamma_{1}^N}}  \mathcal Q_{M-d-6} (t_j ^N-t_i^N, y-\theta^N_{t_j ^N,t^N_i}(x) )\\
 && \leq C \sqrt{\gamma_1 ^N}
 \sum_{k=0}^{i-1} \gamma_{k+1}^N \frac 1 {\sqrt{t_j^N - t_k^N}}  \ln{\left(\frac {C(t_k^N-t^N_i)}{\gamma_{1}^N}\right )}  \mathcal Q_{M-d-6} (t_j ^N-t_i^N, y-\theta^N_{t_j ^N,t^N_i}(x) )\\
&& \leq C \sqrt{\gamma_1 ^N}
 \int _{t^N_i}^{t_j^N} \frac 1 {\sqrt{t_j^N - u}}  \ln{\left(\frac {C(u-t^N_i)}{\gamma_{1}^N}\right )}\mathrm d u \mathcal Q_{M-d-6} (t_j ^N-t_i^N, y-\theta^N_{t_j^N,t^N_i}(x) )\\
 && \leq 2 C \sqrt{\gamma_1 ^N}
 {\sqrt{t^N_j-t_i^N }}  \ln{\left(\frac {C(t_j^N-t^N_i)}{\gamma_{1}^N}\right )}  \mathcal Q_{M-d-6} (t_j ^N-t_i^N, y-\theta^N_{t_j ^N,t^N_i}(x) ).\end{eqnarray*}
 It follows that
  \begin{eqnarray*} 
&& \left |  \tilde p_N \otimes_N  (K_N +M_N)^{(1)})- \tilde p_N \otimes_N  K_N ^{(1)})\right | \otimes_N    | K_N +M_N| 
(t^N_i,t_j^N, x, y) \\
 && \leq 2 C \sqrt{\gamma_1 ^N}
 {\sqrt{t_i^N-t^N_j }}  \ln{\left(\frac {C(t^N_j-t_i^N)}{\gamma_{1}^N}\right )}  \mathcal Q_{M-d-6} (t_j ^N-t_i^N, y-\theta^N_{t_j ^N,t^N_i}(x) )\end{eqnarray*} 
 and 
  \begin{eqnarray*} 
&& \left |  \tilde p_N \otimes_N  (K_N +M_N)^{(2)})- \tilde p_N \otimes_N  K_N ^{(2)})\right | 
(t^N_i,t_j^N, x, y) \\
 && \leq 3 C \sqrt{\gamma_1 ^N}
 {\sqrt{t^N_j-t_i^N }}  \ln{\left(\frac {C(t^N_j-t_i^N)}{\gamma_{1}^N}\right )}  \mathcal Q_{M-d-6} (t_j ^N-t_i^N, y-\theta^N_{t_i ^N,0}(x) ).\end{eqnarray*} 
 Note that in the upper bound we now have the factor $ {\sqrt{t^N_j -t_i^N }}  \ln{\left(\frac {C(t^N_j-t_i^N)}{\gamma_{1}^N}\right )} $ instead of $  \ln{\left(\frac {C(t^N_j-t_i^N)}{\gamma_{1}^N}\right )} $. In the following iterations we make use of $$\int_0^{s} \sqrt {u} \ln {\left ( \frac {Cu}{ \gamma_1^N}\right )} \mathrm d u \leq \frac 2 3 s ^{\frac 3 2} \ln {\left ( \frac {Cs}{ \gamma_1^N}\right )}$$ and  $$\frac 2 3 \int_0^{s} u^{\frac 3 2} \ln {\left ( \frac {Cu}{ \gamma_1^N}\right )} \mathrm d u \leq \frac 2 3 \frac 2 5 s ^{\frac 5 2} \ln {\left ( \frac {Cs}{ \gamma_1^N}\right )}.$$ 
 Continuing in this way we obtain
  \begin{eqnarray*} 
&& \left |  \tilde p_N \otimes_N  (K_N +M_N)^{(r)})- \tilde p_N \otimes_N  K_N ^{(r)})\right | (t^N_i,t^N_j, x, y) \\
 && \leq C^r \frac {(t^N_j-t^N_i)^r}{(2r-1)!} \sqrt{\gamma_1 ^N}
  \ln{\left(\frac {C(t^N_j-t^N_i)}{\gamma_{1}^N}\right )}  \mathcal Q_{M-d-6} (t^N_j-t^N_i, y-\theta^N_{t^N_j,t^N_i}(x) ).\end{eqnarray*} 
  Summing this up we get the bound
 \begin{eqnarray*} 
&& \left | \sum_{r=0} ^N  \tilde p_N \otimes_N  (K_N +M_N)^{(r)}) (t^N_i,t^N_j, x, y)- \sum_{r=0} ^N \tilde p_N \otimes_N  K_N ^{(r)}(t^N_i,t^N_j, x, y)\right |  \\
 && \leq C \sqrt{\gamma_1 ^N}
  \ln{\left(\frac {1}{\gamma_{1}^N}\right )}  \mathcal Q_{M-d-6} (t_j ^N-t_i^N, y-\theta^N_{t^N_j,t^N_i}(x) ),\end{eqnarray*}  
  which concludes the proof of the lemma.
\end{proof}

\subsection {Proof of inequality \eqref{eq:bound MN}} \label{sec:bound MN}
We start by showing \eqref{eq:bound MN} for $j-i > 1$. The case $j= i+1$ will be treated afterwards. First we will show that
\begin{eqnarray} \label{eq:bound MN1}
M_N(t_i^N,t_j^N,x,y) = T_1 + T_2 + T_3 +T_4 , \end {eqnarray}
where with $  \Delta_i\theta^N(y)  =  \theta^N_{t_{i+1}^N,t_j^N}(y)- \theta^N_{t_i^N,t_j^N}(y)$ and $T_{4}= T_{4,1} - T_{4,2}$
\begin{eqnarray*}
T_1&=& \sum_{l=1} ^d \left [ F_N(t_i^N, \theta^N_{t_i^N,t_j^N}(y)\chi_N(\theta^N_{t_i^N,t_j^N}(y)) - \frac 1 {\gamma_{i+1}^N} \Delta_i\theta^N(y) \right ] \frac {\partial}{ \partial x_l} \tilde p_N(t_{i+1}^N,t_j^N,x,y)  ,\\
T_2&=& \frac 1 2 \sum_{l,m=1} ^d \left [  R_{l,m} \left (\bar \theta^N_{t_i}+\chi_N(x) \sqrt{ \gamma_i^N}\right )- R_{l,m} \left (\bar \theta^N_{t_i}+\chi_N(\theta^N_{t_i^N,t_j^N}(y))\sqrt{ \gamma_i^N}\right ) \right ] \\
&& \qquad \times \frac {\partial^2}{ \partial x_l \partial x_m} \tilde p_N(t_{i+1}^N,t_j^N,x,y)  ,\\
T_3&=& \frac {\gamma_{i+1}^N}  2 \sum_{l,m=1} ^d  \left (\left [  F_N \left (t_{i}^N,x\right )\chi_N(x) \right ]_l \left [  F_N \left (t_{i}^N,x\right )\chi_N(\theta^N_{t_i^N,t_j^N}(y)
) \right ]_m - 
\left [  F_N \left (t_{i}^N,\theta^N_{t_i^N,t_j^N}(y)
\right )\right. \right .
\\
&& \qquad
 \left . \left . 
\times
\chi_N(\theta^N_{t_i^N,t_j^N}(y)
) \right ]_l \right .  \left.\left [  F_N \left (t_{i}^N,\theta^N_{t_i^N,t_j^N}(y)\right )\chi_N(\theta^N_{t_i^N,t_j^N}(y)
) \right ]_m \right )
 \frac {\partial^2}{ \partial x_l \partial x_m} \tilde p_N(t_{i+1}^N,t_j^N,x,y)  ,\\
 T_{4,1}&=& 3 \sqrt{\gamma_{i+1}^N}\int_{\R^d} f^N_{t_i^N,\chi_{N}(x)\sqrt{\gamma_{i}^{N}}}(v) \sum_{|\nu| =3} \frac {\left [v + F_N(t_i^N,x) \chi_N(x) \sqrt {\gamma_{i+1} ^N} \right ]^\nu} {\nu !}  \\
&& \qquad \times\int_0^1 ( 1- \delta) ^2 D^\nu \tilde p_N^y\left (t_{i+1}^N, t_j^N,  x + \delta \left (v \sqrt {\gamma_{i+1} ^N} + F_N(t_i^N,x) \chi_N(x) \gamma_{i+1} ^N\right ), y\right ) \mathrm d \delta \mathrm dv,\\
T_{4,2}&=&3 \sqrt{\gamma_{i+1}^N}\int_{\R^d} f^N_{t_i^N,\theta^N_{t_i^N,t_j^N}(y)}(v) \sum_{|\nu| =3} \frac {\left [v + \Delta_i\theta^N(y) \right ]^\nu} {\nu !}  \\
&& \qquad \times\int_0^1 ( 1- \delta) ^2 D^\nu \tilde p_N^y\left (t_{i+1}^N, t_j^N,  x + \delta \left(v \sqrt {\gamma_{i+1} ^N} + \Delta_i\theta^N(y)\right), y\right) \mathrm d \delta \mathrm dv,
\end{eqnarray*}
where $f^N_{t_i^N,\theta} $ is the density of $$\xi_{i,\theta}= H\left ( \bar \theta^N_{t_i} + \chi_N (\theta) \sqrt {\gamma_i^N} , \eta_{i+1}\right ) - \E \left [H\left ( \bar \theta^N_{t_i} + \chi_N (\theta) \sqrt {\gamma_i^N} , \eta_{i+1}\right ) \right ].$$ 
For $1 \leq k \leq 4$ we will show that
\begin{eqnarray} \label{eq:bound MN2}
&&\left | T_k \right | \leq C \frac { \sqrt{\gamma_{i+1}^N}}{t_j^N-t_i^N} 
 \mathcal Q_{M-d-6} (t_j ^N-t_i ^N, \theta^N_{t_i ^N,t_j ^N}(y)-x ),
\end{eqnarray}
for $j-i > 1$. Claim \eqref{eq:bound MN} follows then from \eqref{eq:bound MN1} and \eqref{eq:bound MN2}. We start with a proof of \eqref{eq:bound MN1}.
Note first that for the one-step transition densities we have
\begin{eqnarray*} 
p_N( t_{i}^ N, t_{i+1}^ N, x,z) &=& (\gamma_{i+1}^N)^{-d/2} f^N_{t_i^N,\chi_{N}(x)\sqrt{\gamma_{i}^{N}}} \left ( \frac {y-x-F_N(t_i^N,x) \chi_N(x) \gamma_{i+1}^N}{\sqrt{\gamma_{i+1}^N}}\right ),\\
\tilde p_N^y( t_{i}^ N, t_{i+1}^ N, x,z) &=& (\gamma_{i+1}^N)^{-d/2} f^N_{t_i^N,\theta^N_{t_i^N, t_j^N}(y)} \left ( \frac {y-x-\int
_{t_i^N}^{t_j^N} F_N(
u,
\theta^N_{u, t_j^N}(y)) \chi_N(\theta^N_{u, t_j^N}(y)) \gamma_{i+1}^N \mathrm d u}{\sqrt{\gamma_{i+1}^N}}\right ).
\end{eqnarray*} 
With these representations of the  one-step transition densities we get with  $\varphi = \tilde p^y_N$  that
\begin{eqnarray*} 
&&\mathcal K_N( t_{i}^ N, t_{j}^ N, x,y) = \left ( {\mathcal L}_N - \tilde {\mathcal L}_N\right ) \varphi( t_{i}^ N, t_{j}^ N, x,y)\\
&& \qquad =  \left ( {\mathcal L}_N - \tilde {\mathcal L}_N\right )_1 \varphi( t_{i}^ N, t_{j}^ N, x,y) -  \left ( {\mathcal L}_N - \tilde {\mathcal L}_N\right )_2 \varphi( t_{i}^ N, t_{j}^ N, x,y) ,
\end{eqnarray*} 
where 
\begin{eqnarray*} 
&&\left ( {\mathcal L}_N - \tilde {\mathcal L}_N\right )_1 \varphi( t_{i}^ N, t_{j}^ N, x,y)  =  \frac1 {\gamma_{i+1}^N} \int_{\mathbb R^d} f^N_{t_i^N,\chi_{N}(x)\sqrt{\gamma_{i}^{N}}} \left ( \frac {z-x-F_N(t_i^N,x) \chi_N(x) \gamma_{i+1}^N}{\sqrt{\gamma_{i+1}^N}}\right )\\
&& \qquad \times \left ( \sum_{1 \leq | \nu |\leq 2} \frac {(z-x)^\nu} {\nu!} (D^\nu\varphi)( t_{i+1}^ N, t_{j}^ N, x,y) \right. 
\\
&& \qquad \left. + 3 \sum_{ | \nu |=3} \frac {(z-x)^\nu} {\nu!}\int_0^1 (1 - \delta) ^2  (D^\nu\varphi)( t_{i+1}^ N, t_{j}^ N, x + \delta (z-x),y) \mathrm d \delta \right) \mathrm d z
\end{eqnarray*} 
and 
\begin{eqnarray*} 
&&\left ( {\mathcal L}_N - \tilde {\mathcal L}_N\right )_2 \varphi( t_{i}^ N, t_{j}^ N, x,y)  =  \frac1 {\gamma_{i+1}^N} \int_{\mathbb R^d} f^N_{t_i^N,\theta^N_{t_i^N,t_j^N(y)}} \left ( \frac {z-x-\Delta_i\theta^N(y)}{\sqrt{\gamma_{i+1}^N}}\right )\\
&& \qquad \times \left ( \sum_{1 \leq | \nu |\leq 2} \frac {(z-x)^\nu} {\nu!} (D^\nu\varphi)( t_{i+1}^ N, t_{j}^ N, x,y) \right. 
\\
&& \qquad \left. + 3 \sum_{ | \nu |=3} \frac {(z-x)^\nu} {\nu!}\int_0^1 (1 - \delta) ^2  (D^\nu\varphi)( t_{i+1}^ N, t_{j}^ N, x + \delta (z-x),y) \mathrm d \delta \right) \mathrm d z.
\end{eqnarray*} 
We now note that
\begin{eqnarray} \label{eq:bound MN3}
&&\left ( {\mathcal L}_N - \tilde {\mathcal L}_N\right )_1 \varphi( t_{i}^ N, t_{j}^ N, x,y)  =  \frac1 {\gamma_{i+1}^N} \int_{\mathbb R^d} f^N_{t_i^N,\chi_{N}(x)\sqrt{\gamma_{i}^{N}}} (v))\\
&& \qquad \times  \left( \sum_{1 \leq | \nu |\leq 2} \frac {\left (
v\sqrt{\gamma_{i+1}^N}
+F_N(t_i^N,x) \chi_N(x) \gamma_{i+1}^N
 \right )^\nu} {\nu!} (D^\nu\varphi)( t_{i+1}^ N, t_{j}^ N, x,y) \right. \nonumber 
\\
&& \qquad \left. + 3 \sum_{ | \nu |=3} \frac {(z-x)^\nu} {\nu!}\int_0^1 (1 - \delta) ^2  (D^\nu\varphi)( t_{i+1}^ N, t_{j}^ N, x + \delta \left (
v\sqrt{\gamma_{i+1}^N}
+F_N(t_i^N,x) \chi_N(x) \gamma_{i+1}^N
 \right ),y) \mathrm d \delta \right) \mathrm d z \nonumber 
 \end{eqnarray} 
 \begin{eqnarray}  \nonumber 
     && = \sum_{l=1} ^d \left [ F_N(t_i^N,x) \chi_N(x)\right ]_l \frac {\partial} {\partial x_l} \varphi( t_{i}^ N, t_{j}^ N, x,y)\\
     && \qquad + \frac 1 2  \sum_{l,m=1} ^d \left (R_{l,m}\left (\bar \theta^N_{t_i^N}+ \chi_N(x) \sqrt{\gamma_{i+1}^N}\right)
     + \frac {\gamma_{i+1}^N} 2 \left [ F_N(t_i^N,x) \chi_N(x)\right ]_l \left [ F_N(t_i^N,x) \chi_N(x)\right ]_m
     \right )\nonumber  \\ && \nonumber  \qquad \qquad \times  \frac {\partial^2} {\partial x_l \partial x_m} \varphi( t_{i}^ N, t_{j}^ N, x,y)\\ \nonumber 
     && \qquad + 3  \sqrt{\gamma_{i+1}^N} \int_{\mathbb R^d} f^N_{t_i^N,\chi_{N}(x)\sqrt{\gamma_{i}^{N}}} (v) \sum_{|\nu| = 3} \frac { \left (
v
+F_N(t_i^N,x) \chi_N(x) \sqrt{\gamma_{i+1}^N}
 \right )^\nu}{\nu!} \int_0^1 (1 - \delta) ^2\\ \nonumber 
 && \nonumber   \qquad \qquad \times (D^\nu\varphi) 
  \left ( t_{i+1}^ N, t_{j}^ N, x+ \delta \left (
v\sqrt{\gamma_{i+1}^N}
+F_N(t_i^N,x) \chi_N(x) \gamma_{i+1}^N,y
 \right ) \right ) \mathrm d \delta \mathrm d v. \nonumber 
\end{eqnarray} 
By similar arguments one obtains that
\begin{eqnarray} \label{eq:bound MN4}
&&\left ( {\mathcal L}_N - \tilde {\mathcal L}_N\right )_2 \varphi( t_{i}^ N, t_{j}^ N, x,y)   =\frac 1 {\gamma_{i+1}^N}  \sum_{l=1} ^d \left [ \Delta_i\theta^N(y)\right ]_l \frac {\partial} {\partial x_l} \varphi( t_{i+1}^ N, t_{j}^ N, x,y)\\
     && \qquad + \frac 1 2  \sum_{l,m=1} ^d \left (R_{l,m}\left (\bar \theta^N_{t_i^N}+ \chi_N(\theta^N_{t_i^N},t_j^N(y)) \sqrt{\gamma_{i+1}^N}\right)
    + \frac {\gamma_{i+1}^N} 2 \left [ \Delta_i\theta^N(y)\right ]_l \left [ \Delta_i\theta^N(y)\right ]_m
     \right )\nonumber  \\ && \nonumber  \qquad \qquad \times  \frac {\partial^2} {\partial x_l \partial x_m} \varphi( t_{i+1}^ N, t_{j}^ N, x,y)\\ \nonumber 
    && \qquad + 3  \sqrt{\gamma_{i+1}^N} \int_{\mathbb R^d} f^N_{t_i^N,\theta^N_{t_i^N,t_j^N(y)}} (v) \sum_{|\nu| = 3} \frac { \left (
v+ \Delta_i\theta^N(y)
 \right )^\nu}{\nu!}\\ \nonumber 
 && \nonumber   \qquad \qquad \times  \int_0^1 (1 - \delta) ^2 (D^\nu\varphi) 
  \left ( t_{i+1}^ N, t_{j}^ N, x+ \delta \left (
v\sqrt{\gamma_{i+1}^N}
+\Delta_i\theta^N(y),y
 \right ) \right ) \mathrm d \delta \mathrm d v. \nonumber 
\end{eqnarray}

Claim \eqref{eq:bound MN1} now follows for  $M_N = \mathcal K_N -K_N$ by using \eqref{eq:bound MN3}, \eqref{eq:bound MN4} and 
\begin{eqnarray*} 
&&K_N(t_i^N, t_j^N, z,y) = (L_{t_i^N}^N - \tilde L_{t_i^N}^N)  \tilde p_N(t_i^N, t_j^N,z,y)\\
&& \qquad =\sum_{l,m=1}^d  \left ([F_N(t_i^N,z)]_{l,m}[\chi_N(z)]_m
- [F_N(t_i^N,\theta_{t_i^N,t_j^N} ^N(y))]_{l,m}[\chi_N(\theta_{t_i^N,t_j^N} ^N(y))]_m \right )    \\
&& \qquad \qquad    \times
\frac {\partial} {\partial z_l} \tilde p_N(t_i^N, t_j^N,z,y).
\end{eqnarray*}

We now come to the proof of  \eqref{eq:bound MN2} for $1 \leq k \leq 4$. For  $k=1$ note that by application of \eqref{eqapp5} and \eqref{eq:boundstheta1}
\begin{eqnarray*}
|T_1|&\leq& \sum_{l=1} ^d \frac 1 {\gamma_{i+1}^N} \int_{t_i^N}^{t_{i+1}^N} \left |\left [ F_N(t_i^N, \theta^N_{t_i^N,t_j^N}(y))\chi_N(\theta^N_{t_i^N,t_j^N}(y)) - F_N(u, \theta^N_{u,t_j^N}(y))\chi_N(\theta^N_{u,t_j^N}(y))
 \right ] _l \right | \mathrm d u \\
 && \qquad \times \left |\frac {\partial}{ \partial x_l} \tilde p_N(t_{i+1}^N,t_j^N,x,y)  \right |
 \\
 &\leq& C\left ( \max_{t_i^N \leq u \leq t_{i+1}^N} \left |  \theta^N_{u,t_j^N}(y) -  \theta^N_{t_i^N,t_j^N}(y)\right | + |y| \gamma_{i+1} ^N \right ) \\
 && \qquad \times
 (t_j^N-t_i^N)^{-1/2}
 \mathcal Q_{M-d-6} (t_j ^N-t_i ^N, \theta^N_{t_i ^N,t_j ^N}(y)-x ) \\
 &\leq& C a_N \gamma_{i+1}^N
 (t_j^N-t_i^N)^{-1/2}
 \mathcal Q_{M-d-6} (t_j ^N-t_i ^N, \theta^N_{t_i ^N,t_j ^N}(y)-x ),
\end{eqnarray*}
where the last inequality follows directly if $|y| \leq a_N +1$. If  $|y| > a_N +1$ we get $\theta^N_{t,t_j^N}(y) \equiv y$ for $t \leq t_j^N$ is a solution of    $\frac {\mathrm d } { \mathrm dt } \theta^N_{t,t_j^N}(y)  =  F_N(t,  \theta^N_{t,t_j^N}(y) )\chi_N(\theta^N_{t,t_j^N}(y))$ with terminating value  $\theta^N_{t_j^N,t_j^N}(y) = y$. Thus in the latter case we have that $T_1=0$. This shows that in both cases \eqref{eq:bound MN2} holds for $k=1$.

For the proof of \eqref{eq:bound MN2} for $k=2$, 
\begin{eqnarray*}
|T_2|&\leq &C \sum_{l,m=1} ^d \left |  R_{l,m} \left (\bar \theta^N_{t_i}+\chi_N(x) \sqrt{ \gamma_i^N}\right )- R_{l,m} \left (\bar \theta^N_{t_i}+\chi_N(\theta^N_{t_i^N,t_j^N}(y))\sqrt{ \gamma_i^N}\right ) \right| \\
&& \qquad \times \left | \frac {\partial^2}{ \partial x_l \partial x_m} \tilde p_N(t_{i+1}^N,t_j^N,x,y)  \right | \\
&\leq &C \sqrt{ \gamma_i^N} \sum_{l,m=1} ^d \left |  \chi_N(x) -\chi_N(\theta^N_{t_i^N,t_j^N}(y)) \right| 
\  \left | \frac {\partial^2}{ \partial x_l \partial x_m} \tilde p_N(t_{i+1}^N,t_j^N,x,y)  \right |\\
&\leq& C  \sqrt{ \gamma_i^N} 
 (t_j^N-t_{i+1}^N)^{-1/2} \left | \frac {\theta^N_{t_i^N,t_j^N}(y)-x} {(t_j^N-t_{i+1}^N)^{1/2}} \right |
 \mathcal Q_{M-d-5} (t_j ^N-t_i ^N, \theta^N_{t_i ^N,t_j ^N}(y)-x )\\
&\leq& C  \sqrt{ \gamma_{i+1}^N} 
 (t_j^N-t_{i}^N)^{-1/2} \left | \frac {\theta^N_{t_i^N,t_j^N}(y) -x} {(t_j^N-t_{i}^N)^{1/2}} \right |
 \mathcal Q_{M-d-5} (t_j ^N-t_i ^N, \theta^N_{t_i ^N,t_j ^N}(y)-x ),
 \end{eqnarray*}
 where in the last inequality we used that for some $c>0$ small enough $c/(c+1) \leq (t_j^N-t_{i+1}^N)/(t_j^N-t_{i}^N)\leq 1$ because of 
 $(t_j^N-t_{i+1}^N)/(t_j^N-t_{i}^N) = 1 - \gamma_{i+1}^N / (t_j^N-t_{i}^N)$ and $\gamma_{i+1}^N / (t_j^N-t_{i}^N) \leq 1 / ( 1 + \gamma_{i+2}^N/\gamma_{i+1}^N) \leq 1 / (1+c) $ for $c > 0$ small enough. We conclude that 
 \begin{eqnarray*}
|T_2|&\leq &C  \sqrt{ \gamma_{i+1}^N} 
 (t_j^N-t_{i}^N)^{-1/2}
 \mathcal Q_{M-d-6} (t_j ^N-t_i ^N, \theta^N_{t_i ^N,t_j ^N}(y)-x ),
 \end{eqnarray*}
which shows \eqref{eq:bound MN2} for $k=2$. For $k=3$ we get by using \eqref{eqapp5} that 
\begin{eqnarray*}
T_3&\leq & \frac {a_N \gamma_{i+1}^N}  {t_j^N-t_{i+1}^N} \left | \theta^N_{t_i^N,t_j^N}(y)-x\right | \mathcal Q_{M-d-5} (t_j ^N-t_i ^N, \theta^N_{t_i ^N,t_j ^N}(y)-x ), 
\end{eqnarray*} which by arguments similar to the ones used for $k=2$ can be bounded by
\begin{eqnarray*}
 \frac {a_N \gamma_{i+1}^N}  {(t_j^N-t_{i+1}^N) ^{1/2}}  \mathcal Q_{M-d-6} (t_j ^N-t_i ^N, \theta^N_{t_i ^N,t_j ^N}(y)-x ).
\end{eqnarray*} 
This shows \eqref{eq:bound MN2} for $k=3$. For $k=4$ choose $\nu$ with $|\nu| =3$ and put
\begin{eqnarray*}
p_1(v) &=&  f^N_{t_i^N,\chi_{N}(x)\sqrt{\gamma_{i}^{N}}}(v), \\
p_2(v) &=& f^N_{t_i^N,\theta^N_{t_i^N,t_j^N}(y)}(v), \\
a_1(v) &=& (v + \delta_{1,N})^{\nu},\\
a_2(v) &=& (v + \delta_{2,N})^{\nu}
\end{eqnarray*}
with 
\begin{eqnarray*}
 \delta_{1,N}&=&   F_N(t_i^N,x) \chi_N(x) \sqrt {\gamma_{i+1} ^N}, \\
 \delta_{2,N}&=& \Delta_i\theta^N(y)\\
 &=&  \frac 1 {\sqrt{\gamma_{i+1}^N}} \int_{t_i^N}^{t_{i+1}^N} F_N(u, \theta^N_{u,t_j^N}(y))\chi_N(\theta^N_{u,t_j^N}(y))
 \mathrm d u.
\end{eqnarray*}
From the latter representation of $\delta_{2,N}$ one gets the following bounds with the help of \eqref{eqapp5}, of   \eqref{eq:lem3cent5b} in Lemma \ref{em2:boundstheta}, by using that $F_N$ is Lipshitz with respect to its first argument and that $\chi_N$ is absolutely bounded by $a_N$. 
\begin{eqnarray} \label{eq:mon1}
 \left | \delta_{2,N} -  \delta_{1,N}\right |&= & \tilde  \delta_{1,N} + \tilde  \delta_{2,N},\\
 \label{eq:mon2}  \left | \delta_{l,N}\right |&\leq & C a_N  \sqrt{\gamma_{i+1}^N} , \\
 |a_l(v)| &\leq & C (v +1)^{\nu}  \label{eq:mon3}
 \end{eqnarray}
 for $l=1,2$,
 where \begin{eqnarray*}
 \tilde  \delta_{1,N} &\leq &C \sqrt{\gamma_{i+1}^N} 
\left | x-\theta^N_{t_i^N,t_j^N}(y)\right |,\\  \tilde  \delta_{2,N} &\leq & Ca_N (\gamma_{i+1}^N)^{3/2}
 .\end{eqnarray*}
Furthermore, we put for $l=1,2$
\begin{eqnarray*}
b_l(v)&=&   \int_0^1 ( 1- u) ^2 D^\nu \tilde p_N^y\left (t_{i+1}^N, t_j^N,  x +u \sqrt {\gamma_{i+1} ^N} (v+\delta_{l,N}), y\right ) \mathrm d u .\end{eqnarray*}
We will show
\begin{eqnarray} \label{eq:sun1}
 \left |\int_{\R^d} (p_1(v) -p_2(v)) a_1(v) b_1(v) \mathrm d v \right |&\leq & C \frac { \sqrt{\gamma_{i+1}^N}}{t_j^N-t_i^N} 
 \mathcal Q_{M-d-6} (t_j ^N-t_i ^N, \theta^N_{t_i ^N,t_j ^N}(y)-x )\\ \label{eq:sun2}
 \left |\int_{\R^d} p_2(v) a_1(v) (b_1(v)- b_2(v))  \mathrm d v \right |&\leq & C \frac { \sqrt{\gamma_{i+1}^N}}{t_j^N-t_i^N} 
 \mathcal Q_{M-d-6} (t_j ^N-t_i ^N, \theta^N_{t_i ^N,t_j ^N}(y)-x ),\\ \label{eq:sun3}
 \left |\int_{\R^d} p_2(v) b_2(v) (a_1(v)-a_2(v)) \mathrm d v \right |&\leq & C \frac { \sqrt{\gamma_{i+1}^N}}{t_j^N-t_i^N} 
 \mathcal Q_{M-d-6} (t_j ^N-t_i ^N, \theta^N_{t_i ^N,t_j ^N}(y)-x ).\end{eqnarray}
 These three inequalities  imply that 
 \begin{eqnarray*}
 \left |\int_{\R^d} p_1(v)  a_1(v) b_1(v) \mathrm d v- \int_{\R^d} p_2(v)  a_2(v) b_2(v) \mathrm d v \right |&\leq & C \frac { \sqrt{\gamma_{i+1}^N}}{t_j^N-t_i^N} 
 \mathcal Q_{M-d-6} (t_j ^N-t_i ^N, \theta^N_{t_i ^N,t_j ^N}(y)-x ).\end{eqnarray*}
By summing both sides of the last inequality over $\nu$ with $|\nu| =3$ we get  \eqref{eq:bound MN2} for $k=4$. We now show \eqref{eq:sun1}--\eqref{eq:sun3}. We start with the proof of \eqref{eq:sun2}. For this proof we will use that
\begin{eqnarray} \label{eq:tue1}
 \left |b_1(v)- b_2(v) \right |&\leq & C \frac { \gamma_{i+1}^N}{(t_j^N-t_i^N) ^{3/2}}  (1+ |v| ) ^{M-d-5}
 \mathcal Q_{M-d-6} (t_j ^N-t_i ^N, \theta^N_{t_{i+1} ^N,t_j ^N}(y)-x ).\end{eqnarray}
 With this inequality and \eqref{eq:mon3} we get that 
 \begin{eqnarray*}
&& \left |\int_{\R^d} p_2(v) a_1(v) (b_1(v)- b_2(v))  \mathrm d v \right |\leq 
C \frac { \gamma_{i+1}^N}{(t_j^N-t_i^N) ^{3/2}}  \mathcal Q_{M-d-6} (t_j ^N-t_i ^N, \theta^N_{t_{i+1} ^N,t_j ^N}(y)-x )\\ && \qquad \times
\int_{\R^d}  f^N_{t_i^N,\theta^N_{t_i^N,t_j^N}(y)}(v) (1+ |v| ) ^{M-d-5+3}\mathrm d v \\
&& \leq C \frac { \gamma_{i+1}^N}{(t_j^N-t_i^N) ^{3/2}}  \mathcal Q_{M-d-6} (t_j ^N-t_i ^N, \theta^N_{t_{i+1} ^N,t_j ^N}(y)-x ),
\end{eqnarray*}
which implies \eqref{eq:sun2} because of $ \gamma_{i+1}^N \leq t_j^N-t_i^N$. For \eqref{eq:sun2} it remains to show \eqref{eq:tue1}. For this proof 
note that by \eqref{eq:mon1} and \eqref{eq:lem3cent2add} with $\delta_{N}(\tau^N) = \delta_{1,N} + \tau^N ( \delta_{2,N} - \delta_{1,N})$
\begin{eqnarray*} 
&& \left |b_1(v)- b_2(v) \right |\\
&& \qquad \leq 
  \int_0^1 ( 1- u) ^2 \left| D^\nu \tilde p_N^y\left (t_{i+1}^N, t_j^N,  x +u \sqrt {\gamma_{i+1} ^N} (v+\delta_{l,N}), y\right )\right . \\ && \qquad \quad  \left .  -D^\nu \tilde p_N^y\left (t_{i+1}^N, t_j^N,  x +u \sqrt {\gamma_{i+1} ^N} (v+\delta_{l,N}), y\right ) \right |\mathrm d u\\
  && \qquad \leq  \int_0^1 ( 1- u) ^2
  \left | u \sqrt{\gamma_{i+1}^N} \sum_{k=1} ^d (\delta_{2N}- \delta_{1,N})_k\right . \\
  && \qquad \qquad \left . \times  \int_0^1 D^{\nu+e_k} \tilde p_N^y\left (t_{i+1}^N, t_j^N,  x +u \sqrt {\gamma_{i+1} ^N} (v+\delta_{N}(\tau^N)), y\right )\mathrm d \tau^N \right | \mathrm d u\\
&&  \qquad \leq C \frac { \gamma_{i+1}^N}{(t_j^N-t_i^N) ^{2}}  \left ( \left |\theta^N_{t_{i+1} ^N,t_j ^N}(y)-x\right| + a_N  \gamma_{i+1}^N \right ) 
 \int_0^1 ( 1- u) ^2  \\ && \qquad \qquad \times \int_0^1 
 \mathcal Q_{M-d-6} (t_j ^N-t_i ^N, \theta^N_{t_{i+1} ^N,t_j ^N}(y)-x - u \sqrt {\gamma_{i+1} ^N} (v+\delta_{N}(\tau^N)) ) \mathrm d \tau^N\  \mathrm d u.\end{eqnarray*}
 Now by application of \eqref{eq:Qr1} we get that
 \begin{eqnarray*} 
&& \left |b_1(v)- b_2(v) \right |\\
&& \qquad \leq 
 C \frac { \gamma_{i+1}^N}{(t_j^N-t_i^N) ^{2}}  \left ( \left |\theta^N_{t_{i+1} ^N,t_j ^N}(y)-x\right| + a_N  \gamma_{i+1}^N \right )  
 \left ( 1+ |v|  + a_N  \sqrt{\gamma_{i+1}^N} \right )  ^{M-d-5}
 \int_0^1 ( 1- u) ^2  \\ && \qquad \qquad \times \int_0^1 
 \mathcal Q_{M-d-6} (t_j ^N-t_i ^N, \theta^N_{t_{i+1} ^N,t_j ^N}(y)-x - u \sqrt {\gamma_{i+1} ^N} \tau^N \tilde  \delta_{1,N}) ) \mathrm d \tau^N\  \mathrm d u.\end{eqnarray*}
 Claim \eqref{eq:mon3} now follows by application of  \eqref{eq:Qr2}.
 
 We now show  \eqref{eq:sun3}. Note that by application of \eqref{eq:lem3cent2add}, \eqref{eq:Qr1}
 
  \begin{eqnarray*} 
&&  \left |\int_{\R^d} p_2(v) b_2(v) (a_1(v)-a_2(v)) \mathrm d v \right |\\
&& \qquad \leq C \int_{\R^d} 
f^N_{t_i^N,\theta^N_{t_i^N,t_j^N}(y)}(v)
\int_0^1 ( 1- u) ^2 \left |D^\nu \tilde p_N^y\left (t_{i+1}^N, t_j^N,  x +u \sqrt {\gamma_{i+1} ^N} (v+\delta_{2,N}), y\right ) \right | \mathrm d u
 \\
 &&  \qquad \qquad \times \left | (v + \delta_{1,N})^{\nu} - (v + \delta_{2,N})^{\nu}  \right  | \mathrm d v\\
&& \qquad \leq C  (t_j^N-t_{i+1}^N) ^{-3/2} \int_{\R^d} 
f^N_{t_i^N,\theta^N_{t_i^N,t_j^N}(y)}(v) \mathcal Q_{M-d-5} (t_j ^N-t_i ^N, \theta^N_{t_{i+1} ^N,t_j ^N}(y)-x )  \\
 &&  \qquad \qquad \times
(1+|v |+ \delta_{2,N} ) ^{M-d-5} \left ( |v|^2 |\delta_{1N} - \delta_{2,N}|+ |v| |\delta_{1N} ^2- \delta_{2,N}^2|+ |\delta_{1N}^3- \delta_{2,N}^3|\right )
 \mathrm d v\\
&& \qquad \leq C  (t_j^N-t_{i+1}^N) ^{-3/2} 
\mathcal Q_{M-d-5} (t_j ^N-t_i ^N, \theta^N_{t_{i+1} ^N,t_j ^N}(y)-x )  \\
 &&  \qquad \qquad \times
\sqrt{\gamma_{i+1}^N}  \left ( |\theta^N_{t_{i+1} ^N,t_j ^N}(y)-x| + \gamma_{i+1}^N\right ),
\end{eqnarray*}
 where in the last step \eqref{eq:mon1}--\eqref{eq:mon2} has been applied. Claim  \eqref{eq:sun3} now follows by application of \eqref{eq:Qr2}. We now prove  \eqref{eq:sun1}.
 For the proof of this claim first note  that by our assumptions and by \eqref{eq:lem3cent2} and \eqref{eq:Qr1}:
 \begin{eqnarray*} 
&&|p_1(v)-p_2(v)|=  \left | f^N_{t_i^N,\chi_{N}(x)\sqrt{\gamma_{i}^{N}}}(v) - f^N_{t_i^N,\theta^N_{t_i^N,t_j^N}(y)}(v) \right |\\
&& \quad \leq C \sqrt{\gamma_{i+1}^N} 
 \left | \theta^N_{t_i^N,t_j^N}(y) - x \right |Q_m(v),
\\
&&|b_1(v)|\leq  \int_0^1 ( 1- u) ^2 \left |D^\nu \tilde p_N^y\left (t_{i+1}^N, t_j^N,  x +u \sqrt {\gamma_{i+1} ^N} (v+\delta_{l,N}), y\right ) \right | \mathrm d u \\
&& \quad \leq C (t_j^N-t_{i+1}^N) ^{-3/2}  \int_0^1 ( 1- u) ^2  \mathcal Q_{M-d-5} (t_j ^N-t_i ^N, \theta^N_{t_{i+1} ^N,t_j ^N}(y)-x - u \sqrt {\gamma_{i+1} ^N} (v+\delta_{2,N}))  \mathrm d u\\
&& \quad \leq C (t_j^N-t_{i+1}^N) ^{-3/2} (1+|v|)^ {M-d-5}   \mathcal Q_{M-d-5} (t_j ^N-t_i ^N, \theta^N_{t_{i+1} ^N,t_j ^N}(y)-x   ) .
\end{eqnarray*}
Claim \eqref{eq:sun1} now follows by combining these bounds with the bound \eqref{eq:mon3} for $a_1(v)$ and by applying \eqref{eq:Qr2}.

This concludes the proof of  \eqref{eq:bound MN} for $j-i > 1$. It remains to show \eqref{eq:bound MN} for the case $j= i+1$. 
 Because of \eqref{eq:boundKN} we have to show that 
  \begin{eqnarray} \label{eq:wed1}
\mathcal K_N (t_i^N, t_{i+1}^N, x,y) &\leq & C \frac 1 { \sqrt{\gamma_{i+1}^N}}
 \mathcal Q_{M-d-6} (t_{i+1} ^N-t_i ^N, \theta^N_{t_i ^N,t_{i+1} ^N}(y)-x ). 
 \end{eqnarray}
{Note that 
\begin{eqnarray*} 
&&\mathcal K_N (t_i^N, t_{i+1}^N, x,y) = (\mathcal L_N -\tilde {\mathcal L}_N) \tilde p_N (t_i^N, t_{i+1}^N, x,y)\\
&& \qquad = \frac 1 {\gamma_{i+1}^N} \left (  p_N (t_i^N, t_{i+1}^N, x,y) - \tilde p_N (t_i^N, t_{i+1}^N, x,y)\right )\\
&& \qquad = \frac 1 {\gamma_{i+1}^N} \left ( \left(\gamma_{i+1}^N\right) ^{-d/2} f^N_{t_i^N,\chi_{N}(x)\sqrt{\gamma_{i}^{N}}} \left ( \frac {y-x - \int_{t_i^N}^{t_{i+1} ^N} F_N\left ( u, \theta^N_{u, t_{i+1} ^N}(y)\right) \chi_N ( \theta_{u, t_{i+1} }^N(y)) \mathrm d u} {\sqrt{\gamma_{i+1}^N}}\right  ) \right .\\
&&  \qquad \qquad- \left .  \left(\gamma_{i+1}^N\right) ^{-d/2} f^N_{t_i^N,\chi_{N}(x)\sqrt{\gamma_{i}^{N}}} \left ( \frac {y-x -  F_N\left ( t_{i} ^N, \theta^N_{u, t_{i+1} ^N}(y)\right) \chi_N ( x) \gamma_{i+1}^N} {\sqrt{\gamma_{i+1}^N}}\right  ) \right ).
 \end{eqnarray*}
 We now use that $\int_{t_i^N}^{t_{i+1} ^N} F_N\left ( u, \theta^N_{u, t_{i+1} ^N}(y)\right) \chi_N ( \theta^N_{u, t_{i+1} ^N}(y)) \mathrm d u =  \theta^N_{t_{i+1} ^N, t_{i+1} ^N}(y) - \theta^N_{t_{i} ^N, t_{i+1} ^N}(y) =  y - \theta^N_{t_{i} ^N, t_{i+1} ^N}(y) $. Thus with 
 \begin{eqnarray*}  \Delta_i(x,y) = \frac 1 {\sqrt {\gamma_{i+1}^N}}   \int_{t_i^N}^{t_{i+1} ^N} \left [ F_N\left ( u, \theta^N_{u, t_{i+1} ^N}(y)\right) \chi_N ( \theta^N_{u, t_{i+1} ^N}(y)) - F_N\left ( t_{i} ^N,x\right) \chi_N ( x) \right ]\mathrm d u
\end{eqnarray*}
we can write 
\begin{eqnarray*}
\left| \mathrm{\Delta}_{i}(x,y) \right| &\leq& \frac{1}{\sqrt{\gamma_{i + 1}^{N}}}\int_{t_{i}^{N}}^{t_{i + 1}^{N}}{\left| F_{N}(u,\theta_{u,t_{i + 1}^{N}}^{N}(y))\chi_{N}(\theta_{u,t_{i + 1}^{N}}^{N}(y))\  - F_{N}(u,x)\chi_{N}(x) \right|du}\\
&& \qquad + \frac{1}{\sqrt{\gamma_{i + 1}^{N}}}\int_{t_{i}^{N}}^{t_{i + 1}^{N}}{\left| \chi_{N}(x) \right|  \left| F_{N}(u,x)\  - F_{N}(t_{i}^{N},x) \right|}du \\
& \leq& \frac{L}{\sqrt{\gamma_{i + 1}^{N}}}\int_{t_{i}^{N}}^{t_{i + 1}^{N}}{\left| \theta_{u,t_{i + 1}^{N}}^{N}(y)\  - x \right|du}\\
&& \qquad + C\ a_{N}{ (\gamma_{i + 1}^{N})}^{3/2}\\
&\leq& L\sqrt{\gamma_{i + 1}^{N}}\ \left| \theta_{t_{i}^{N},t_{i + 1}^{N}}^{N}(y)\  - x \right| + \frac{L}{\sqrt{\gamma_{i + 1}^{N}}}\int_{t_{i}^{N}}^{t_{i + 1}^{N}}{\left| \theta_{u,t_{i + 1}^{N}}^{N}(y)\  - \theta_{t_{i}^{N},t_{i + 1}^{N}}^{N}(y)\  \right|du} \\
&& \qquad + C\ a_{N}{ (\gamma_{i + 1}^{N})}^{3/2} \\
& \leq& L\sqrt{\gamma_{i + 1}^{N}}\ \left| \theta_{t_{i}^{N},t_{i + 1}^{N}}^{N}(y)\  - x \right| + \frac{Ca_{N}}{\sqrt{\gamma_{i + 1}^{N}}}\int_{t_{i}^{N}}^{t_{i + 1}^{N}}{(u\  - t_{i}^{N})du\  +}C\ a_{N}{ (\gamma_{i + 1}^{N})}^{3/2} \\
&\leq& L\sqrt{\gamma_{i + 1}^{N}}\ \left| \theta_{t_{i}^{N},t_{i + 1}^{N}}^{N}(y)\  - x \right| + C\ a_{N}{ (\gamma_{i + 1}^{N})}^{3/2}.\end{eqnarray*}
Now write
\begin{eqnarray*} 
\mathcal K_N (t_i^N, t_{i+1}^N, x,y) = A + B,
 \end{eqnarray*}
where
\begin{eqnarray*} 
A &=& \left(\gamma_{i+1}^N\right) ^{-1-d/2}\left (  f^N_{t_i^N,\chi_{N}(x)\sqrt{\gamma_{i}^{N}}} \left ( \frac {\theta^N_{t_{i} ^N, t_{i+1} ^N}(y)-x} {\sqrt{\gamma_{i+1}^N}}+ \Delta_i(x,y)\right  ) -  f^N_{t_i^N,\chi_{N}(x)\sqrt{\gamma_{i}^{N}}} \left ( \frac {\theta^N_{t_{i} ^N, t_{i+1} ^N}(y)-x} {\sqrt{\gamma_{i+1}^N}}\right  ) \right ),\\
B &=& \left(\gamma_{i+1}^N\right) ^{-1-d/2}\left (  f^N_{t_i^N,\chi_{N}(x)\sqrt{\gamma_{i}^{N}}} \left ( \frac {\theta^N_{t_{i} ^N, t_{i+1} ^N}(y)-x} {\sqrt{\gamma_{i+1}^N}}\right  ) -  f^N_{t_i^N,\chi_{N}(\theta ^N_{t_{i} ^N,t_{i+1} ^N}(y))\sqrt{\gamma_{i}^{N}}} \left ( \frac {\theta^N_{t_{i} ^N, t_{i+1} ^N}(y)-x} {\sqrt{\gamma_{i+1}^N}}\right  ) \right ). 
 \end{eqnarray*}
 Now, with similar arguments as for $j-i > 1$ one gets
 \begin{eqnarray*} 
|A| &\leq& {(\gamma_{i + 1}^{N})}^{- 1 - d/2}\left| \sum_{|\nu| = 1}^{}{\mathrm{\Delta}_{i}(x,y)}^{\nu}\int_{0}^{1}D^{\nu}f^N_{t_{i}^{N},x}\left( \frac{\theta_{t_{i}^{N},t_{i + 1}^{N}}^{N}(y)\  - x}{\sqrt{\gamma_{i + 1}^{N}}}\  + \delta\mathrm{\Delta}_{i}(x,y) \right)d\delta \right| \\ &\leq& C\ { (\gamma_{i + 1}^{N})}^{- 1 - d/2}\left( \sqrt{\gamma_{i + 1}^{N}}\ \left| \theta_{t_{i}^{N},t_{i + 1}^{N}}^{N}(y)\  - x \right| + \ a_{N}{ (\gamma_{i + 1}^{N})}^{3/2} \right)\\ && \qquad
\times \int_{0}^{1}{\mathcal{Q}_{M}\left( \frac{\left| \theta_{t_{i}^{N},t_{i + 1}^{N}}^{N}(y)\  - x \right|}{\sqrt{\gamma_{i + 1}^{N}}}\  + \delta\mathrm{\Delta}_{i}(x,y) \right)d\delta}\\
& \leq & C{ (\gamma_{i + 1}^{N})}^{- d/2}\ \left( \ \frac{\left| \theta_{t_{i}^{N},t_{i + 1}^{N}}^{N}(y)\  - x \right|}{\sqrt{\gamma_{i + 1}^{N}}} + \ a_{N}{ (\gamma_{i + 1}^{N})}^{1/2} \right)  \int_{0}^{1}{\mathcal{Q}_{M}\left( \frac{\left| \theta_{t_{i}^{N},t_{i + 1}^{N}}^{N}(y)\  - x \right|}{\sqrt{\gamma_{i + 1}^{N}}}\   (1 - L\delta\sqrt{\gamma_{i + 1}^{N}}\ ) \right)d\delta} \\ &\leq& C{ (\gamma_{i + 1}^{N})}^{- d/2}\left( \ \frac{\left| \theta_{t_{i}^{N},t_{i + 1}^{N}}^{N}(y)\  - x \right|}{\sqrt{\gamma_{i + 1}^{N}}} + \ a_{N}{ (\gamma_{i + 1}^{N})}^{1/2} \right)\mathcal{Q}_{M}\left( \frac{\left| \theta_{t_{i}^{N},t_{i + 1}^{N}}^{N}(y)\  - x \right|}{\sqrt{\gamma_{i + 1}^{N}}}\  \right) 
\\
&\leq& C{ (\gamma_{i + 1}^{N})}^{- d/2}\mathcal{Q}_{M - 1}\left( \frac{\left| \theta_{t_{i}^{N},t_{i + 1}^{N}}^{N}(y)\  - x \right|}{\sqrt{\gamma_{i + 1}^{N}}}\  \right)\\
&: =& C\ \mathcal{Q}_{M - 1}(t_{i + 1}^{N}\  - \ t_{i}^{N},\ \theta_{t_{i}^{N},t_{i + 1}^{N}}^{N}(y)\  - x).
 \end{eqnarray*}
 With our assumption that} $| f^N_{t,x}(z) - f^N_{t,y}(z) | \leq C  |x-y| Q_M(z)$ we get
 for $B$ the following bound:
 \begin{eqnarray*} 
&&|B |\leq  \left(\gamma_{i+1}^N\right) ^{-1-d/2}\left | \theta^N_{t_{i} ^N, t_{i+1} ^N}(y)-x  \right |
Q_M\left ( \frac {\theta^N_{t_{i} ^N, t_{i+1} ^N}(y)-x} {\sqrt{\gamma_{i+1}^N}}\right  )
\\
&& \qquad \leq C \left(\gamma_{i+1}^N\right) ^{-(d+1)/2} Q_{M-1}\left ( \frac {\theta^N_{t_{i} ^N, t_{i+1} ^N}(y)-x} {\sqrt{\gamma_{i+1}^N}}\right  )
\\
&& \qquad \leq C \left(\gamma_{i+1}^N\right) ^{-1/2}\mathcal Q_{M-1}\left (t_{i+1} ^N - t_{i} ^N, \theta^N_{t_{i} ^N, t_{i+1} ^N}(y)-x \right  )
.
 \end{eqnarray*}
Our bounds on $A$ and $B$ imply \eqref{eq:wed1}. Thus we have also that the inequality 
 \eqref{eq:bound MN}  holds for the case $j= i+1$.  \hfill $\square$

\end{document}